\documentclass[11pt]{amsart}

\usepackage{fullpage}
\usepackage{graphicx}
\usepackage{amsmath, amssymb, amsfonts, amsthm}
\usepackage{url}
\usepackage{amsaddr}

\usepackage{tikz} \usetikzlibrary{arrows}

\usepackage{algorithm}
\usepackage{algorithmic}

\usepackage{color}
\definecolor{darkblue}{rgb}{0,0,0.4}

\newtheorem{thm}{Theorem}[section]
\newtheorem{cor}[thm]{Corollary}
\newtheorem{prop}[thm]{Proposition}
\newtheorem{lem}[thm]{Lemma}

\theoremstyle{remark}
\newtheorem{rem}[thm]{Remark}

\newtheorem{assum}[thm]{Assumption}

\newcommand{\ie}{{\it i.e.}}

\newcommand{\ad}{{\mathcal A}}

\newcommand{\wstarto}{\xrightarrow{w*}}

\newcommand{\revision}[1]{\textcolor{black}{\small {\sf #1}}}

\begin{document}
\title{Extremal Spectral Gaps for \\ Periodic Schr\"odinger Operators}

\author{Chiu-Yen Kao}
\address{Department of Mathematical Sciences, 
Claremont McKenna College, 
Claremont, CA}
\email{Chiu-Yen.Kao@claremontmckenna.edu}
\author{Braxton Osting}
\address{Department of Mathematics, University of Utah, Salt Lake City, UT}
\email{osting@math.utah.edu}
\thanks{Chiu-Yen Kao acknowledges partial support from Simons Foundation: Collaboration Grants for Mathematicians 514210 and Braxton Osting acknowledges partial support from NSF DMS 16-19755.}

%\subjclass[2000]{Primary }
%    For articles to be published after 1 January 2010, you may use
%    the following version:


\date{\today}

\begin{abstract} 
The spectrum of a Schr\"odinger operator with periodic potential generally consists of bands and gaps. In this paper, for fixed $m$, we consider the problem of maximizing the gap-to-midgap ratio for the $m$-th spectral gap over the class of potentials which have fixed periodicity and are pointwise bounded above and below.  We prove that the potential maximizing the $m$-th gap-to-midgap ratio exists. In one dimension, we prove that the optimal potential attains the pointwise bounds almost everywhere in the domain and is a  step-function attaining the imposed minimum and maximum values on exactly $m$ intervals. Optimal potentials are computed numerically using a rearrangement algorithm and are observed to be periodic.  In two dimensions, we develop an efficient rearrangement method for this problem based on a semi-definite formulation and apply it to study properties of extremal potentials. We show that, provided a geometric  assumption about the maximizer holds, a lattice of disks maximizes the first gap-to-midgap ratio in the infinite contrast limit. Using an explicit parametrization of two-dimensional Bravais lattices, we also consider how the optimal value varies over all equal-volume lattices. 
\end{abstract}

\maketitle

\section{Introduction}
As described by Floquet-Bloch theory, the spectrum of a self-adjoint, linear differential operator with periodic coefficients consists of spectral bands, and perhaps, spectral gaps. Spectral gaps are significant in a variety of physical applications, where they often describe frequency intervals at which waves cannot propagate.  Examples abound, but spectral gaps are used to control the propagation of electromagnetic waves in a photonic crystal and the energy spectrum of an electron in a solid-state device. In this paper, we study the spectral gaps of a periodic Schr\"odinger operator.

We consider the periodic Schr\"odinger operator, $H_V \colon H^2(\mathbb R^d) \to L^2(\mathbb R^d)$, given by 
$$H_V = - \Delta + V.$$
Here, $V \in L^\infty(\mathbb R^d)$ is a real-valued, $\Gamma$-periodic function for a Bravais lattice, $\Gamma$.  For a general discussion of the spectrum of $H_V$, see, {\it e.g.}, the recent review \cite{kuchment2016}. 
The spectral problem is to find $(E,\psi)$ satisfying 
\begin{subequations}\label{eq:SpectralProblem}
\begin{align} 
& H_V \  \psi = E \  \psi \\ 
& \psi \ \textrm{bounded}. 
\end{align}
\end{subequations}
We say that $\psi(x;k) = e^{i k\cdot x } p(x)$  for a $\Gamma$-periodic function $p \in H^2 (\mathbb R^d)$ is a \emph{Bloch-Floquet solution} with \emph{quasi-momentum} $k \in \mathcal{B}$. Here, $\mathcal B \subset \mathbb R^d$ is the \emph{Brillouin zone}, taken as the Voronoi cell of the origin in the reciprocal lattice, $\Gamma^*$.   
We can decompose $H^2$ into  spaces with different quasi-momenta, $H^2_k$. It is convenient to define the twisted Schr\"odinger operator, 
$$H_V(k) = - (\nabla+i k )\cdot (\nabla+i k ) + V,$$ 
which acts on $\Gamma$-periodic functions. Thus, the spectral problem \eqref{eq:SpectralProblem} can be rewritten as the eigenvalue problem,
\begin{subequations}\label{eq:SpectralProblem2}
\begin{align} 
\label{eq:SpectralProblem2a}
& H_V(k) \ p = E \ p \\ 
& p(x+ X) = p(x) \qquad X \in \Gamma. 
\end{align}
\end{subequations}
Using periodicity, we can restrict $p(x;k)$ to the torus $ \mathbb R^d / \Gamma$. 
The \emph{dispersion relation} (Bloch variety),   is given by 
\begin{align*}
B_H & = \{ (k,E) \in \mathbb R^{d+1} \colon H_V  \textrm{ has a Bloch-Floquet solution $\psi$ with quasi-momentum } k \} \\ 
&= \{ (k,E) \in \mathbb R^{d+1} \colon H_V(k) \textrm{ has a  $\Gamma$-periodic solution } p \} 
\end{align*}
For any $k \in \mathcal B$, the twisted Schr\"odinger operator $H_V(k)$ has a discrete spectrum, so it is convenient to decompose $B_H$ into spectral bands $E_j(k)$.  Eigenvalues, $E_j(k)$,  are $\Gamma^*$-periodic with respect to $k$, so can be considered over the first Brillouin zone,  $\mathcal B$. 

The spectrum of $H_V$ is then given by 
$$
\sigma(H_V) = \bigcup\limits_{k \in \mathcal B} \sigma \left(H_V(k) \right) 
$$
and generally consists of bands and gaps. 
For a given potential $V \in \ad$, denote the left and right edges of the $m$-th gap in the spectrum, $\sigma(H_V)$,  by 
$$\alpha_{m} =  \max_{k \in \mathcal B} \ E_m(k)
\qquad \textrm{and} \qquad  
\beta_{m} = \min_{k \in \mathcal B} \ E_{m+1}(k). 
$$
If the $m$-th gap is non-empty, then $\beta_m > \alpha_m$ but we allow for the possibility that the $m$-th gap is empty.  
For $m\in \mathbb N^{+}$ fixed, we define the \emph{gap-to-midgap ratio}, 
\begin{equation}
\label{eq:G}
G_{m}[V]:= \frac{\beta_{m}- \alpha_{m}}{( \alpha_{m} + \beta_{m})/2 }. 
\end{equation}

For $\Gamma$ a fixed Bravais lattice and $V_+ > 0$, we define the \emph{admissible set} 
\begin{align}
\label{eq:Ad}
 \ad(\Gamma,V_+) :=\{ V \in L^{\infty}(\mathbb R^d)\colon  
& V(x+X)=V(x)   \ \  \textrm{and} \ \  V(x) \in [0,V_+]  \\ 
\nonumber
& \textrm{ for almost all } x \in \mathbb R^d \textrm{ and all } \ X\in \Gamma \}.
\end{align}

We consider the optimization problem of maximizing the gap-to-midgap ratio, $G_m$, over potentials in $\ad = \ad(\Gamma, V_+)$. The following Theorem is immediate. 
\begin{thm} \label{thm:exist}
For fixed $m\in \mathbb N^{+}$, $V_+> 0$, and Bravais lattice $\Gamma$, there exists  $V_{m, \Gamma, V_+}^{\star} \in \ad(\Gamma, V_+)$ such that 
\begin{equation} \label{e:opt} 
G_{m}[V_{m, \Gamma, V_+}^{\star}] = G_{m, \Gamma, V_+}^{\star} := \sup_{V\in \ad(\Gamma, V_+)} G_{m}[V].
\end{equation}
\end{thm}
\begin{proof}[Proof of Theorem \ref{thm:exist}] 
For $\beta \geq \alpha \geq 0$, we have that $0 \leq G_m \leq 2$.  Let $\{ V_{\ell} \}_{\ell=1}^{\infty}$ be a maximizing sequence, \ie,   $G_{m,\Gamma,V_+}^{\star} = \lim_{\ell\uparrow\infty}G_{m}[V_{\ell}] $. 
Since $\ad $ is weak* compact, there exists $V^\star \in \ad$ and a weak* convergent subsequence $V_{\ell}\wstarto V^{\star}$. The mappings $V \mapsto \alpha_{m}[V]$ and $V \mapsto \beta_{m}[V]$ are weak* continuous over $\ad$; see the proof of \cite[Proposition 2.1(ii)]{Dobson1999}. It follows that $V \mapsto G_{m}[V]$ is also weak* continuous over $\ad$. Thus $G_{m}[V^{\star}] = G_{m,\Gamma,V_+}^{\star}$.
\end{proof}

We remark that $V_{m, \Gamma, V_+}^{\star}$ is never unique since $G_m$ is invariant to translations. {\color{black}{We consider the gap-to-midgap ratio of the $m$-th spectral gap, $G_m$, in \eqref{eq:G} rather than just the length of the $m$-th spectral gap because it is a non-dimensional quantity. We also prefer this quantity to fixing $\omega_0$ and maximizing the objective, $\min \{ \beta_m - \omega_0^2, \omega_0^2 - \alpha_m \}$, as in  \cite{Dobson1999,cox2000band}, since (i) this involves the introduction of an additional parameter, $\omega_0$, and (ii) from the optimization viewpoint, introduces additional non-differentiability.}}

\subsection*{Overview} The goal of this work is twofold: (i) develop and study efficient computational methods for finding optimal potentials satisfying \eqref{e:opt} and (ii) study the properties of optimal potentials using both computational and analytical methods.  

In one dimension, we prove that the optimal potential is a  step function attaining the imposed minimum and maximum values on exactly $m$ intervals. Such potentials are sometimes referred to as \emph{bang-bang}. Optimal potentials are computed numerically using a rearrangement algorithm (Algorithm \ref{alg:Rearrange}) and observed to be periodic with period $X/m$. In Proposition \ref{p:1DhighContrast}, we prove that periodic potentials are optimal in the high contrast limit ($V_+ = \infty$). 

In Section \ref{sec:2D}, we change variables in the two-dimensional  periodic problem posed on the torus to obtain a formulation of the problem on a square. 
In Section \ref{s:Opt2d}, we develop an efficient rearrangement method for this problem based on a semi-definite reformulation (Algorithm \ref{alg:SDPrearrange}). We prove in Proposition \ref{prop:Wbb} that the optimal potential has at least one grid point $x$ at which either $V(x) = 0$ or $V(x) = V_+$, a property that we refer to as \emph{weakly bang-bang}. Using the KKT conditions for optimality, we explain in Proposition \ref{prop:Alg2GenAlg1} how this algorithm generalizes Algorithm \ref{alg:Rearrange}, used in one dimension. We use Algorithm \ref{alg:SDPrearrange} to compute optimal potentials with the translational symmetries of the square and triangular lattices for $m=1,2,\ldots,8$; see Figures \ref{f:2DsqA}--\ref{f:2DtriB}. We also study the dependence of the optimal potentials on the parameter $V_+$. We observe from the computational results that the optimal potential as $V_+ \to \infty$ that the region where {\color{black}{$V = 0$}} consists of $m$ disks in the primitive cell; see Figure \ref{f:varyV+}. 
We prove, in Propositions \ref{prop:OptDisc} and Corollary \ref{prop:OptDiscMg2}, the infinite contrast asymptotic result ($V_+ = \infty$), that for $m\geq1$, subject to a geometric  assumption, that the optimal potential has {\color{black}{$\{V=0\}$}} on exactly $m$ equal-size disks. 
Finally, using a parameterization of two-dimensional Bravais lattices, we also consider how $G_{m, \Gamma, V_+}^{\star}$ varies over all equal-volume Bravais lattices, $\Gamma$. 

\subsection*{Related work}  Our one-dimensional results are most similar to \cite{Ashbaugh1992} and \cite{osting2012bragg}. 

In \cite{Ashbaugh1992}, the problem of minimizing the width of the lowest spectral band for the one-dimensional Schr\"odinger operator is studied using methods of proof similar to the present paper. In particular, potentials which maximize the length of the gap between the two lowest Neumann and the gap between the first Neumann and the first Dirichlet eigenvalues are studied. It is shown that such potentials are bang-bang and a necessary condition for the optimal potentials in terms of the associated eigenfunctions is presented. These results are also discussed and put in context of \cite[Ch.8]{Henrot:2006fk}, which is a good general reference for extremal eigenvalue problems, though with less emphasis on extremal properties of the spectrum for periodic operators studied in the present work. 

In \cite{osting2012bragg}, the gap-to-midgap ratio for a one-dimensional periodic Helmholtz operator is studied. It is shown that the Bragg structure (a.k.a. quarter-wave stack) uniquely maximizes the first spectral gap-to-midgap ratio within an admissible class of pointwise-bounded, periodic coefficients. 
This structure also arises asymptotically in the study of long-lived solutions to the wave equation in an infinite domain \cite{BOMWHelm}. 

\bigskip

In two dimensions, the spectrum of Schr\"odinger operators is considerably more complex which causes the study of its extremal properties to be yet more challenging. One of the first studies in this area and arguably the closest  to the present work is  \cite{Dobson1999,cox2000band}. Here, the authors study the spectrum of the TE and TM Helmholtz operators. The objective function to be maximized is $\min \{ \beta_m - \omega_0^2, \omega_0^2 - \alpha_m \}$ with a given $\omega_0$. Optimal potentials are proven to exist within an admissible set and characterized via optimality conditions. In addition, optimal potentials are studied via a numerical method based on the subdifferential of the objective function. The paper focuses on refractive indices with the symmetries of the square lattice. 

In \cite{kao2005}, the authors consider gaps for the two-dimensional Helmholtz operator by using a level set approach to capture the interface between two materials of different dielectrics and shape derivative to deform the interface to find the optimal structure. The optimal solutions computed there reveal additional symmetries, which in part motivates the present study. In later work \cite{he2007incorporating}, both shape derivatives and topological derivatives are incorporated with level set methods in order to flexibly allow changes in the topology so that optimal structures with holes can be easily identified. 

In \cite{sigmund2008geometric}, an exhaustive search on a coarse grid and topology optimization were used to find periodic coefficients in both the TE and TM Helmholtz operators for which the gap-to-midgap ratio is maximized. Based on these numerical results, Sigmund and Hougaard reached the bold conjecture that the globally optimal structure has a particular structure related to a centroidal Voronoi tessellation (CVT). The generators of this CVT correspond to the optimal TM coefficients and the walls of the tessellation correspond to the optimal TE coefficients. 

In recent work \cite{men2010bandgap}, it has been shown that the optimization problem of  maximizing the gap-to-midgap ratio can be reformulated using subspace methods and cast as a sequence of linear semidefinite programs (SDP). In the current work, we follow this approach as well.  Numerical results are given for both the TE and TM Helmholtz operators for a square lattice. These methods have been extended to study spectral gaps of Helmholtz operators in three dimensions, with applications to photonic crystals \cite{Men2014}. 
 
We refer to the numerical methods developed in this work as \emph{rearrangement methods}. Rearrangement methods were introduced by  Schwarz and  Steiner and have wide applications in variational problems \cite{Polya1951,Marshall2011,Bandle1980,kawohl2000,Henrot:2006fk}. They involve a sequence of steps which rearrange the domain or a coefficient in an operator as to provably reduce an objective function. Recently, rearrangement methods have been used to devise computational methods for shape optimization problems, including Krein's problem \cite{Krein:1955ye,Cox1991,Chanillo:2000,Kao2013}, population dynamics \cite{Kao2008, hintermuller2012principal, chugunova2016study}, Dirichlet partitions \cite{Partition2013}, and biharmonic vibration \cite{chen2016minimizing,kang2017minimization}, and have proven  to be extremely efficient in practice.  In one of the examples studied in \cite{Kao2013}, a  method based on rearrangement is able to find an optimal solution in as little as 4 iterations, compared to the 200 iterations (each of equal computational cost) required by a gradient-based, level-set-method evolution  \cite{osher-santosa-jcp01}.

Finally, we mention another connection with the present work. If we consider the spectral problem \eqref{eq:SpectralProblem} with $V\equiv 0$, the spectral gaps close. In \cite{KLO2015}, the authors, together with Rongjie Lai, consider the periodic problem with $k=0$. Denoting the  eigenvalues of the periodic problem by $\lambda_m$, it is shown that among flat tori of volume one, the  $m$-th eigenvalue has a local maximum with value  
\[ \lambda_m = 4\pi^2 \left\lceil \frac{m}{2} \right\rceil^2 \left( \left\lceil \frac{m}{2} \right\rceil^2  - \frac{1}{4}\right)^{-\frac{1}{2}}. \]  

\subsection*{Outline} 
In Section \ref{sec:1D}, we study the one-dimensional problem. 
In Section \ref{sec:2D}, we present some background material needed for the study of spectral gaps for the two-dimensional problem. 
In Section \ref{s:Opt2d}, we describe the SDP reformulation of the problem and present a rearrangement algorithm based on this formulation. The results from several computational experiments are presented. 
We conclude in Section \ref{sec:disc} with a discussion.

%%%%%%%%%%%%%%%%%%%%%%%%%%%%%%%%%%%%%%%%%%%%%%%%%%

\section{One-dimensional case}\label{sec:1D} 
In this section, we consider \eqref{eq:SpectralProblem2} in one-dimension, which is sometimes also referred to as Hill's equation. 
%\begin{subequations}\label{eq:Hill}
%\begin{align} 
%\label{eq:SpectralProblem2a}
%& - (\partial_x + i k)^2  p + V  p = E  p  \qquad  \qquad x \in \Omega := [0,X] \\ 
%& p(x+ X) = p(x). 
%\end{align}
%\end{subequations}
We assume that the potential, $V$, is assumed to be admissible, as in \eqref{eq:Ad}. The one-dimensional case is considerably simpler since the edges of a nonempty spectral gap are characterized by either anti-periodic ($m$ odd) or periodic ($m$ even) eigenproblems, for which the eigenvalues are simple. In fact, for even gaps with $k=0$, this problem reduces to maximizing the $m$-th gap between eigenvalues for a Schr\"odinger operator on $S^1$ where the potential is point-wise bounded. The results proven here are analogous to the results proven in   \cite{osting2012bragg} for the Helmholtz operator. 

Recall the definition of $G_m=\frac{\beta_{m}- \alpha_{m}}{( \alpha_{m} + \beta_{m})/2 }$ from \eqref{eq:G}. The following Lemmas give the variation of $G_m$ with respect to the potential.

\begin{lem} \label{prop:sens}
Let $(p,E)$ be a simple eigenpair satisfying \eqref{eq:SpectralProblem2}, for a potential $V_0\in\ad$, normalized such that $\int_{\mathbb R^d / \Gamma}  |p(x)|^{2}  \ dx = 1$. The Fr\'echet derivative of $E(V)$  at $V=V_0$ is 
\begin{align*}
\delta E =  \int_{\mathbb R^d / \Gamma}  |p(x)|^{2}  \delta V(x) \ dx    \quad \implies \quad \frac{\delta E}{\delta V} =  |p|^{2}. 
\end{align*}
\end{lem}

\begin{lem} \label{lem:derivJ} Let $d=1$ and fix $m\in \mathbb N^{+}$. Let $(\alpha,\psi_\alpha)$ and $(\beta,\psi_\beta)$ denote eigenpairs satisfying \eqref{eq:SpectralProblem} corresponding the left and right edges of the $m$-th gap in the spectrum.  
If $\beta > \alpha$, then the  variation of $G_m$ with respect to $V$ is given by 
\begin{align*}
\frac{\delta G_m}{\delta V} = \frac{\alpha \beta }{ (\alpha + \beta)^2/4} \left(  \psi_\beta^2 / \beta - \psi_\alpha^2 / \alpha \right).
\end{align*}
\end{lem}
\begin{proof}
If $\beta>\alpha$, then $\alpha$ and $\beta$ are simple eigenvalues. The proof then follows from Lemma \ref{prop:sens} and the fact that $\psi_\alpha$ and $\psi_\beta$ are real. 
\end{proof}

\begin{thm} \label{thm:BangBang} Let $d=1$ and fix $m\in \mathbb N^{+}$. 
The maximizer of $G_{m}[V]$ over $\ad$ is piecewise constant and  achieves the prescribed point-wise bounds, $0$ and $V_+$, almost everywhere, \ie, $V^{\star}_{m}$ is a bang-bang control. 
Furthermore, any local maximizer $\tilde V \in \ad$ with corresponding eigenpairs $(\alpha, \psi_{\alpha})$ and $(\beta, \psi_\beta)$ with nonzero gap ({\i.e.} $\alpha \neq \beta$) satisfies 
\begin{equation}
\label{eq:Vstar}
\tilde{V}(x) = \begin{cases} 
V_{+}  &  x \in \Omega_{+}:= \{ x \colon  \psi_\alpha^2 (x) / \alpha <  \psi_\beta^2 (x)  / \beta  \} \\
0   &  x \in \Omega_{-} := \{ x \colon   \psi_\alpha^2 (x) / \alpha >   \psi_\beta^2 (x)  / \beta \}.
\end{cases}
\end{equation}
\end{thm}
\begin{proof}
Let $\tilde{V} \in \ad$ be any local maximizer. 
Consider the set $A = \{x\in [0,X]\colon  0 <\tilde{V}(x) < V_{+} \} $ and let $S\subset A$ be arbitrary. 
For $\delta V(x)= 1_{S}(x)$, the indicator function on $S$, by Lemma \ref{lem:derivJ}, local optimality of $\tilde{V}$ requires 
\begin{equation}
\label{eq:opt}
\langle \frac{\delta G}{\delta V} [\tilde{V}], 1_{S} \rangle = 0  
\qquad \iff \qquad   
\alpha \psi_\beta ^2 = \beta  \psi_\alpha^2 \quad \text{ a.e.  on  } A,
\end{equation}
where $(\alpha, \psi_{\alpha})$ and $(\beta, \psi_\beta)$ are eigenpairs for $\tilde{V}$. We consider an interval where both $\psi_\alpha$ and $\psi_\beta \neq 0$. Multiplying $\psi_\alpha$ by $-1$ if necessary, we have that 
$$
\sqrt{\alpha} \psi_\beta =  \sqrt{\beta} \psi_\alpha. 
$$
Applying $H_V$ to both sides, we obtain 
$$
\sqrt{\alpha} \beta \psi_\beta =  \sqrt{\beta} \alpha \psi_\alpha
\qquad \implies \qquad 
\beta = \alpha.
$$
But this contradicts the assumption that the gap is nonempty. Thus, $A$ has zero measure, \ie, $\tilde{V}(x) \in \{0,V_+ \}$ for a.e.  $x\in[0,X]$. 

We now consider a set $\Omega_- = \{x\in [0,X]\colon  \tilde{V}(x) \equiv 0\}$ and let $S\subset \Omega_-$ be arbitrary. The perturbation $\delta V(x) = 1_S(x)$ is admissible. Local optimality requires that
$$
\langle \frac{\delta G}{\delta V} [\tilde{V}], 1_{S} \rangle \leq 0  
\qquad \iff \qquad   
\alpha \psi_{\beta}^2 \leq  \beta  \psi_{\alpha}^2   \quad \text{ a.e.  on  } \Omega_-,
$$
as desired. 

A similar perturbation argument for the set $\Omega_+ = \{x\in [0,X]\colon  \tilde{V}(x) \equiv V_+\}$ completes the proof. 
\end{proof}

\begin{thm} \label{thm:Vstep} Let $d=1$, fix $m \in \mathbb N^{+}$, and let $\tilde V(x)  \in \ad$ be a local maximizer of $G_m$ with $G_m (\tilde V) > 0$. Then there are only a finite number of transitions between where $\tilde V$ is $0$ and $V_+$ and therefore $\tilde V$ is a step function. 
\end{thm}
\begin{proof} Suppose there are an infinite number of transition points $\{x_j\}$. Then there exists an accumulation  point, say $x_\star \in [0, X)$, such that, along a subsequence which we again denote by $\{x_j\}$,  $x_j \to x_\star$. 
By \eqref{eq:Vstar}, at each $x_j$, we have $\beta \psi_\alpha^2(x_j) = \alpha \psi_\beta^2(x_j)$. Taking $\psi_\alpha (x_\star) \geq 0$ and $\psi_\beta (x_\star) \geq 0$, we  can pass to a further subsequence so that $\sqrt{\beta} \psi_\alpha(x_j) = \sqrt{\alpha} \psi_\beta (x_j)$. Taking the limit as $x_j \to x_\star$, we obtain 
$$
\sqrt{\beta} \psi_\alpha(x_\star) = \sqrt{\alpha} \psi_\beta (x_\star). 
$$
We also have that 
$$
0 = \lim_{j \to \infty} \frac{ \left( \sqrt{\beta} \psi_\alpha(x_j) - \sqrt{\alpha} \psi_\beta (x_j) \right) - \left( \sqrt{\beta} \psi_\alpha(x_\star) - \sqrt{\alpha} \psi_\beta (x_\star) \right) }{x_j - x_\star} = 
\sqrt{\beta} \psi_\alpha ' (x_\star) - \sqrt{\alpha} \psi_\beta ' (x_\star),
$$
where the prime denotes a spatial derivative. 
Define
$$
\tilde \psi_\beta = \sqrt{ \frac{\alpha}{\beta} } \psi_\beta 
\qquad \textrm{ so that } \qquad 
\tilde \psi_\beta(x_\star) = \psi_\alpha(x_\star) 
\quad \textrm{ and } \quad 
\tilde \psi_\beta'(x_\star) = \psi_\alpha'(x_\star). 
$$
By assumption, $G_m (\tilde V) > 0$ which implies that $\beta > \alpha$. 
It follows that $(\beta, \tilde \psi_\beta)$ and $(\alpha, \psi_\alpha)$ are periodic or semi-periodic eigenpairs satisfying \eqref{eq:SpectralProblem} for different values of $E$, but have the same Cauchy data at $x = x_\star$. We show that this is a contradiction. We recall that the Sturm Oscillation Theorem implies that $\psi_\alpha$ and $\psi_\beta$ take the same number of zeros on any interval of length $X$ \cite[Theorem 3.1.2]{Eastham1973}. 

Without loss of generality, we may assume that $\psi_\alpha(x_\star) = \psi_\beta(x_\star) > 0$. If $\psi_\alpha(x_\star) = \psi_\beta(x_\star) = 0$, then since $\beta > \alpha$, the Sturm Oscillation Theorem would imply that $\psi_\beta$ takes at least one more zero on $[x_\star, x_\star + X)$ than $\psi_\alpha$, but this is a contradiction. 

Let $a,b$ be successive zeros of $\psi_\alpha$ with $x_\star \in (a,b)$. \emph{Claim:}  The solution $\tilde \psi_\beta$ takes two zeros in $(a,b)$: one in $(a,x_\star)$ and another in $(x_\star, b)$. But this completes the proof since $\psi_\beta$ must also take a zero between any other consecutive zeros of $\psi_\alpha$, contradicting the fact that they take the same number of zeros on any interval of length $X$. 

To prove the claim, consider the \emph{Wronskian},  $W(x) = \tilde \psi_\beta (x) \psi_\alpha'(x) - \psi_\alpha(x) \tilde \psi_\beta' (x)$.  Using \eqref{eq:SpectralProblem}, we compute  
$$
W(x) = (\beta- \alpha) \int_{x_\star}^x \psi_\alpha(y) \tilde\psi_\beta(y)  dy. 
$$

Suppose $\tilde \psi_\beta > 0$ on $(a, x^\star)$. Then on one hand $W(a) = \tilde \psi_\beta(a) \psi_\alpha'(a) > 0$ and on the other $W(a) = - (\beta-\alpha) \int_a^{x_\star} \psi_\alpha(y) \tilde \psi_\beta(y) dy < 0$ which is a contradiction. 

Similarly, suppose $\tilde \psi_\beta > 0$ on $(x^\star,b)$. Then on one hand $W(b) = \tilde \psi_\beta(b) \psi_\alpha'(b) < 0$ and on the other $W(b) = (\beta-\alpha) \int_{x_\star}^b \psi_\alpha(y) \tilde \psi_\beta(y) dy > 0$ which is a contradiction. 
\end{proof}

\subsection{Reduction of \eqref{e:opt} to the Kronig-Penney model for $m=1$.} 
The  optimality result in \eqref{eq:opt} means that the potential is bang-bang, {\it i.e.}, it attains the imposed  pointwise bounds almost everywhere. 

\begin{thm} \label{thm:OneInterval} For $m=1$, every locally optimal  potential of \eqref{e:opt}  with $\beta > \alpha$ can be translated to take the simple form  
$$
V_b(x) = \begin{cases} 
V_{+}  &  x \in [0,b]\\
0   &  x \in [b,X] .
\end{cases}
$$
where $b$ is a positive real number.
\end{thm}
\begin{proof} 
We assume that $V^\star$ is a locally optimal potential for $m=1$ with more than two (but by Theorem \ref{thm:Vstep} a finite number) of transition points. Let $(\alpha, \psi_\alpha)$ and $(\beta, \psi_\beta)$ be the  eigenpairs corresponding to the spectral band edges of the first gap. Recall that $\psi_\alpha$ and $\psi_\beta$ vanish at exactly one point each, say $x_\alpha$ and $x_\beta$, with $x_\alpha \neq x_\beta$. By translating $x$ if necessary, we may assume $x_\beta > x_\alpha$. By changing signs if necessary, we may assume that $\psi_\alpha> 0$ and $\psi_\beta > 0$ on $(x_\alpha, x_\beta)$. 

We  consider the Wronskian, $W(x) =  \psi_\beta (x) \psi_\alpha'(x) - \psi_\alpha(x)  \psi_\beta' (x)$. Clearly $W(x_\alpha) > 0$ and $W'(x) = (\beta - \alpha) \psi_\alpha \psi_\beta > 0$ on $(x_\alpha, x_\beta)$. Thus, on $(x_\alpha, x_\beta)$ 
\begin{equation} 
\label{e:IncArg}
W>0 
\qquad \implies \qquad
\frac{\psi_\alpha' }{ \psi_\alpha } > \frac{\psi_\beta'}{ \psi_\beta} 
\qquad \implies \qquad
\frac{d}{dx} \log \left( \frac{\psi_\alpha(x)}{ \psi_\beta(x) }\right) > 0. 
\end{equation}
It follows that $\log \left( \frac{\psi_\alpha(x)}{ \psi_\beta(x) }\right)$ is strictly increasing on $(x_\alpha, x_\beta)$. 

We now suppose that there are more than one transition points of $V(x)$ in the interval $(x_\alpha, x_\beta)$. Let $y$ and $z$ be two such distinct points. By the optimality condition \eqref{eq:Vstar}, we have that 
$\psi_\alpha(x)/ \sqrt{\alpha} = \psi_\beta(x)/ \sqrt \beta$
at both $x= y$ and $x=z$. But this implies that 
$$
\log \left( \frac{\psi_\alpha(y)}{ \psi_\beta(y) }\right) = \log \left( \frac{\psi_\alpha(z)}{ \psi_\beta(z) }\right)
= \log \left( \frac{\sqrt \alpha}{\sqrt \beta} \right),
$$
which contradicts the fact that  $\log \left( \frac{\psi_\alpha(x)}{ \psi_\beta(x) }\right)$ is strictly increasing on $(x_\alpha, x_\beta)$. Thus, there can be only transition point in  $(x_\alpha, x_\beta)$. A similar argument shows that there can only be one transition point in $[0,X] \setminus  [x_\alpha, x_\beta]$.
\end{proof} 

Theorem \ref{thm:OneInterval} shows that the optimal potential is given by the Kronig-Penney model, which has been well-studied in solid-state physics \cite{kronig1931quantum}.  In this case, \eqref{e:opt} reduces to a one-dimensional optimization problem---find the value of $b \in [0,X]$ so that $G_1$ is maximized.

\begin{rem} \label{rem:OptSolNeuDir}
Since the interval can be translated so that the optimal potential is symmetric, it follows that the semi-periodic eigenfunctions are either symmetric or antisymmetric. It follows from the proof of Theorem \ref{thm:OneInterval}  that for the optimal potential for $m=1$,  $\psi_\beta'$ and $\psi_\alpha$ simultaneously vanish and visa-versa.
\end{rem}

\begin{rem} For the analogous Helmholtz problem, the maximal first spectral gap-to-midgap ratio  is obtained by the Bragg structure \cite{osting2012bragg}. For the Schr\"odinger operator, it isn't obvious if the optimal potential can be written explicitly. 
\end{rem}

\subsubsection{Numerical Computation} 
In the following, we develop some notation so that we can compute the solution to \eqref{eq:SpectralProblem2} in one dimension and find optimal potentials. 
Fix $X,V^+,b,k$. 
Denote 
\[
Q = \sqrt{V^{+}-E}, \qquad K = \sqrt{E}, \quad \textrm{and} \quad a=X-b. 
\]
Continuity of $\psi(x;k)$ and $\psi'(x;k)$ at $x=0$ and $x=b$ requires that $Q$ and $K$ satisfy 
\[
\frac{Q^{2}-K^{2}}{2QK}\sinh(Qb)\sin(Ka)+\cosh(Qb)\cos(Ka)=\cos(Xk).
\]
For $k=\pi/X$ ($\psi$ is an anti-periodic solution), this yields the  two equations
\begin{subequations} \label{e:findRoots}
\begin{align}
& \frac{Q^{2}-K^{2}}{2QK}\sinh(Qb)\sin(Ka)+\cosh(Qb)\cos(Ka)=-1 \\
& Q^2+ K^2 = V_+. 
\end{align}
\end{subequations}
The solutions $Q(E),K(E)$ of these equations determine the eigenvalues $E$ that correspond to the odd spectral gap edges. 

\begin{figure}[t]
\begin{center}
\includegraphics[width=.48\textwidth]{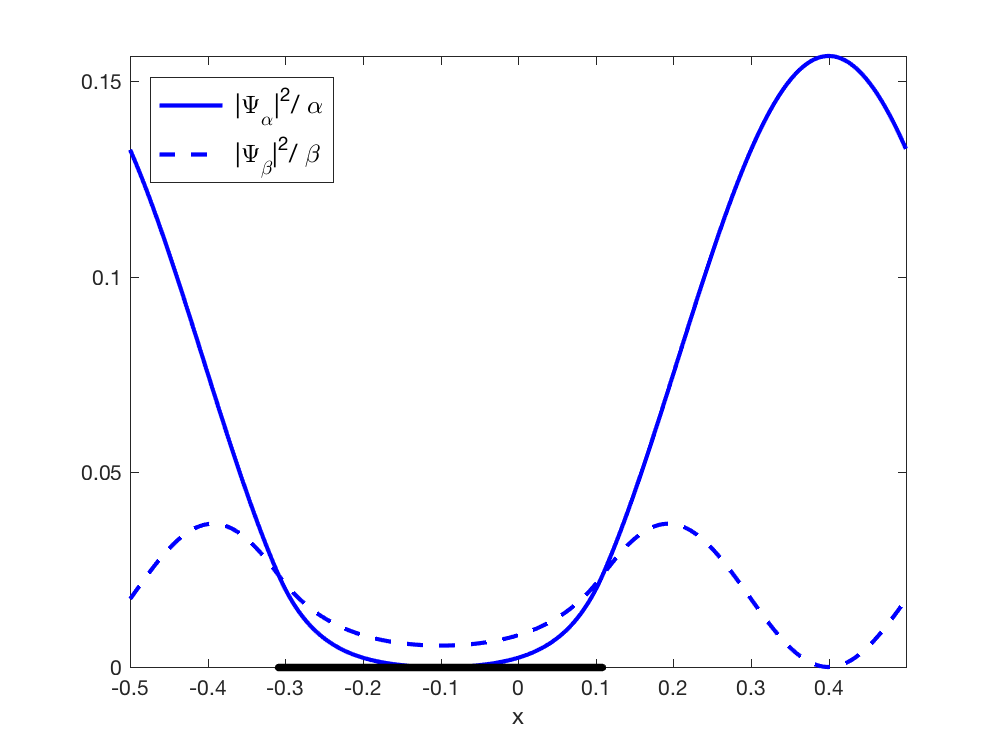}
\includegraphics[width=.48\textwidth]{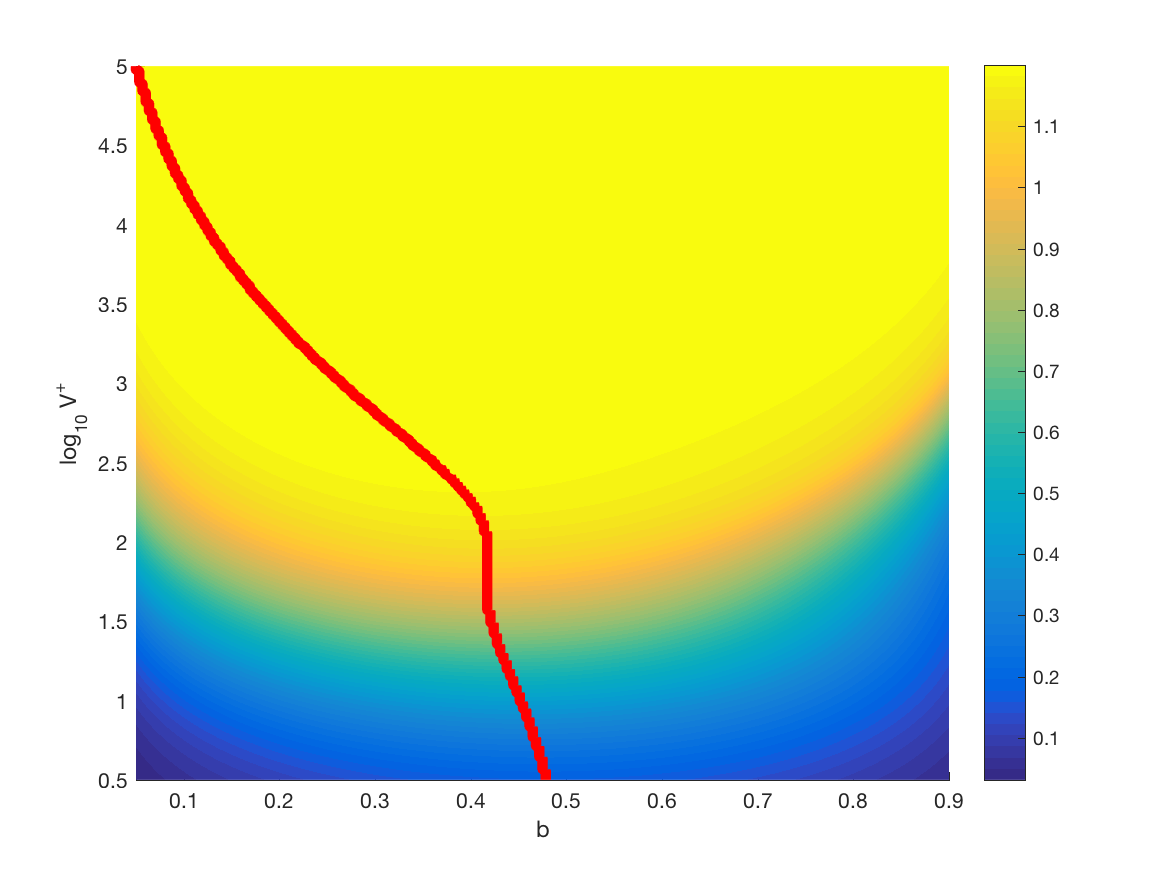}
\caption{{\bf (Left)} An illustration of the optimality condition in \eqref{eq:Vstar} for $X=1$, $V_+ = 100$, and $m=1$. The set $\{ x \colon  V(x) = V_+ \} = \{ x\colon |\psi_\alpha (x) |^2 / \alpha <  | \psi_\beta (x) |^2 / \beta \}$ is indicated on the $x$-axis by a thick black line. {\bf (Right)} Take $X=1$.  For different values of $b$ (x-axis) and $V_+$ (y-axis), we plot the contours of  $G$. For each value of $V_+$, the value of $b$ that maximizes $G_1$ is indicated by the red line. }
\label{f:bvsGJ}
\end{center}
\end{figure}

\begin{figure}[t]
\begin{center}
\includegraphics[width=.48\textwidth]{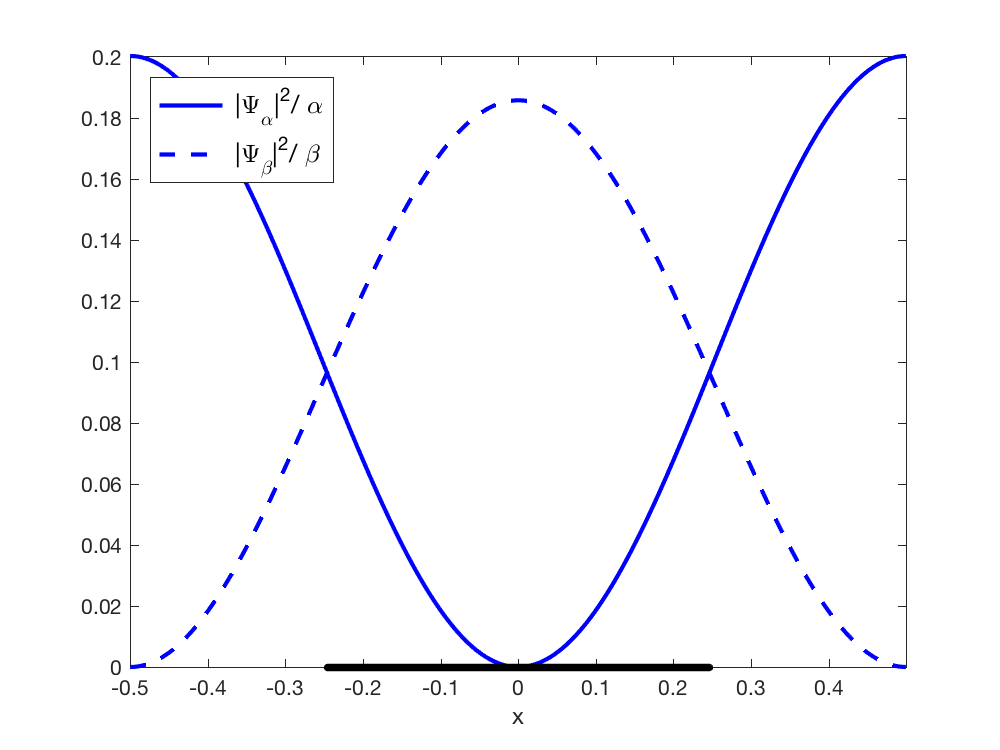}
\includegraphics[width=.48\textwidth]{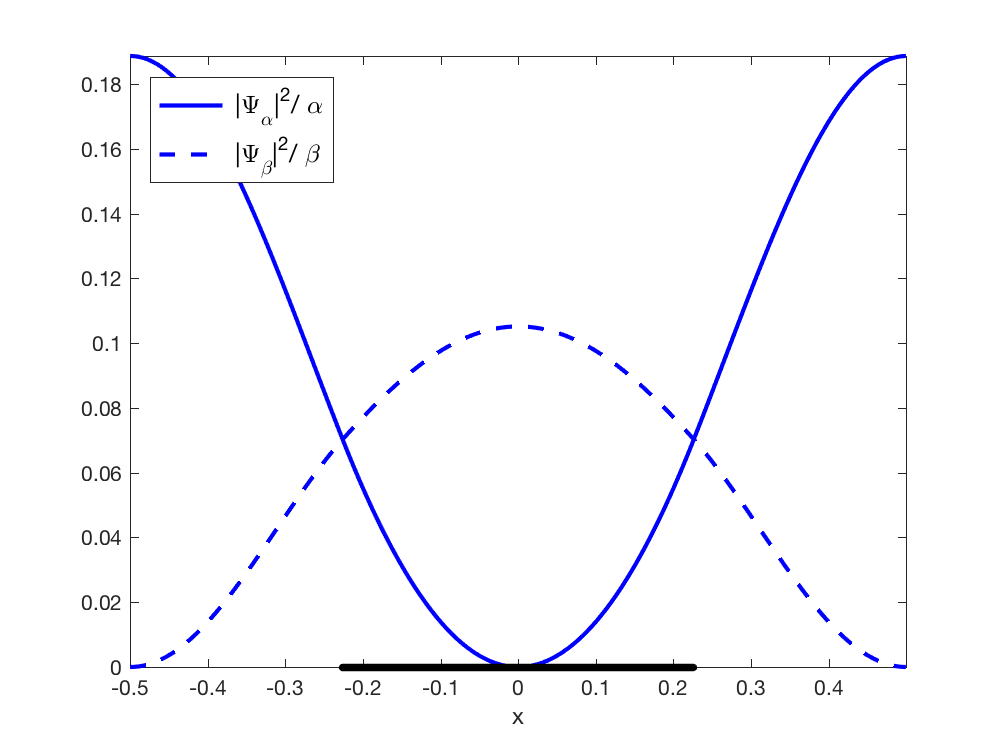}
\includegraphics[width=.48\textwidth]{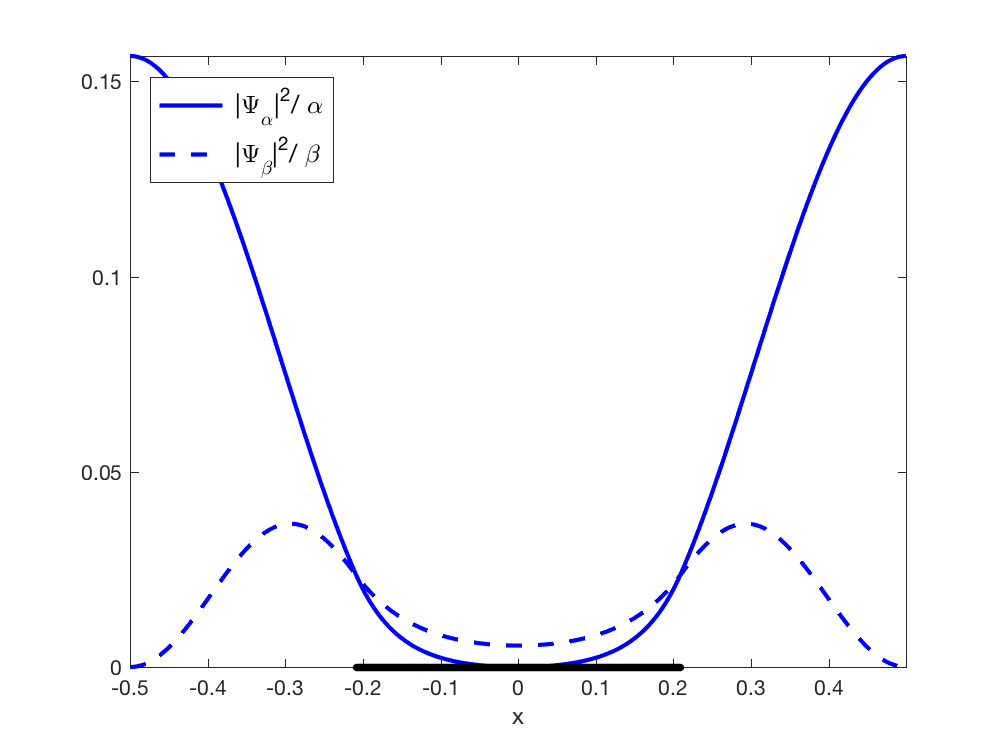}
\includegraphics[width=.48\textwidth]{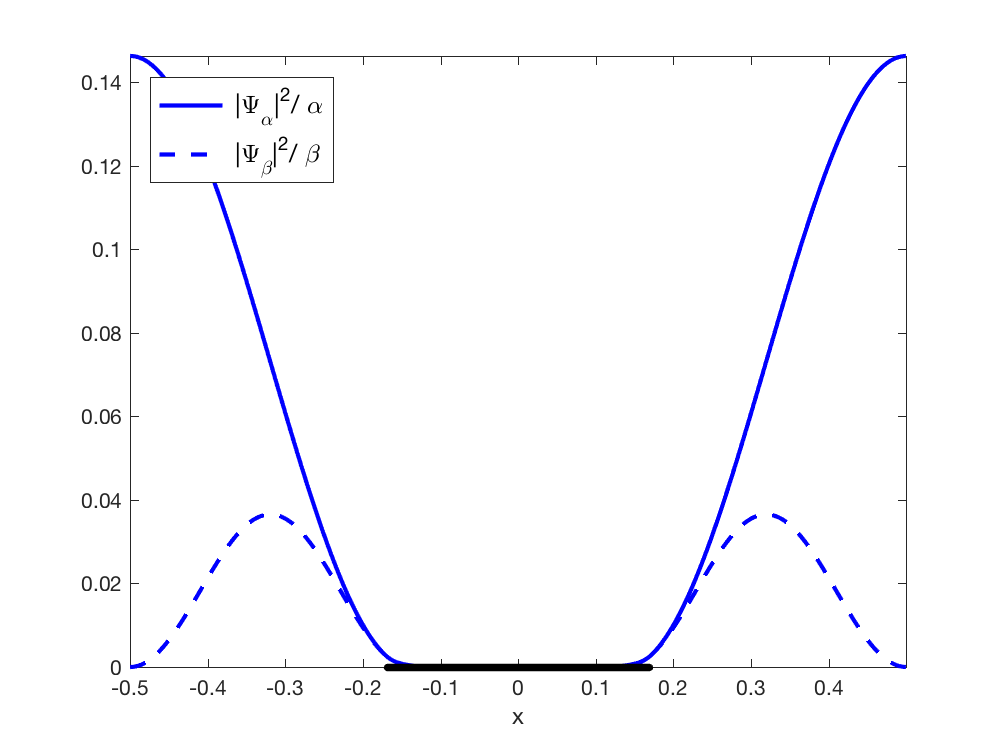}
\caption{Let $X=1$ and $m=1$. For  
$V_+ = 1$ (top left),  
$V_+ = 10$ (top right), 
$V_+ = 100$ (bottom left), and
$V_+ = 1000$ (bottom right), 
we plot the eigenfunctions corresponding to the optimal potential. The shape of the eigenfunctions change with respect to $V_+$; see text for a discussion. In particular, the eigenfunctions converge to the eigenfunctions of the Dirichlet-Laplace operator on the set $\{V=0\}$. }
\label{f:m1EigFunPlot}
\end{center}
\end{figure}

Thus, the objective can be  evaluated by  solving either \eqref{eq:SpectralProblem}, \eqref{eq:SpectralProblem2},  or \eqref{e:findRoots}. In Figure \ref{f:bvsGJ}(left), for fixed $V_+$, we illustrate that the  eigenfunctions corresponding to the optimal potential satisfy the optimality conditions \eqref{eq:Vstar}.  In Figure \ref{f:bvsGJ}(right) we plot the optimal value of $b$ for different values of $\log_{10}V_+$. From the plot we additionally observe that the value of $b$ which maximizes $G_1$ is unique. 

In Figure \ref{f:m1EigFunPlot}, we study how the eigenfunctions change as $V_+$ is varied. It is known that as $V_+\to \infty$, the eigenfunctions vanish on the set $\{V=V_+\}$ (see Proposition \ref{prop:ConvVinfty}).  In particular, for small $V_+$, say $V_+=1$ as in the top left panel, the second eigenfunction ($\psi_\beta$) takes large values in the region $\{V=V_+\}$. However, as $V_+$ is increased, the second eigenfunction takes smaller values on this region; the eigenfunction transitions from having a single maximum to having two. As $V_+ \to \infty$, the eigenfunctions converge to the Dirichlet-Laplace eigenfunctions for the set $\{V=0\}$. 

\subsubsection{Asymptotics for $m=1$} 
Here, we consider the optimal value of $b$ as  $V_+ \to 0$ and $V_+ \to \infty$. 

\begin{lem} \label{l:V+to0} 
Using the notation of Theorem \ref{thm:OneInterval}, as $V_+ \to 0$, the optimal value of $b$ is $X/2$. 
\end{lem}
\begin{proof} 
We apply the perturbation formula in Lemma \ref{lem:derivJ}. For $V_+=0$ the anti-periodic eigenfunctions are $\sin(\pi x/ X)$ and $\cos(\pi x / X)$ which both correspond to the spectral value $E = \pi^2/X^2$. The largest perturbation will occur if we set $\Omega_+ = \{ x \in [0,X]\colon | \cos(\pi x / X) | >  | \sin( \pi x / X) | \}$. Using periodicity, this corresponds to taking $b=X/2$.  
\end{proof}

\begin{lem} \label{l:V+toinfty}
As $V_+ \to \infty$, the value of $G_1$ for any $b$ and any $X$  is $\frac{6}{5}$. 
\end{lem}
\begin{proof} 
As $V_+ \to \infty$, the potential barrier forces the eigenfunction to be zero on $\Omega_+$. In this case, we get a Dirichlet-Laplace eigenvalue equation with eigenvalues $\left( \frac{n\pi}{ X-b} \right)^2$, so the value of $G_1$ for any $b$ is given by $2\frac{2^2 - 1^2}{2^2 + 1^2} = \frac{6}{5}$. 
\end{proof}

The results in Lemmata \ref{l:V+to0} and \ref{l:V+toinfty} are observed in Figure \ref{f:bvsGJ}(right). From this plot, we also observe that the optimal value of $b$ tends to $0$ as $V_+ \to \infty$. 

\subsubsection{Rearrangement algorithm}
The idea for the rearrangement algorithm is to use the optimality criterion  \eqref{eq:Vstar} to define a sequence of potentials; see Algorithm \ref{alg:Rearrange}. In the first step, for fixed $V$, we compute the eigensolutions corresponding to the edges of the $m$-th spectral gap. In the second step, we redefine the potential via \eqref{eq:Vstar}. These steps are repeated until a potential satisfying the necessary conditions for optimality \eqref{eq:Vstar} is identified. 

In Figure \ref{f:OptCond}, we plot  iterations of Algorithm \ref{alg:Rearrange} for the first  gap ($m=1$) with $X=1$ and $V_+ = 100$. We observe that the algorithm converges in 10 iterations for the initial condition with $b = |\Omega_+|/|\Omega| = 0.8$. The optimal configuration has $|\Omega_+|/|\Omega| = 0.42$, as can also be seen in Figure \ref{f:bvsGJ}(left).

\begin{rem} \label{rem:MonDec}
We observe that the  value of $G_m$ is strictly increasing for non-stationary iterations of the rearrangement algorithm (Algorithm \ref{alg:Rearrange}). 
\end{rem}  

\begin{algorithm}[t!]
\caption{\label{alg:Rearrange} 
The rearrangement algorithm for the one-dimensional problem in \eqref{e:opt}. }
\vspace{.2cm}

\begin{algorithmic}[t]
\STATE{\bfseries Input:} Fix $V_+ > 0$, $m \in \mathbb N^{+}$. Initialize $V$ in $\mathcal A(V_+)$ defined in \eqref{eq:Ad}. 

\vspace{.2cm}

\WHILE {the potential is not stationary}
\STATE 1. Compute eigensolutions $(\alpha, \psi_\alpha)$ and $(\beta, \psi_\beta)$ satisfying \eqref{eq:SpectralProblem} corresponding the edges of the $m$-th spectral gap. 

\bigskip

\STATE 2. Rearrange the potential by defining  
$$
V(x) = \begin{cases} 
V_{+}  &  x \in  \{ x \colon  \psi_\alpha^2 (x)/ \alpha < \psi_\beta^2 (x) / \beta  \} \\
0   &  x \in  \textrm{otherwise}.
\end{cases}
$$
\ENDWHILE
\end{algorithmic}
\end{algorithm}

\begin{figure}[t!]
\begin{center}
Iteration 0 \\ 
\includegraphics[width=.88\textwidth]{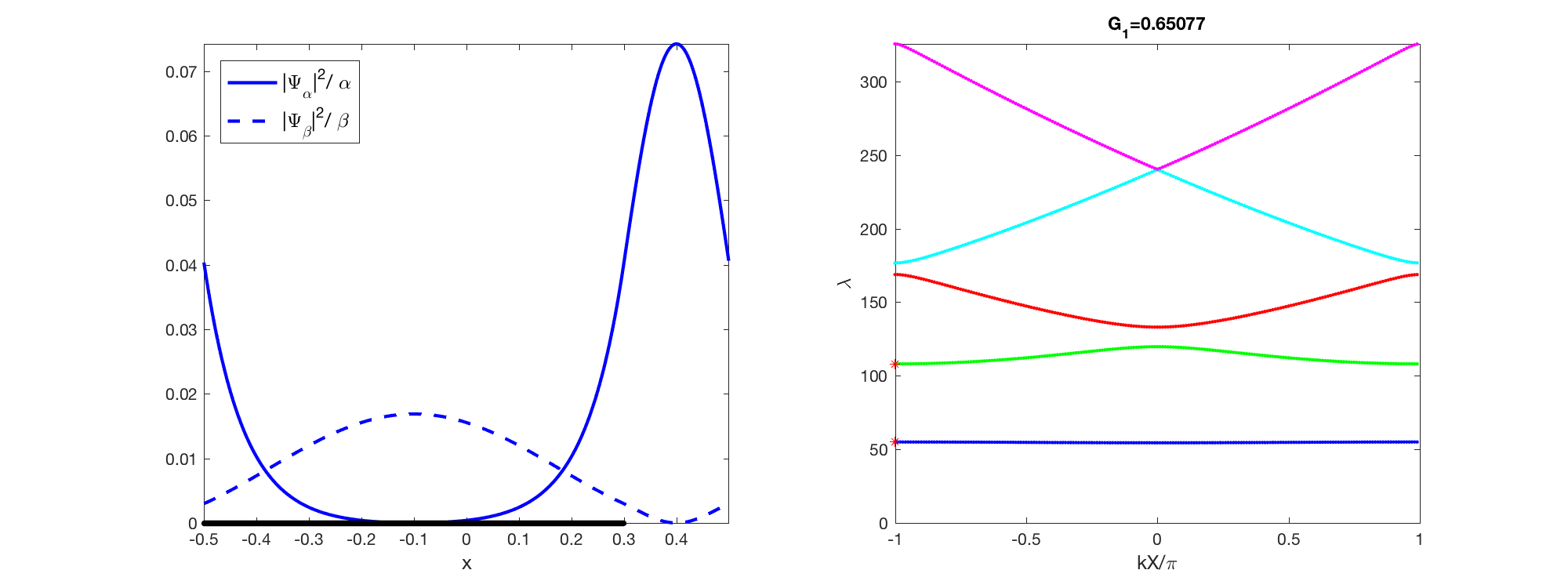} \\ \ \\
Iteration 1 \\
\includegraphics[width=.88\textwidth]{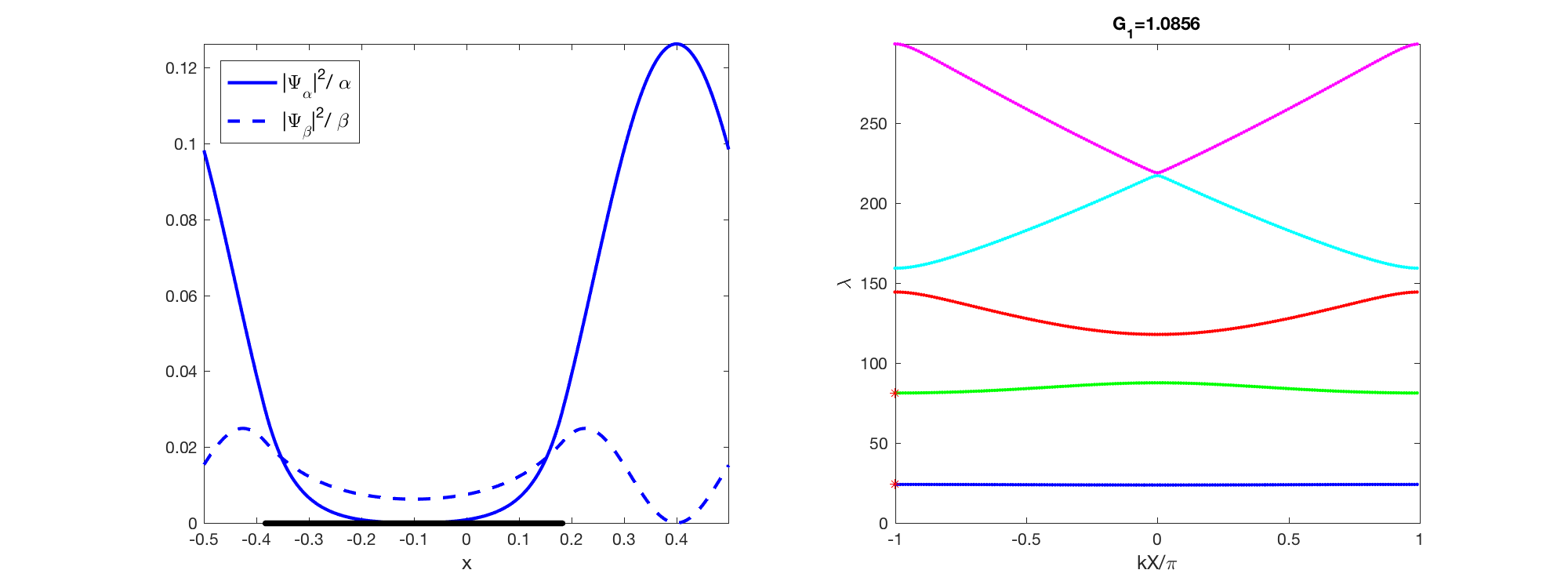} \\ \ \\
Iteration 10 \\
\includegraphics[width=.88\textwidth]{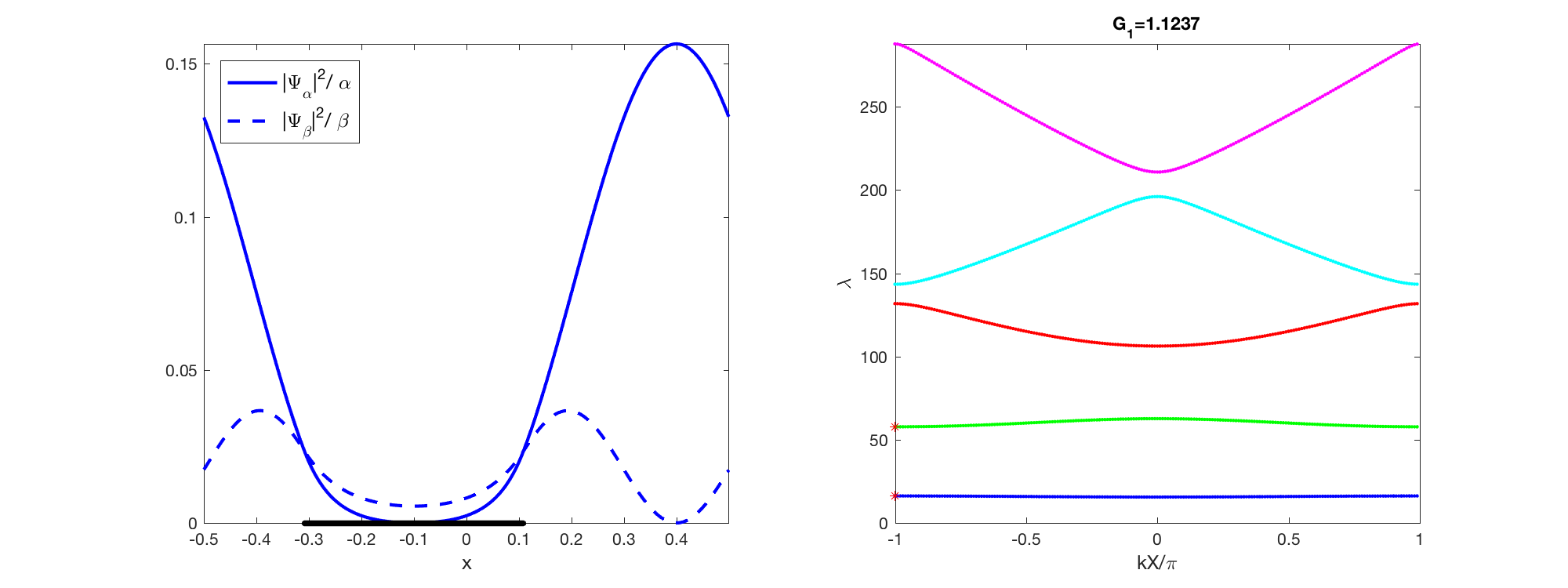}
\end{center}
\caption{An illustration of iterations 0, 1, and 10 of the rearrangement method in one dimension for $m=1$, $X=1$, and $V_+ = 100$.  
{\bf (Left)} The eigenfunctions corresponding to the spectral gap edges are plotted together with the set $\{ x \colon V(x) = V_+ \}$  indicated on the $x$-axis by a thick black line. 
{\bf (Right)} The dispersion relation for the Schr\"odinger operator. 
}
\label{f:OptCond}
\end{figure}

\subsection{Optimal potentials of \eqref{e:opt} for  $m\geq2$.} 
By arguing as in the proof of Theorem \ref{thm:OneInterval}, one may prove the following corollary. 
\begin{cor} \label{thm:mTransPoints}
Fix $m\in \mathbb N^{+}$. Every locally optimal potential of \eqref{e:opt}  with $\beta > \alpha$ is a step function with exactly $2m$ transition points. In other words, there are $m$ intervals where $V= V_+$ and $m$ intervals where $V = 0$. 
\end{cor}

The next result gives an upper bound on $G_m[V]$ for any $ V \in  \mathcal A(X,V_+)$. 
\begin{prop} \label{prop:minfty}
Let $V \in  \mathcal A(X,V_+)$. Then 
\[
\revision{G_m[V] \leq}  \frac{ 2X^2V_+}{ 2 \pi^2 m^2 + X^2V_+}. 
\]
\end{prop}
\begin{proof}
For $V \in  \mathcal A(X,V_+)$,  we have the semidefinite ordering
\[
- (\partial_x + i k)^2 \preceq - (\partial_x + i k)^2 + V \preceq - (\partial_x + i k)^2 + V_+. 
\]
which  implies that 
\[
\mu_j(k) \leq E_j(k) \leq \mu_j(k) + V_+, 
\]
where $\mu_j(k)$ are the eigenvalues of $H_0(k)$ with periodic boundary conditions. 

The $m$-th gap occurs at $k=0$ for $m$ even and $k=\pi$ for $m$ odd. 
Recall that $\mu_{1}(0) = 0$ and  $\mu_{2j}(0) = \mu_{2j+1}(0) =  \left( \frac{2 j \pi}{X}\right)^2$. 
For an even $m=2j$ and $k=0$, we now write
$$
 G_m 
 = 2 \frac{E_{m+1}(0) - E_m(0)}{E_{m+1}(0) + E_m(0) } 
 \leq 2 \frac{\mu_{2j+1}(0)+V_{+} - \mu_{2j}(0)}{\mu_{2j+1}(0)+V_{+}+ \mu_{2j}(0) } 
 =  \frac{ 2X^2V_+}{ 2 \pi^2 m^2 + X^2V_+}. 
 $$
 Here we have used the fact that $f(\alpha,\beta) = 2\frac{\alpha - \beta}{\alpha +\beta}$ is increasing in $\alpha$ and decreasing in $\beta$ for $\alpha, \beta>0$.  
Recall that $\mu_{2j-1}(\pi) = \mu_{2j}(\pi) =  \left( \frac{ (2j-1) \pi}{X}\right)^2$. 
For an odd $m=2j-1$ and $k=\pi$, we have that 
$$
 G_m 
 = 2 \frac{E_{m+1}(\pi) - E_m(\pi)}{E_{m+1}(\pi) + E_m(\pi) } 
 \leq 2 \frac{\mu_{2j}(\pi)+V_{+} - \mu_{2j-1}(\pi)}{\mu_{2j}(\pi)+V_{+}+ \mu_{2j-1}(\pi) } 
 =  \frac{ 2X^2V_+}{ 2 \pi^2 m^2 + X^2V_+}. 
 $$
Putting the even and odd bounds together gives the desired result. 
\end{proof}

\subsubsection{High-contrast asymptotic results for $m\geq 1$.}
\begin{prop} \label{p:1DhighContrast} In the high-contrast limit $(V_+ = \infty$), the periodic arrangement where all $m$ intervals are the same length attains the maximum of $G_m$ with value $G_m^\star = \frac{6}{5}$. 
\end{prop}
\begin{proof}
In the high contrast limit, the eigenvalues converge to the Dirichlet-Laplacian eigenvalues for $m$ intervals. We denote  the length of the $m$ intervals where $V = 0$ by $L_1$, $L_2$, \ldots $L_m$ and without loss of generality we can assume that $L_1 \leq L_2 \leq \cdots \leq L_m$. The eigenvalues are then given by  
$$
\left\{ \left(\frac{2 \pi j}{ L_m} \right)^2, \ \left(\frac{2 \pi j}{ L_{m-1}} \right)^2, \ \ldots \ ,  \ \left(\frac{2 \pi j}{ L_1}\right)^2  \right\}, \qquad j \in \mathbb N_{+}. 
$$

If the $m$-th gap in the spectrum is between the eigenvalues $\left(\frac{2 \pi }{ L_1} \right)^2$ and $\left(2 \frac{2 \pi }{ L_m} \right)^2$, then the gap-to-midgap ratio is 
\begin{equation} \label{e:val1}
2 \frac{\left(2 \frac{2 \pi }{ L_m} \right)^2 - \left(\frac{2 \pi }{ L_1} \right)^2}{\left(2 \frac{2 \pi }{ L_m} \right)^2 + \left(\frac{2 \pi }{ L_1} \right)^2 } 
= 2 \frac{4 L^2_1 / L^2_m -1 }{ 4 L^2_1 / L^2_m + 1} = f\left( \frac{4 L_1^2}{L_m^2} \right).
\end{equation} 
where $f(\alpha) = 2\frac{\alpha - 1}{ \alpha + 1}$. Since $f(\alpha)$ is increasing and $L_1 \leq L_m$, \eqref{e:val1} is maximized when $L_1 = L_m$, which implies all intervals are of the same length and $G_m = f(4) = \frac{6}{5}$. 

\medskip

If not, the $m$-th gap must lie in one of the intervals
\[ 
\left( \left( \frac{2 \pi n}{L_m} \right)^2 , \left( \frac{2 \pi (n+1)}{L_m} \right)^2  \right), 
\qquad \qquad \qquad n = 1,\ldots,m. 
\]
Then we obtain the bound
\[
G_m \leq 2 \frac{ \frac{4 \pi^2 (n+1)^2}{L_m^2} - \frac{ 4 \pi^2 n^2}{L_m^2} }{ \frac{4 \pi^2 (n+1)^2}{L_m^2} + \frac{4 \pi^2 n^2}{L_m^2} } = f\left( \frac{(n+1)^2}{n^2} \right). 
\]
Since $g(n) := \frac{(n+1)^2}{n^2}$ is decreasing on $(0,\infty)$ and $f$ is increasing, we have that the composition $f\circ g(n)$ is a decreasing function. It follows that $f\circ g(n) < f\circ g(1) = f(4)$  for all $n =2,\ldots,m$. Since we can construct a configuration where $G_m = f(4)$, we conclude that the optimal $m$-th gap must lie in the interval 
$\left( \left( \frac{2 \pi}{L_m} \right)^2 , \left( 2 \frac{2 \pi}{L_m} \right)^2  \right)$. But clearly the $m$-th gap must lie above $m$ eigenvalues, so it must be that the gap lies in the interval $\left( \left( \frac{2 \pi}{L_1} \right)^2 , \left( 2 \frac{2 \pi}{L_m} \right)^2  \right)$, as considered above. 
\end{proof}

\subsection{Rearrangement Algorithm} For $m \geq 2$, we use the rearrangement algorithm (Algorithm \ref{alg:Rearrange}) to find the optimal potentials in  \eqref{e:opt}. 
We initialize the algorithm with the potential 
$$
V(x) = \begin{cases}
V_+ & \textrm{if } \cos(2\pi m x/X) >0 \\ 
0 & \textrm{otherwise}
\end{cases}. 
$$
For this initialization, the algorithm converges in just a few iterations; \revision{similar results were observed for other initializations.} As in Remark \ref{rem:MonDec}, we observe that the value of $G_m$ is decreasing on non-stationary iterations.  In Figure \ref{f:HigherModes}, we plot the optimal potentials and eigenfunctions corresponding to spectral band edges (left) and the dispersion relation (right). 
We make the following observations: 
\begin{enumerate}
\item As is well-known from the theory of Hill's equation, the eigenfunctions associated with the edges of the $m$-th gap have exactly $m$ zeros. Note that from Figure \ref{f:m1EigFunPlot}, depending on the value of $V_+$,  the eigenfunction may exhibit positive local  minima.  
\item The result in Corollary \ref{thm:mTransPoints} is observed; the potential maximizing $G_m$ has $m$ intervals where $V= V_+$. Additionally, the optimal potential is $X/m$-periodic. Although we can prove this result in the high contrast limit (see Proposition \ref{p:1DhighContrast}),  we are unable to prove this observation at this time for finite contrast, $V_+$.  
\item   In Figure \ref{f:HigherModes}, we observe that many of the spectral gaps are trivial. For example, in Figure \ref{f:HigherModes} for $m=3$, the gaps numbered 1,2,4,5,\ldots are trivial. Assuming that the optimal potential is $X/m$-periodic, this follows from the following easily proven Lemma. 
\begin{lem} Let $V$ be a periodic potential with period $X/m$. Then the $n$-th spectral gap of the operator $H_V$ acting on $H^1[0,X]$ can be non-trivial only if $m \mid n$. 
\end{lem}
\item Where one of the eigenfunctions (either $\psi_\alpha$ or $\psi_\beta$) takes a zero, the other eigenfunction has zero derivative; see Remark \ref{rem:OptSolNeuDir}. 
\item The value of $G_m^\star$ is decreasing in $m$; see Table \ref{t:OptValsOneDim} and Proposition \ref{prop:minfty}.
\end{enumerate}

\begin{table}[t!]
\begin{center}
\begin{tabular}{c|c}
$m$ & $G_m^\star$  \\
\hline
1 & 1.12370  \\
2 & 0.74391 \\
3 & 0.46766 \\
4 & 0.30895 \\
5 & 0.21550 
\end{tabular}
\end{center}
\caption{For $X = 1$ and $V_+ = 100$, the values of $G_m^\star$ for $m=1,\ldots,5$. }
\label{t:OptValsOneDim}
\end{table}%

\begin{figure}[t!]
\begin{center}
\includegraphics[width=.75\textwidth]{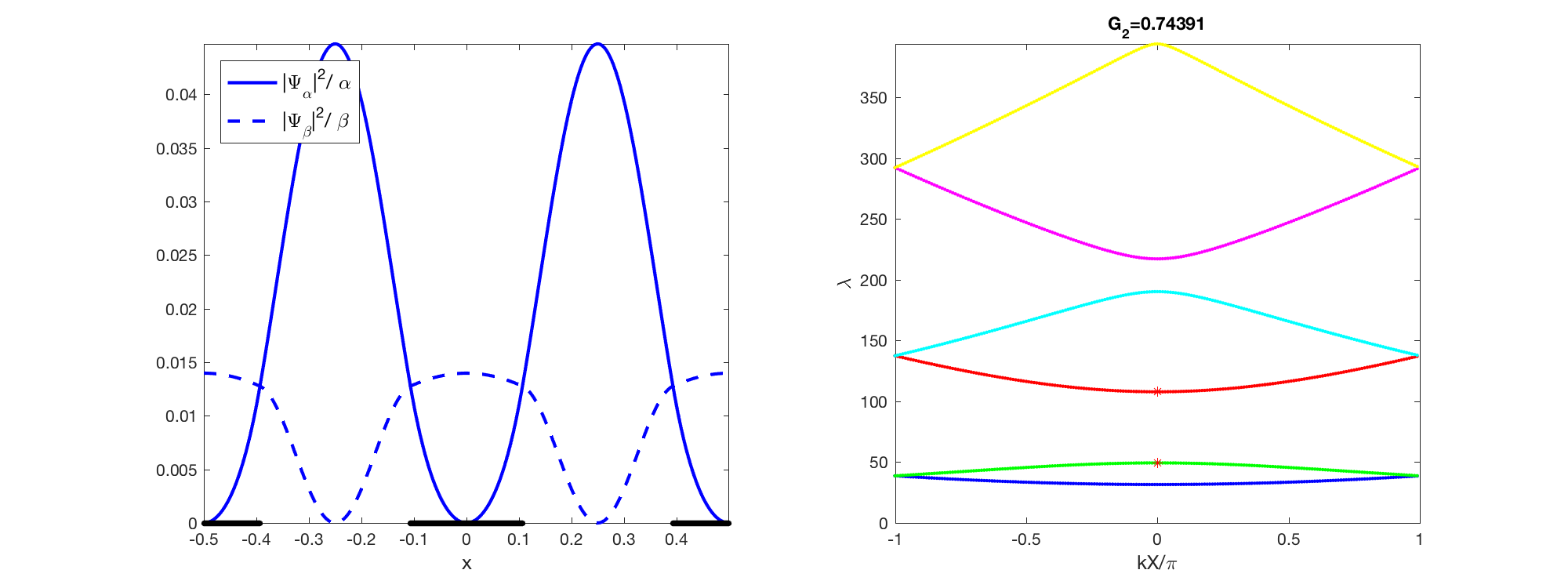}
\includegraphics[width=.75\textwidth]{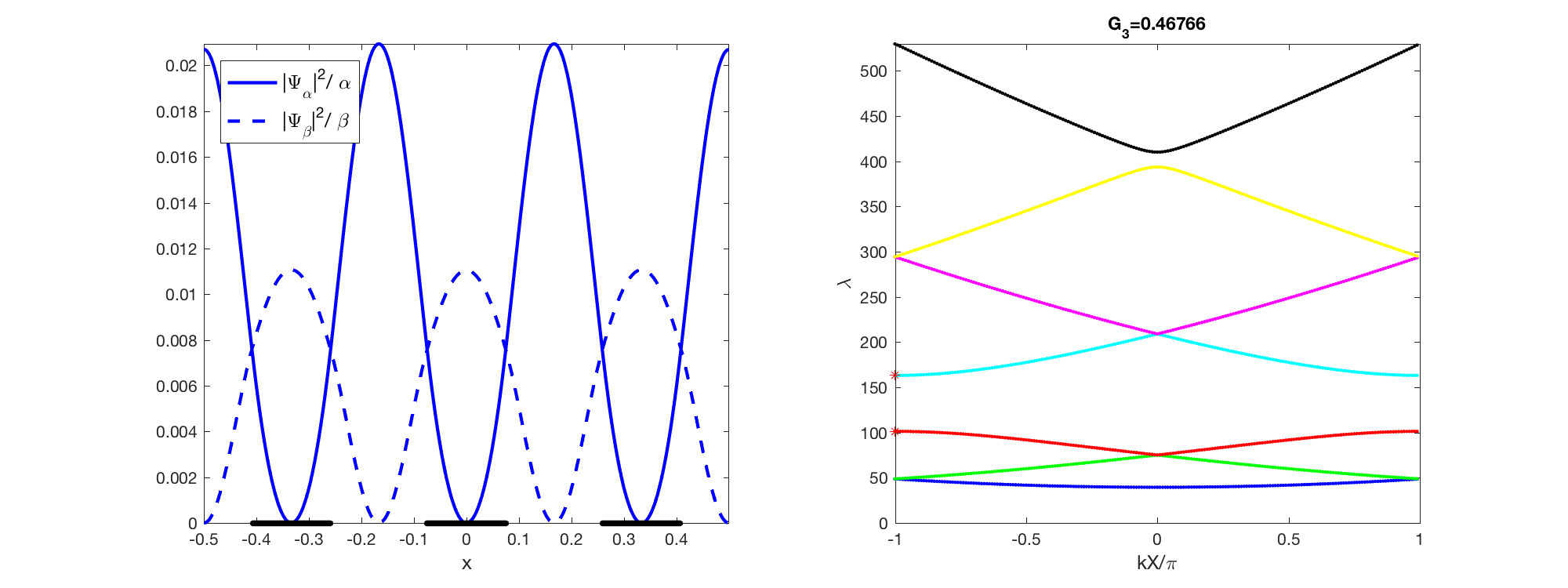}
\includegraphics[width=.75\textwidth]{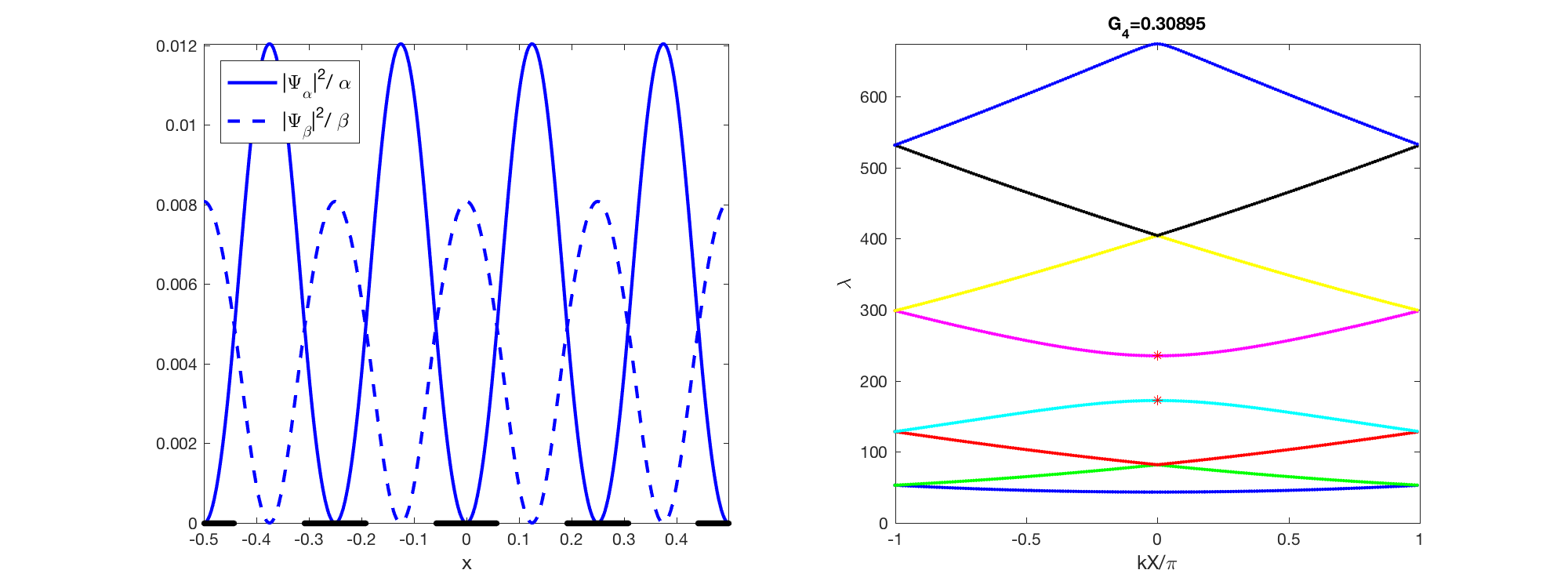}
\includegraphics[width=.75\textwidth]{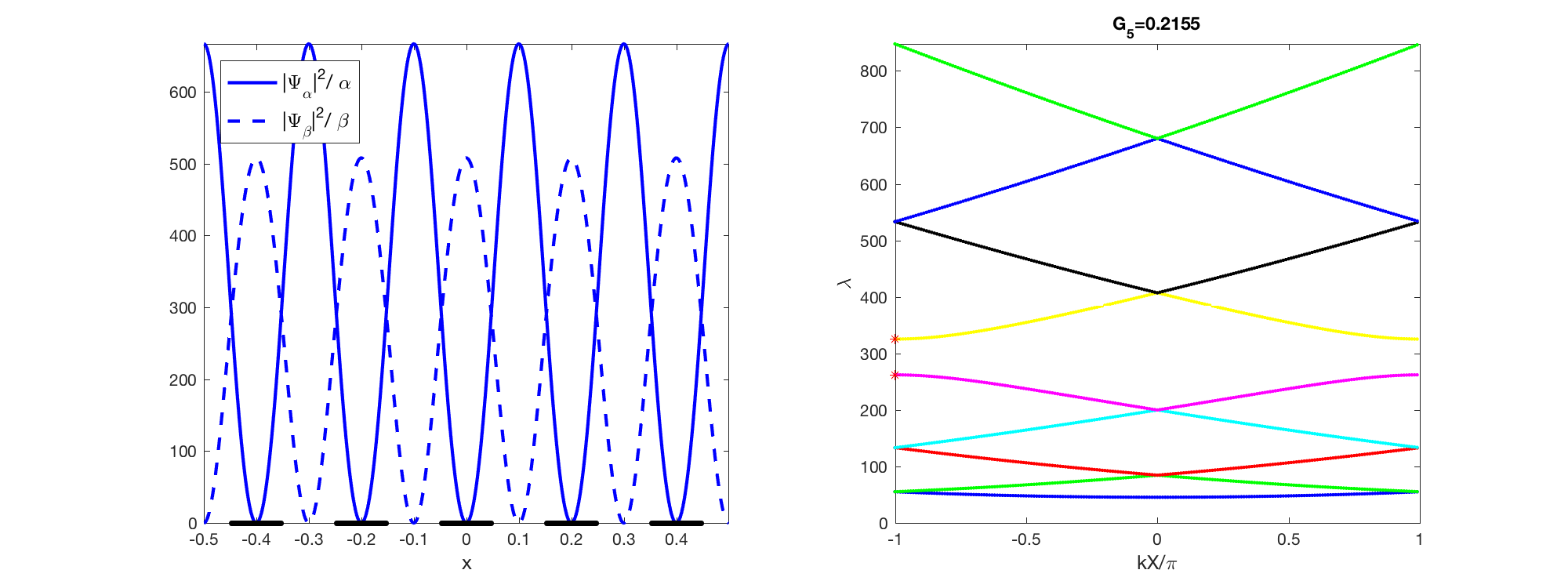}
\caption{Fix $X=1$ and $V_+ = 100$.  Each row corresponds to the maximizer of $G_m$ for $m=2,\ldots,5$.  
{\bf (Left)} The eigenfunctions corresponding to the spectral gap edges are plotted together with the set $\{ x \colon  V(x) = V_+  \}$  indicated on the $x$-axis by a thick black line. 
{\bf (Right)} The dispersion relation for the Schr\"odinger operator.  For $m=1$, see Figure \ref{f:OptCond}. 
}
\label{f:HigherModes}
\end{center}
\end{figure}

%%%%%%%%%%%%%%%%%%%%%%%%%%%%%%%%%%%%%%%%%%%%%%%%%%
\section{The two-dimensional eigenvalue  problem} \label{sec:2D}
We consider computing solutions of the eigenvalue problem \eqref{eq:SpectralProblem2} for fixed $\Gamma, V,$ and $k$. 

\subsection{Transformation of  \eqref{eq:SpectralProblem2} to a square domain.} 
Let $\Gamma$ be a unit-volume lattice. 
It is shown in Appendix \ref{sec:LattParam} that one can find parameters $(a,b) \in U$ such that \revision{$\Gamma$} is isometric to the lattice  with basis 
$$ 
B_{a,b} =  \begin{pmatrix} \frac{1}{\sqrt b} & \frac{a}{\sqrt b} \\ 0 & \sqrt b \end{pmatrix}.
$$ 
We denote this lattice by $\Gamma_{a,b}$ and the associated  $(a,b)$-torus by $T_{a,b} = \mathbb R^2/ \Gamma_{a,b}$.  
We consider the linear transformation $T_{0,1} \to T_{a,b}$ given by 
$$
y = B_{a,b} x 
\qquad \implies \qquad 
\nabla_y = B^{-t} \nabla_x. 
$$
where $B^{-t}$ is the inverse transpose matrix of $B$.
Transforming variables in \eqref{eq:SpectralProblem2}, we obtain 
\begin{subequations}
\label{e:TransfSpectralProb}
\begin{align}
& - (B_{a,b}^{-t}\nabla_x + i k) \cdot (B_{a,b}^{-t}\nabla_x + i k) p + \tilde V p = E p \\
& p \textrm{ periodic on } T_{0,1} . 
\end{align}
\end{subequations}
Here $\tilde V (x) = V(y)$ is the transformed potential.  
Thus, for an arbitrary lattice, we have transformed \eqref{eq:SpectralProblem2}  to a problem on the square. 
We refer to 
\begin{align*}
H_V(k;a,b): & =  - (B_{a,b}^{-t}\nabla_x + i k) \cdot (B_{a,b}^{-t}\nabla_x + i k) + \tilde V \\ 
&= - (\nabla_x + i B_{a,b}^{t} k) \cdot (B_{a,b}^{t} B_{a,b} )^{-1} (\nabla_x + i B_{a,b}^{t} k) + \tilde V
\end{align*}
as the \emph{transformed twisted Schr\"odinger operator}.

\subsection{Discretization} \label{sec:NumDisc}
We consider a square grid discretization of $T_{0,1}$. 
We use a simple nine point finite difference approximation to find spectrum of the transformed twisted Schr\"odinger operator in \eqref{e:TransfSpectralProb}. Denoting $N = (B^{t}B)^{-1}$ and $k_p = B^{-1} k$, the stencil for this discretization with lattice spacing $h$ is given by 
\begin{equation*}
\begin{bmatrix}
\frac{N(2,1) + N(1,2)}{ 4 h^2} & - \frac{N(2,2)}{h^2} - \frac{i}{h} k_p(2) & - \frac{N(2,1) + N(1,2)}{ 4 h^2} \\ 
- \frac{N(1,1)}{h^2} +\frac{i}{h} k_p(1) &  2 \frac{N(1,1) + N(2,2)}{h^2} + k^tk  & - \frac{N(1,1)}{h^2} - \frac{i}{h} k_p(1) \\  
- \frac{N(2,1) + N(1,2)}{ 4 h^2} & - \frac{N(2,2)}{h^2} + \frac{i}{h} k_p(2) &  \frac{N(2,1) + N(1,2)}{ 4 h^2} 
\end{bmatrix}. 
\end{equation*}
Let $L_{a,b}$ denote the matrix for this discretization of the transformed twisted Schr\"odinger operator incorporating the periodic boundary conditions. Abusing notation, we use $V$ to denote the discretization of the transformed potential, $\tilde V$. We obtain the parameterized  family of eigenvalue problems 
\begin{equation}
\label{e:DiscreteProb}
\left[ L_{a,b}(k) + \textrm{diag}( V)  \right] \ u = E \ u,  \qquad \qquad k \in \mathcal B.
\end{equation}
Note that $ L_{a,b}(k) + \textrm{diag}( V)$ is a Hermitian matrix.  The dispersion surfaces are approximated by $E_j(k)$. 

\subsection{Symmetry and Discretization of the Brillouin Zone} \label{sec:IBZ} 
Symmetries within  the Brillouin zone can be used to further reduce the number of eigenvalue problems  need to be solved for the optimization problem \eqref{e:opt}. In this section, we review these well-known symmetries. 

For any real potential, the dispersion relation is symmetric with respect to $k \mapsto -k$. This follows from taking the complex conjugate of both sides of \eqref{eq:SpectralProblem2a}, 
\[ H_V(-k)  \ \overline{p} = \overline{ H_V(k) \ p } = \overline{ E \ p} = E \  \overline{p}.  \]
In one dimension, this symmetry can be observed in Figure \ref{f:HigherModes}. 

If the potential has the full symmetry of a lattice (both translational and rotational) then additional symmetry is inherited by the Brillouin zone. In two dimensions, square and triangular lattices have  rotational symmetries. Let $R_\theta = \begin{pmatrix} \cos (\theta)  & - \sin(\theta) \\ \sin(\theta) & \cos(\theta) \end{pmatrix}$ denote the matrix which rotates points about the origin in the counterclockwise direction by angle $\theta$. We denote the operator $\mathcal{R}_\theta \colon L^2(\mathbb R^2) \to L^2(\mathbb R^2)$ defined by 
\[ \mathcal{R}_\theta f(x) = f(R_\theta^t x). \]
If $V$ is $\mathcal{R}_\theta$ invariant, then we have that 
\begin{align*} \mathcal{R}_\theta  \ H_V(k)  \ f 
&=  -\mathcal{R}_\theta  \left( \Delta + 2 i k^t \nabla - k^t k \right)  f  + \mathcal{R}_\theta  V f  \\
&=  -  \left(\Delta + 2 i (R_\theta k) ^t \nabla - (R_\theta k)^t (R_\theta k ) \right) \mathcal{R}_\theta  f  +  V \mathcal{R}_\theta f  \\ 
&=  H_V(R_\theta k)  \ \mathcal{R}_\theta \ f. 
 \end{align*}
 Here, we have used the facts that the Laplacian commutes with $\mathcal{R}_\theta$, the identity $\mathcal{R}_\theta \nabla = R_\theta^t \nabla \mathcal{R}_\theta$, and that $R_\theta$ is unitary. 
If $(E,p)$ is an eigenapair satisfying \eqref{eq:SpectralProblem2}, we compute 
\[ H_V(R_\theta k) \ (\mathcal{R}_\theta p) = \mathcal{R}_\theta \  (H_V(k) \ p) = \mathcal{R}_\theta \  (E \ p) = E \ (\mathcal{R}_\theta p). \] 
This implies that $(E,\mathcal{R}_\theta p)$ is also an eigenpair of \eqref{eq:SpectralProblem2} with quasi-momentum $R_\theta k$. 

\begin{figure} [t!]
\begin{center}
\includegraphics[height=.30\textwidth]{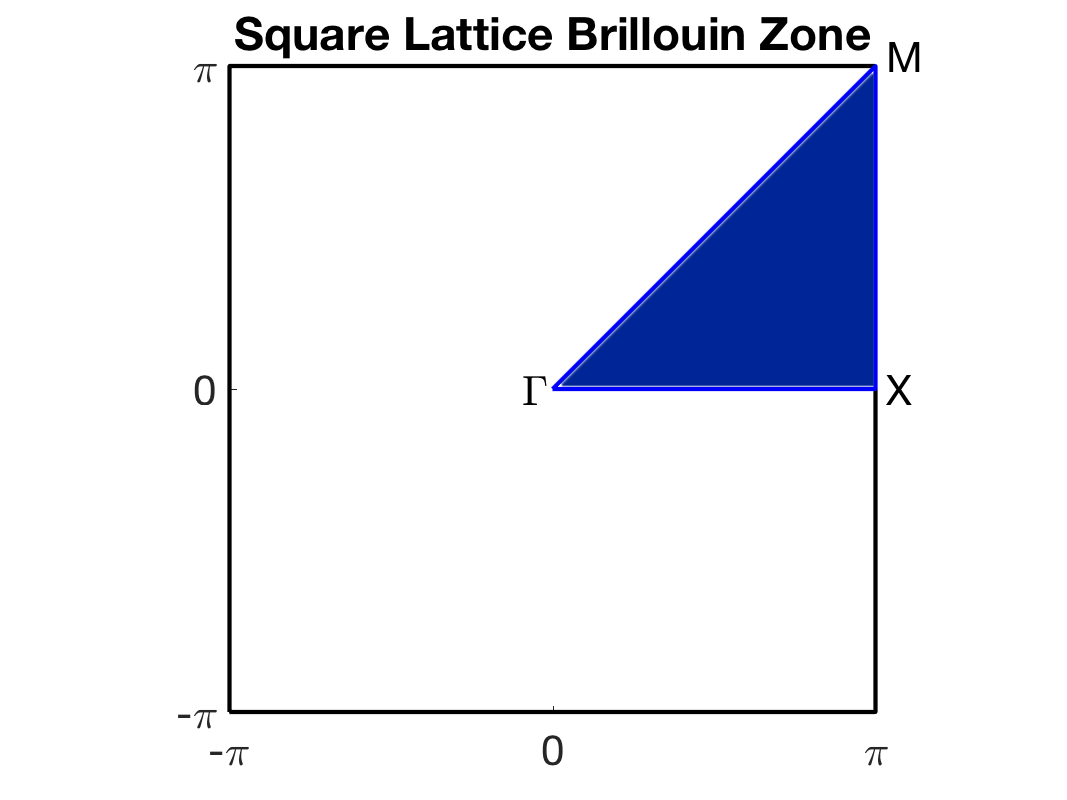}
\includegraphics[height=.30\textwidth]{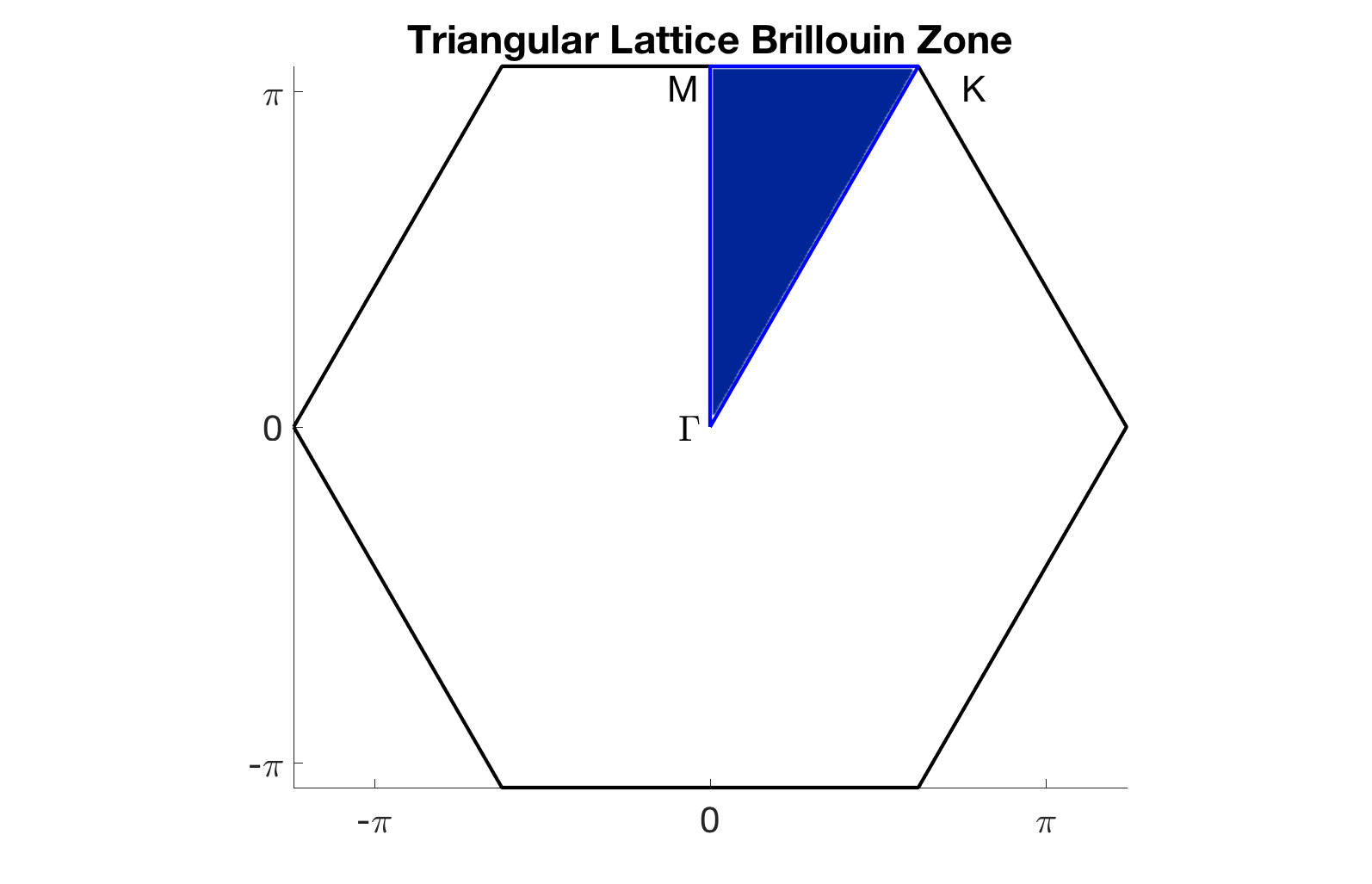}
\caption{The irreducible Brillouin zones (IBZ) for the square (left) and triangular (right) lattices, along with the traditional names of the vertices of the IBZ.}
\label{f:BriZone}
\end{center}
\end{figure}

The rotational and inversion symmetries for the square and triangular lattice together imply that the Brillouin zones have eight-fold and twelve-fold symmetry. The region modulo this symmetry is referred to as the \emph{irreducible Brillouin zone} (IBZ). The IBZ for the square and triangular lattices are \revision{colored in blue} in Figure \ref{f:BriZone}. We label the vertices and origin of the IBZ in the standard way. In computations involving the square and triangular lattices, we use this symmetry.  

As described in \cite{kuchment2016}, it  was formally conjectured and widely believed that the extrema of spectral bands were attained at the boundary of the IBZ. Although this has been shown to be false in general, in practice for generic potentials, extrema are often located at such symmetry points. We do not understand which classes of potentials satisfy this condition. 

Nonetheless, this observation motivates the following heuristic for studying extremal gaps, and in particular solving \eqref{e:opt}, which was introduced in \cite{Dobson1999,cox2000band}. First we find a potential which attains the maximal gap-to-midgap ratio for quasi-momentum only on the boundary of the IBZ. It may be the case that the spectral bands corresponding to this potential has extrema  at the boundary or interior of the IBZ. If it happens that the extrema are attained at the boundary, a condition that can easily be checked \emph{ex post facto}, then this potential is optimal for \eqref{e:opt}. This is the approach taken here, and it is observed that the extrema for the potentials attaining the maximum in \eqref{e:opt} have extremal spectra on the interior of the IBZ.   

For computations for the square and triangular lattices we discretize the boundary of the IBZ \revision{by uniformly distributing  $15$ points on each side of the triangles $\Gamma$-$X$-$M$ and $\Gamma$-$K$-$M$ in Figure \ref{f:BriZone}, respectively.} In Section \ref{s:Opt2dLattices}, when we consider other lattices besides the square and triangular lattices, we'll discretize half of the Brillouin zone, as justified by the fact that the potential is real (see above). 
   
\section{Optimization problem in two dimensions} \label{s:Opt2d}
We consider the optimization problem  \eqref{e:opt} in two dimensions.   Theorem \ref{thm:exist}, guarantees  the existence of an optimal potential in two dimensions for fixed $\Gamma$. The computational challenge for the two-dimensional problem is that the eigenvalue may no longer be simple and so the Fr\'echet derivative in Lemma \ref{prop:sens} may no longer be valid. In particular, if  the rearrangement algorithm (Algorithm \ref{alg:Rearrange}),  introduced in Section \ref{sec:1D},  is applied directly, one finds that the potential will alternate between non-optimal potentials. We modify the rearrangement algorithm, by reformulating the eigenvalue problem as a semi-definite inequality. Discretizing this reformulation gives a semi-definite program (SDP).

\subsection{SDP formulation}
The optimization problem \eqref{e:opt} of maximizing the $m$-th gap can be fully written out as 
\begin{subequations}
\label{e:opt2}
\begin{align}
\label{e:opt2a}
\max \quad &   \ \frac{\beta - \alpha}{(\alpha + \beta)/2 } \\
\label{e:opt2b}
\textrm{s.t.} \quad & E_m(V,k) \leq \alpha \textrm{ and }  E_{m+1}(V,k) \geq \beta &&  k \in \mathcal{B} \\ 
& H_V(k) \ p_j(x) = E_j \   p_j(x)  &&  x \in \mathbb R^2 / \Gamma, \ k \in \mathcal{B}, \ j=m, m+1 \\ 
& p_j \textrm{ is $\Gamma$-periodic} && j=m, m+1 \\
& 0  \leq V(x)  \leq  V_+  && x\in  \mathbb R^2 / \Gamma
\end{align}
\end{subequations}
Note that here we have introduced two additional parameters $\alpha, \beta > 0$. At the optimum,  the constraints in \eqref{e:opt2b} will be active for some value of $k \in \mathcal B$ so that $\alpha = \max_{k \in \mathcal B} \ E_m(V,k)$ and $\beta = \min_{k \in \mathcal B} \ E_{m+1}(V,k)$.  The equivalence of  \eqref{e:opt}  and  \eqref{e:opt2} then follows from the fact that the objective in \eqref{e:opt2a} is monotonically increasing in $\beta$ and decreasing in $\alpha$. 

Let $H^{1}_{\textrm{per}}$ denote the space of  periodic  $H^1$ functions on the torus, $\mathbb R^2 / \Gamma$. 
For each $k \in \mathcal B$, let $\Pi_m^k \colon H^{1}_{\textrm{per}} \to H^{1}_{\textrm{per}}$ be a  rank $m$ projection. Then we can further rewrite the constraints for the optimization problem \eqref{e:opt2}, as
\begin{subequations} \label{e:opt3}
\begin{align}
\max \quad& \frac{\beta - \alpha}{(\alpha + \beta)/2 }  \\
\textrm{s.t.} \quad & \Pi_m^k \left( H_V(k) - \alpha \textrm{Id}  \right) \Pi_m^k \preceq 0 && k \in \mathcal B  \\ 
& (I- \Pi_m^k) \left( H_V(k) - \beta  \textrm{Id} \right) (I- \Pi_m^k ) \succeq 0  && k \in \mathcal B \\ 
& \Pi^k_m \textrm{ is a rank $m$ projection} && k \in \mathcal B\\
&0  \leq V(x)  \leq  V_+  &&   x\in  \mathbb R^2 / \Gamma. 
\end{align} 
\end{subequations}
The advantage of this formulation is that it no longer references the (non-differentiable) eigenvalues.  
 At the optimum, for the values of $k\in \mathcal B$ which attain $ \max_{k \in \mathcal B}  E_m(V,k)$, the rank $m$ projection  will have the from 
$$
\Pi^k_m u  (x) = \sum_{j\in [m]}  \ p_j(x,k)  \  \int_{\mathbb R^2/\Gamma} \   p_j(y,k) \ u(y) \ dy. 
$$
That is, $\Pi^k_m$ will be the projection onto a space spanned by the first $m$ eigenfunctions of $H_V(k)$. 

Finally, the objective function in  \eqref{e:opt3} is a linear fractional function. A homogenization transformation of the variables can be used to equivalently rewrite  \eqref{e:opt3} as an optimization problem with linear objective \cite{men2010bandgap}. Namely making the substitutions $\alpha = \tilde \alpha / \theta$, $\beta = \tilde \beta / \theta$, and $V = \tilde V/\theta$, we obtain the equivalent problem 
\begin{subequations} \label{e:opt3}
\begin{align}
\max \quad& \tilde \beta - \tilde \alpha   \\
\textrm{s.t.} \quad & \Pi_m^k \left( \theta H(k) + \tilde V - \tilde \alpha \textrm{Id}  \right) \Pi_m^k \preceq 0 && k \in \mathcal B  \\ 
& (I- \Pi_m^k) \left( \theta H(k) + \tilde V - \tilde\beta  \textrm{Id} \right) (I- \Pi_m^k ) \succeq 0  && k \in \mathcal B \\ 
& \Pi^k_m \textrm{ is a rank $m$ projection} && k \in \mathcal B\\
&0  \leq \tilde V(x)  \leq  \theta V_+  &&   x\in  \mathbb R^2 / \Gamma \\
& \tilde\alpha + \tilde\beta = 2
\end{align} 
\end{subequations}

We used the discretizations described in Section \ref{sec:NumDisc} to obtain a family of finite-dimensional eigenvalues problems. Further discretizing the Brillouin zone $\{k_j\}_{j\in[q]} \subset \mathcal B$, gives a finite-dimensional approximation of the optimization problem \eqref{e:opt3}. 

As in  Algorithm  \ref{alg:Rearrange}, we will employ an algorithm which alternates between two steps. In the first step, the potential $V$ is fixed and a subspace corresponding to the projections $\Pi_m^k$ is computed. In the second step, the projections are fixed and a linear SDP (a convex optimization problem) is solved \revision{by using CVX, a package for specifying and solving convex programs \cite{cvx,gb08}} to update the potential. The details are given in  Algorithm \ref{alg:SDPrearrange}. Note that for simplicity, in \eqref{e:SDP}, we have written SDP with a linear fractional objective, but the same homogenization transformation made just above \eqref{e:opt3}, can be used to rewrite \eqref{e:SDP} as a linear SDP. 

\begin{algorithm}[t!]
\caption{\label{alg:SDPrearrange} 
The rearrangement algorithm for the two-dimensional problem in \eqref{e:opt}. }
\vspace{.2cm}

\begin{algorithmic}[t]
\STATE{\bfseries Input:} Fix $V_+ > 0$, $m \in \mathbb N^{+}$. Initialize $V$ in $\mathcal A(\Gamma, V_+)$ defined in \eqref{eq:Ad}. 

\vspace{.2cm}

\WHILE {the potential is not stationary}
\STATE 1. For an $n$-point spatial discretization, $\{x_\ell\}_{\ell=1}^n \subset T_{0,1}$, and discretization of the Brillouin zone, $\{k_j\}_{j \in [q]} \subset \mathcal B$,  compute eigenfunctions $u_i(k_j) \in \mathbb C^n$ solving \eqref{e:DiscreteProb} for $i=1,\ldots, m+\mu$. Form the matrices 
\begin{align*}
& U_{\alpha,j} =  [u_1(k_j) | \cdots | u_m(k_j)]  \in \mathbb C^{n \times m}, \quad j\in[q] \\ 
& U_{\beta,j}  = [u_{m+1}(k_j) | \cdots | u_{m+\mu}(k_j)]  \in \mathbb C^{n\times \mu}, \quad j\in[q]. 
\end{align*}

\bigskip

\STATE 2. Solve the linear fractional SDP
\begin{subequations}
\label{e:SDP}
\begin{align}
\label{eq:SDPa}
\max_{a,b,V} \quad & \frac{ \beta - \alpha}{ (\alpha + \beta)/2} \\
\label{eq:SDPb}
\textrm{s.t.} \quad & U_{\alpha,j}^* \left[ L_{a,b}(k_j) + \textrm{diag}(V) - \alpha  \textrm{Id} \right] U_{\alpha,j} \preceq 0 && j\in[q]  \\ 
\label{eq:SDPc}
& U_{\beta,j}^* \left[ L_{a,b}(k_j) + \textrm{diag}(V) - \beta  \textrm{Id}  \right] U_{\beta,j}  \succeq 0  && j \in [q]  \\ 
\label{eq:SDPd}
& 0 \leq V_\ell  \leq  V_+ &&   \ell \in [n]. 
\end{align}
\end{subequations}
\ENDWHILE
\end{algorithmic}
\end{algorithm}

\subsection{Karush Kuhn Tucker (KKT) conditions} \label{s:KKT}
We derive the KKT equations for the semi-definite optimization problem in   \eqref{e:SDP}. Since the constraints are linear, the KKT equations are necessarily satisfied at every maximum point $(\alpha^*, \beta^*, V^*)$. To reduce notation, we understand $a,b$ to be fixed and denote $L_j = L_{a,b}(k_j)$. Introducing the dual variables $A_j \in \mathbb S^m$  and  $B_j \in \mathbb S^\mu$ for $j\in [q]$ and  $f_+, f_- \in \mathbb R^n$, we have the Lagrangian  
\begin{align*}
\mathcal L(\alpha, \beta, V; A_j, B_j, f_+,f_-) & = \frac{\beta - \alpha}{ (\alpha + \beta)/2}  +  \langle f_+\ , \ V_+ - V \rangle + \langle f_-\ , \ V-0 \rangle \\ 
& \quad - \sum_{j\in q} \Big\langle A_j \ ,  \ U_{\alpha,j}^* \left[ L_j  + \textrm{diag}(V) - \alpha  \textrm{I} \right]  U_{\alpha,j}  \Big\rangle_F \\
& \quad + \sum_{j\in q}  \Big\langle  B_j \ , \  U_{\beta,j}^* \left[ L_j + \textrm{diag}(V) - \beta  \textrm{I}  \right] U_{\beta,j} \Big\rangle_F. 
\end{align*}
 
The stationarity conditions are obtained from the equations $\frac{\partial \mathcal L}{ \partial \alpha} = 0$, $\frac{\partial \mathcal L}{ \partial \beta} = 0$, and $\nabla_V \mathcal L = 0$. The first two conditions are given by  
\begin{subequations}
\label{eq:Stat12}
\begin{align}
\sum_{j \in [q]} \textrm{tr} (A_j )  &= \frac{4 \beta}{(\alpha+\beta)^2} \\
\sum_{j \in [q]} \textrm{tr} (B_j ) &= \frac{4 \alpha}{(\alpha+\beta)^2}. 
\end{align}
\end{subequations}
The third stationary condition is given by 
\begin{align}
\label{eq:Stat3}
\sum_{j \in [q]} U_{\beta,j} B_j U_{\beta,j}^* - U_{\alpha,j} A_j U_{\alpha,j}^*   = \textrm{diag}(f_+ - f_-). 
\end{align}

The primary feasibility conditions \eqref{eq:SDPb}, \eqref{eq:SDPc}, and \eqref{eq:SDPd} must hold. The dual feasibility conditions are given by 
\begin{subequations}
\label{eq:DualFeas}
\begin{align}
& A_j \succeq 0, \quad  j \in [q] \\
& B_j \succeq 0, \quad  j \in [q] \\
& f_+, \ f_- \geq 0. 
\end{align}
\end{subequations}
Finally, the complementary slackness conditions state that 
\begin{subequations}
\label{eq:CompSlack}
\begin{align}
\label{eq:CompSlacka}
& \langle f_+ \ , \ V_+ - V \rangle = 0 \\
\label{eq:CompSlackb}
& \langle f_- \ , \ V- 0  \rangle = 0  \\ 
&  \Big\langle A_j \ ,  \ U_{\alpha,j}^* \left[ L_j  + \textrm{diag}(V) - \alpha  \textrm{I} \right]  U_{\alpha,j}  \Big\rangle_F  = 0 \\ 
& \Big\langle  B_j \ , \  U_{\beta,j}^* \left[ L_j + \textrm{diag}(V) - \beta  \textrm{I}  \right] U_{\beta,j} \Big\rangle_F = 0 . 
\end{align}
\end{subequations}

\subsection{Properties of Algorithm \ref{alg:SDPrearrange}} In this section we use the KKT equations for the linear-fractional SDP \eqref{e:SDP}, derived in Section \ref{s:KKT},   to prove properties about Algorithm \ref{alg:SDPrearrange}. 

We say a potential $V$, defined on a grid $\{x_\ell\}_{\ell=1}^n$, and constrained so that $ V_\ell \in [0 ,   V_+]$ for all $\ell \in [n]$ is \emph{bang-bang} if $V_\ell \in \{0,V_+\}$ for every $\ell \in [n]$. We say that the potential is \emph{weakly bang-bang} if  there exists at least one grid point $\ell \in [n]$ at which either $V_\ell = 0$ or $V_\ell = V_+$. 

\begin{prop} \label{prop:Wbb}
At every iteration of Algorithm \ref{alg:SDPrearrange} such that $\beta \neq \alpha$, the potential is weakly bang bang. 
\end{prop}
\begin{proof}
In the second step of Algorithm \ref{alg:SDPrearrange}, we obtain a new potential by  solving the  linear-fractional SDP in \eqref{e:SDP}. Such a potential necessarily satisfies the KKT equations, given in  \eqref{eq:Stat12}, \eqref{eq:Stat3},  \eqref{eq:SDPb}-\eqref{eq:SDPd}, \eqref{eq:DualFeas}, and \eqref{eq:CompSlack}. 

Proceeding by contradiction, assume that $V_\ell \in (0,V_+)$ for every $\ell \in [n]$.
By \eqref{eq:CompSlacka} and \eqref{eq:CompSlackb} we have that 
\[f_+ = f_- = 0. \]
From \eqref{eq:Stat3}, we obtain 
\[ \sum_j U_{\beta,j} B_j U_{\beta,j}^* = \sum_j  U_{\alpha,j} A_j U_{\alpha,j}^*.  \]
 Taking the trace of both sides and using the circular trace identity and $U_{\cdot,j}^*  U_{\cdot,j} = \textrm{I}$, we have that 
\[ \sum_j \textrm{tr} (B_j) = \sum_j  \textrm{tr}(A_j).  \]
 But, by the identities in \eqref{eq:Stat12}, this would imply that $\alpha  = \beta$. 
\end{proof}

The following proposition shows  that Algorithm \ref{alg:SDPrearrange} generalizes Algorithm \ref{alg:Rearrange}  to higher dimensions. 
\begin{prop} \label{prop:Alg2GenAlg1}
The sequence of potentials defined in Algorithm \ref{alg:Rearrange} solve the KKT conditions for the linear fractional SDPs  \eqref{e:SDP}  in Algorithm  \ref{alg:SDPrearrange}. 
\end{prop}
\begin{proof}
Since the spectral bands in one dimension are monotonic on the intervals $[-\pi,0]$ and $[0,\pi]$, we may assume that $q=1$ and the subspaces are of dimension one ($p=1$).   Let $k=0$ for $m$ even and $k=\pi$ for $m$ odd so that $\alpha = E_m(k)$ and  $\beta = E_{m+1}(k)$. We also have that the bases for the subspaces in Algorithm  \ref{alg:SDPrearrange} are given by  $U_\alpha = u_m(k) = \psi_\alpha$ and $U_\beta = u_{m+1}(k) = \psi_\beta $. 

The stationary conditions in \eqref{eq:Stat12} reduce to 
\begin{align*}
A  = \frac{4 \beta}{(\alpha+\beta)^2}  
\qquad \qquad \textrm{and} \qquad \qquad
B = \frac{4 \alpha}{(\alpha+\beta)^2}. 
\end{align*}

On any set of grid points where $V = V_+$, we have that $V \neq 0$ so by the complementary slackness condition, we have that $f_- = 0$.  It follows from \eqref{eq:Stat3} that on such nodes, $\ell$, we must have 
$$
(f_+)_\ell =  \frac{4 \alpha}{(\alpha+\beta)^2} (\psi_\beta^2)_\ell -  \frac{4 \beta}{(\alpha+\beta)^2}  ( \psi_\alpha^2)_\ell  \geq 0. 
$$
Similarly, on any set of grid points where $V = 0$, we must have 
$$
(f_-)_\ell =    \frac{4 \beta}{(\alpha+\beta)^2}  (\psi_\alpha^2)_\ell  - \frac{4 \alpha}{(\alpha+\beta)^2} (\psi_\beta^2)_\ell  \geq 0. 
$$
But the potential chosen in Algorithm \ref{alg:Rearrange} defines the potential by choosing the sets $\{ V= V_+ \}$ and $\{V=0\}$ exactly so that these two inequalities are satisfied. 
\end{proof}

\subsection{Computational Results. } \label{sec:CompResults}
We study the dependence of $G^\star_{m,\Gamma,V_+}$ and the optimizer on the parameters $m$ for fixed  $V_+=100$ and lattice $\Gamma$. In this section, we will take $\Gamma$ to either be the square or triangular lattice. 
In Section \ref{s:varyV+}, we discuss the optimizer as the parameter $V_+$ varies and in Section \ref{s:Opt2dLattices} we discuss the dependence on  the lattice, $\Gamma$.

For $V_+ = 100$, the optimal potentials and corresponding dispersion relations for $m=1,\ldots,8$ are plotted for the square (Figures \ref{f:2DsqA} and \ref{f:2DsqB}) and triangular lattices (Figures \ref{f:2DtriA} and \ref{f:2DtriB}). The periodic extension of the potential is plotted on a $3\times 3$ array of the primitive cell. The dispersion relations are plotted over the boundary of the irreducible Brillouin zone as shown in Figure \ref{f:BriZone}.  The optimal values found are recorded in Table \ref{t:OptVals}. 

We  used a $64\times 64$ square grid \revision{$T_{0,1}$} for the computations of the spectrum for each value $k$. The values of $k$ used come from discretizing \revision{the boundary of} IBZ using 45 points. To generate these computational results, we  initialized the potential using a variety of different guesses and  report the potentials found with largest objective values. 

\begin{table}[t!]
\begin{center}
\begin{tabular}{c|c|c}
$m$ & square & triangular  \\
\hline
1 & 0.7722 & 0.7963 \\
2 & 0.5461 & 0.4773 \\
3 & 0.4130 & 0.4973 \\
4 & 0.3957 & 0.3674  \\
5 & 0.1663 & 0.2262  \\
6 & 0.1572 & 0.2075  \\
7 & 0.1978 & 0.2087  \\
8 & 0.1939 & 0.09982  
\end{tabular}
\end{center}
\caption{For $V_+ = 100$, and $m=1,\ldots,8$, the optimal values, $G_m^\star$, for the unit-volume square and triangular lattices.}
\label{t:OptVals}
\end{table}%

We make the following  observations: 
\begin{enumerate}
\item For a fixed lattice, the values of $G_m^\star(V_+,\Gamma)$  are not necessarily decreasing in $m$, see $m=6,7$ for either the square or triangular lattice.
\item As in Remark \ref{rem:MonDec}, we observe that the  value of $G_m$ is strictly increasing for non-stationary iterations of Algorithm \ref{alg:SDPrearrange}. 
\item In Proposition \ref{prop:Wbb}, we prove that the potentials at every iteration of Algorithm \ref{alg:SDPrearrange} are weakly bang-bang. In practice, we observe them to be bang-bang. 
\item When $G_m$ is maximized, the $m$-th and $(m+1)$-th spectral bands are very flat. 
\item For triangular potentials  with honeycomb symmetry, {\it e.g.}, m=1, 2, 3, and 4, we observe that the spectral bands feature Dirac points at the $K$ points of the Brillouin zone; see  \cite{fefferman2012honeycomb}.
\end{enumerate}

\subsection{An asymptotic result for $m \to \infty$}
Similar to Proposition \ref{prop:minfty}, we prove the following asymptotic result for $G_m$ as $m \to \infty$. 
\begin{prop} \label{prop:minftyD2}
Let $ \Gamma $ be a Bravais lattice and $T = \mathbb R^2 / \Gamma$. 
Let $V \in  \mathcal A(\Gamma,V_+)$. Then 
\[
G_m \leq  \frac{ \lambda_{m+1}^D + V_+ - \lambda_m^N }{\lambda_{m+1}^N + V_+ +  \lambda_m^N  }. 
\]
Here $\lambda_j^D$ and $\lambda_j^N$ denote the $j$th eigenvalue of the Dirichlet- and Neumann-Laplacian on $T$ respectively. 
\end{prop}
\begin{proof}
For $V \in  \mathcal A(\Gamma,V_+)$,  we have the semidefinite ordering
\[
- \Delta^N \ \preceq \ H_0(k) \ \preceq \ H_0(k) + V \ \preceq \ H_0(k) + V_+ \ \preceq \  - \Delta^D + V_+, 
\]
where $ \Delta^D$ and $- \Delta^N$ denote the Dirichlet- and Neumann-Laplacians respectively. The ordering follows from the variational formulation for \eqref{eq:SpectralProblem} and realizing that the admissible set satisfies $H_0^1 \subset H_k^1 \subset H^1$, where  $H_k^1$ is the set of $H^1$ functions that have quasi-momentum $k \in \mathcal B$. 
This semidefinite ordering implies that 
\[ \lambda_j^N \leq E_j(k) \leq \lambda_j^D + V_+.  \]
We then compute 
 \begin{align*}
 G_m 
 &=  2 \frac{ \min_{k\in \mathcal B} E_{m+1}(k) - \max_{k\in \mathcal B} E_{m}(k) }{\min_{k\in \mathcal B} E_{m+1}(k) + \max_{k\in \mathcal B} E_{m}(k) }  \\ 
 &\le  2 \frac{ \lambda_{m+1}^D + V_+ - \lambda_m^N }{\lambda_{m+1}^N + V_+ + \lambda_{m}^N }.
 \end{align*}
  Here, as in the proof of Proposition \ref{prop:minfty}, we have used the fact that $f(\alpha,\beta) = 2\frac{\alpha - \beta}{\alpha +\beta}$ is increasing in $\alpha$ and decreasing in $\beta$ for $\alpha, \beta>0$.  
\end{proof}

Taking $m\to \infty$ in Proposition \ref{prop:minftyD2}, and using Weyl's law, 
$ \lambda_m^N, \  \lambda_m^D = \frac{4 \pi}{ |T|} m + O(1), $
we have that 
\[G_m \leq 2 \frac{V_+ + \frac{4 \pi}{|T|} + O(1)}{V_+ + \frac{4 \pi}{|T|} (2m+1) + O(1)} = O(m^{-1}). \] 

\subsection{Computational results for varying $V_+$}  \label{s:varyV+}

In Figure \ref{f:varyV+}, we also plot the optimal potentials and dispersion surfaces for $m=3$ and the square lattice (fixed $\Gamma$) as we vary the strength of the potential, $V_+$.  Of course, the maximum value, $G^*_{m,V_+}$, is non-decreasing in $V_+$, since the admissible set, $\ad(\Gamma,V_+)$, is enlarging. Numerically we observe both a change in the symmetries of the optimal potential and the number of components per unit cell  on the set where $V = 0$. For $V_+$ large, the optimal potential consists of three regions where $V=0$, that are roughly disk-like arranged in a triangular grouping. As $V_+$ is decreased, the regions merge. 

This leads us to ask the natural question: 
\emph{For fixed $m\in \mathbb N^{+}$ and Bravais lattice, $\Gamma$, what is the smallest value of $V_+$ such that $G_{m,V_+,\Gamma}^* > 0$?} We find that the 3rd gap with a square-lattice potential can be open if $V_+$ is greater than $\approx 40$. 

We study this question  further in Figure \ref{f:OptValvaryV+}. Here we plot $G_{m,\Gamma,V_+}^*$ vs. $V_+$ for $m=1,2,3$ and $\Gamma$ the square and triangular lattices. An approximate answer to the question is given by the $x$-intercept of these curves. Note that for $m=1,3$ a triangular lattice potential can have a spectral gap for smaller values of $V_+$ than a square lattice potential. The opposite is true for $m=2$. 

\subsection{An asymptotic result in the high contrast limit ($V_+ \to \infty$)} Motivated by the computational results in Section \ref{sec:CompResults}, 
throughout this section, we will make the following assumption. 
\begin{assum} \label{assum:bangbang}
Fix $m\in \mathbb N_{+}$, $V_+>0$, and $\Gamma$. Let $T =\mathbb R^2 / \Gamma$. Assume any potential attaining the maximum of \eqref{e:opt} is of the bang-bang form 
\begin{equation} \label{e:BangBang}
V(x) = \begin{cases}
V_+ & x \in \Omega_+ \\
0 & x \in \Omega_-
\end{cases},
\end{equation}
where $\Omega_- \in \Theta$ and  $\Omega_+ = T\setminus \Omega_-$.
Here $\Theta$ denotes  the class of domains  
\begin{align*}
\Theta = \{ \Omega \colon & \Omega  \textrm{ is an open subset, compactly contained in  $T\subset \mathbb R^2$}  \}. 
\end{align*}
\end{assum}

We  recall the following preliminary result.  
\begin{thm}[\cite{hempel2002spectral}] \label{prop:ConvVinfty}
Using the notation and assumptions in Assumption \ref{assum:bangbang}, we consider the periodic Schr\"odinger problem \eqref{eq:SpectralProblem2}. 
For every $k \in \mathcal B$,  and $j \in \mathbb N^{+}$, we have  $E_j(k) \to \lambda_j$ as  $V_+ \to \infty$, where 
$\lambda_j$ is the $j$th eigenvalue of the Dirichlet-Laplace operator on $\Omega_-$, satisfying 
\begin{subequations} \label{e:LDeig}
\begin{align}
- \Delta  u &= \lambda u  && \textrm{in }  \Omega_-   \\
u & = 0  && \textrm{on } \Omega_+ . 
\end{align}
\end{subequations}
\end{thm}

We'll denote by $\lambda_j(\Omega_-)$ the Laplace-Dirichlet eigenvalue of $\Omega_-\in \Theta$ and $E_j(\Omega_-,V_+,k)$ the $j$-th eignevalue with quasi-momentum $k \in \mathcal B$ of the twisted Schr\"odinger operator with potential given as in \eqref{e:BangBang}.  If we fix $\Omega_- \in \Theta$, by Theorem \ref{prop:ConvVinfty},  we have as $V_+ \to \infty$ that 
\[  E_j(\Omega_-, V_+,k)  \to  \lambda_j(\Omega_-) \]
and 
\[ G_m \to \frac{ \lambda_{m+1} - \lambda_m }{ (\lambda_{m} + \lambda_{m+1})/2} = 2 \frac{\lambda_{m+1}/\lambda_m - 1 }{\lambda_{m+1}/\lambda_m + 1} . \] 
It follows that, with Assumption  \ref{assum:bangbang}, in the high contrast limit ($V_+ = \infty$), the maximal Schr\"odinger gap problem is equivalent to the shape optimization problem,
\begin{equation}
\label{e:ShapeOpt}
\sup_{\Omega \in \Theta } \  f \left(  \frac{\lambda_{m+1}(\Omega)}{\lambda_{m}(\Omega)}  \right), 
\qquad \qquad \textrm{where} \quad
f(\alpha) = 2 \frac{\alpha -1 }{\alpha + 1}. 
\end{equation}
Since $f$ is an increasing function, this is equivalent to taking the supremum of $ \frac{\lambda_{m+1}(\Omega)}{\lambda_{m}(\Omega)}$ over $\Omega \in \Theta$. It is an open problem to show that the supremum  in \eqref{e:ShapeOpt} is attained for $m>1$; see open problem 16 in \cite{Henrot:2006fk}.

However, for the first gap ($m=1$), the shape optimization problem \eqref{e:ShapeOpt}  is well-defined. We recall the following result, conjectured by Payne, P\'olya, and Weinberger and proven by Ashbaugh and Benguria. 
\begin{thm}[\cite{ashbaugh1991proof}] \label{thm:AB}
Among all connected, open domains $\Omega \subset \mathbb R^2$, only a disk attains the maximum of the ratio of the second to first Dirichlet-Laplace eigenvalues, so we have the isoperimetric inequality
$$
\frac{\lambda_2(\Omega)}{\lambda_1(\Omega)} \leq \frac{\lambda_2(B)}{\lambda_1(B)} \approx 2.539. 
$$
\end{thm}
Since $f$ is strictly increasing, by the Ashbaugh-Benguria inequality (Theorem \ref{thm:AB}), we have that $G_m$ as $V_+ \to \infty$ is maximized only if $\Omega_-$ is a disk. The previous discussion is summarized in the following proposition. 

\begin{prop}  \label{prop:OptDisc}
Let $m=1$ and assume Assumption \ref{assum:bangbang} holds. 
For $V_+ = \infty$, any $\Omega_- $ such that the potential of the form \eqref{e:BangBang} maximizing $G_1[V]$ over $\ad(\Gamma,V_+)$ is a disk and the maximal value satisfies
\[G_{1,V_+,\Gamma}^\star \to  g, 
\qquad \qquad \textrm{where } g :=  2  \frac{j_{1,1}^2 - j_{0,1}^2}{ j_{1,1}^2 + j_{0,1}^2 } \approx 0.8697, \]
Here, $j_{k,\ell}$ is the $k$-th positive zero of the $\ell$-th Bessel function. 
Furthermore, for any finite $V_+ >0$, we have that 
$G_{1,V_+,\Gamma}^\star \leq  g $. 
\end{prop}
The last statement of Proposition \ref{prop:OptDisc} follows from the fact that $G_{1,V_+,\Gamma}^\star$ is non-decreasing in $V_+$.

The numerics for $m=1$ for both the square lattice (top panel of Figure \ref{f:2DsqA}) and triangular lattice (top panel of Figure \ref{f:2DtriA}) support Proposition \ref{prop:OptDisc} and further suggests that a periodic array of disks maximizes $G_1$ for \emph{finite} $V_+$. This conjecture has been made several times; see \cite{sigmund2008geometric}. We remark that the optimal configuration of disks is insensitive to the radii of the disks is analogous to the one-dimensional result in Lemma \ref{l:V+toinfty}.

\bigskip

We now consider the higher gaps $m > 1$ in the high contrast limit ($V_+ = \infty$). 
We consider a simpler problem than \eqref{e:ShapeOpt}, where the admissible set consists of exactly $m$ disjoint disks, 
\begin{equation}
\label{e:ShapeOpt2}
\sup_{\Omega \in \Theta_m} \  f \left(  \frac{\lambda_{m+1}(\Omega)}{\lambda_{m}(\Omega)}  \right). 
\end{equation}
Here 
\[ \Theta_m =  \{ \Omega \colon  \Omega  \textrm{ is the union of exactly $m$ disjoint, open disks, compactly contained in  $T\subset \mathbb R^2$} \} \subset \Theta.   \]

\revision{\begin{assum} \label{assum:bangbang2}
In addition to Assumption \ref{assum:bangbang}, assume any potential attaining the maximum of \eqref{e:opt} is of the bang-bang form \eqref{e:BangBang} where $\Omega_- \in \Theta_m$ and  $\Omega_+ = T\setminus \Omega_-$.
\end{assum}}

The following proposition is then the two-dimensional analogue of Proposition \ref{p:1DhighContrast}. 

\begin{prop}\label{prop:OptDiscMg1}
Let $m\geq 1$ be fixed. \revision{Assume Assumption  \ref{assum:bangbang2} holds.} The solution of \eqref{e:ShapeOpt2} is uniquely attained by the union of 
$m$ disks of equal radius. 
 \end{prop}
\begin{proof}
We denote the radii of
$m$ disks by $R_{1}$, $R_{2}$, $\ldots$ $R_{m}$ and without loss
of generality we can assume that $R_{1}\leq R_{2}\leq\cdots\leq R_{m}$.
The eigenvalues are then given by 
\[
\left\{ \frac{j_{k,\ell}^{2}}{R_{m}},\ \frac{j_{k,\ell}^{2}}{R_{m-1}},\ \ldots\ ,\ \frac{j_{k,\ell}^{2}}{R_{1}}\right\} ,\qquad k\in\mathbb{N}_{+}, \ \ell\in\mathbb{N},
\]
where $j_{k,\ell}$ is the $k$-th positive zero of the $\ell$-th
Bessel function.

If the $m$-th gap in the spectrum is between the eigenvalues  $\frac{j_{0,1}^{2}}{R_{1}}$ and $\frac{j_{1,1}^{2}}{R_{m}}$, then the gap-to-midgap ratio is
\begin{equation}
2\frac{\frac{j_{1,1}^{2}}{R_{m}}-\frac{j_{0,1}^{2}}{R_{1}}}{\frac{j_{1,1}^{2}}{R_{m}}+\frac{j_{0,1}^{2}}{R_{1}}}=2\frac{\frac{j_{1,1}^{2}R_{1}}{j_{0,1}^{2}R_{m}}-1}{\frac{j_{1,1}^{2}R_{1}}{j_{0,1}^{2}R_{m}}+1}=f\left(\frac{j_{1,1}^{2}R_{1}}{j_{0,1}^{2}R_{m}}\right).\label{e:val2}
\end{equation}
where $f(\alpha)=2\frac{\alpha-1}{\alpha+1}$. Thus (\ref{e:val2})
is maximized when $R_{1}=R_{m}$, which implies all radii are of the
same size and $G_{m}=f(\frac{j_{1,1}^{2}}{j_{0,1}^{2}})\approx0.8697$. 

If not, the $m$-th gap must lie in one of the intervals
\[
\left(\frac{j_{k,1}^{2}}{R_{m}},\ \frac{j_{k+1,1}^{2}}{R_{m}}\right),\qquad k=0,1,...\ ,m.
\]
It is known that $j_{k,\ell}/k$ decreases as $k$ increases for $0<\ell<\infty$ \cite{Lewis_1977}. 
We then have 
\[
\frac{j_{k+1,1}}{j_{k,1}}<\frac{k+1}{k},\qquad k=0,1,...\ ,m.
\]
When $k\ge2$, $\frac{j_{k+1,1}}{j_{k,1}}<1.5<\frac{j_{1,1}}{j_{0,1}} \approx \frac{3.8317}{2.4048}$, so the optimal gap cannot be in any of these intervals. 
The optimal gap also can't be in the $k=1$ interval, since $\frac{j_{2,1}}{j_{1,1}}\approx\frac{5.1356}{3.8317}\approx1.3403<\frac{j_{1,1}}{j_{0,1}}$. 
It follows that the optimal gap is in the interval  $\left(\frac{j_{0,1}^{2}}{R_{m}},\ \frac{j_{1,1}^{2}}{R_{m}}\right)$. Since the $m$-th gap must lie above $m$ eigenvalues, the only possibility is that the optimal gap is in the interval $\left(\frac{j_{0,1}^{2}}{R_{1}},\ \frac{j_{1,1}^{2}}{R_{m}}\right)$, as considered above. 
\end{proof}

From Proposition \ref{prop:OptDiscMg1} and the preceding discussion, we have the following corollary.

\begin{cor} \label{prop:OptDiscMg2}
Let $m\geq 1$ be fixed. Assume Assumption  \ref{assum:bangbang2} holds. For $V_+ = \infty$, 
any $\Omega_- $ such that the potential of the form \eqref{e:BangBang} maximizing $G_m[V]$ over $\ad(\Gamma,V_+)$
is the disjoint union of $m$ equal-radius disks and the maximal value satisfies
\[G_{m,V_+,\Gamma}^\star \to g. \] 
Furthermore, for any finite $V_+ >0$, we have that $G_{m,V_+,\Gamma}^\star \leq  g $. 
\end{cor}
Note that in Corollary \ref{prop:OptDiscMg2},  the optimal value $G_{m,V_+,\Gamma}^\star$ does not depend on $m$ or $\Gamma$.

\begin{figure} [t!]
\begin{center}
\includegraphics[width=.40\textwidth]{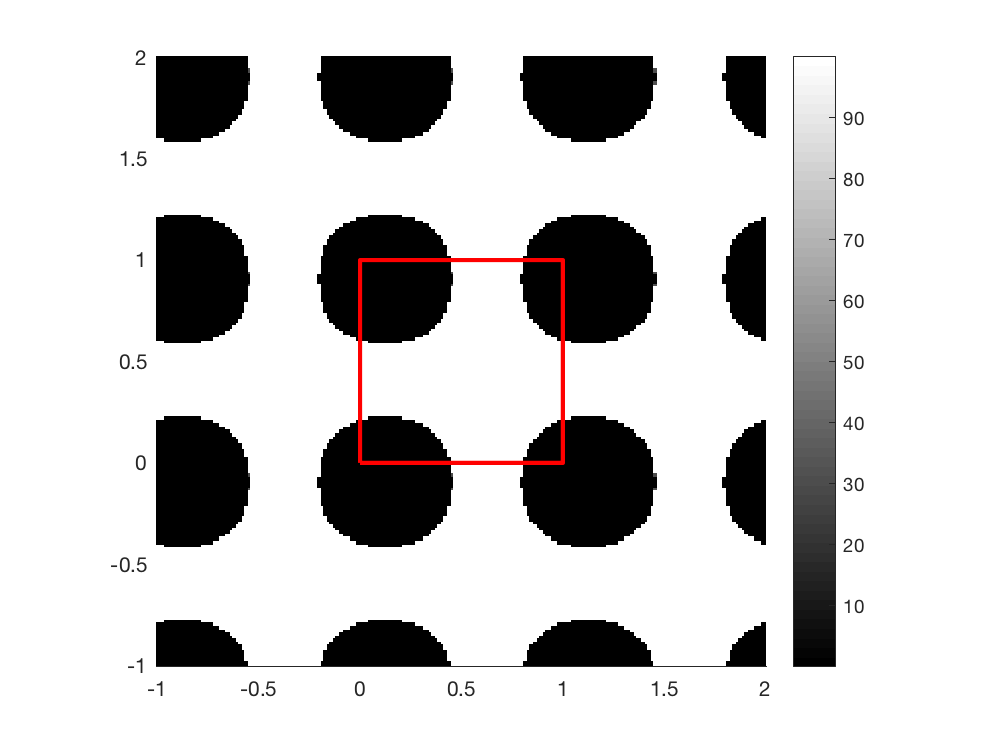}
\includegraphics[width=.40\textwidth]{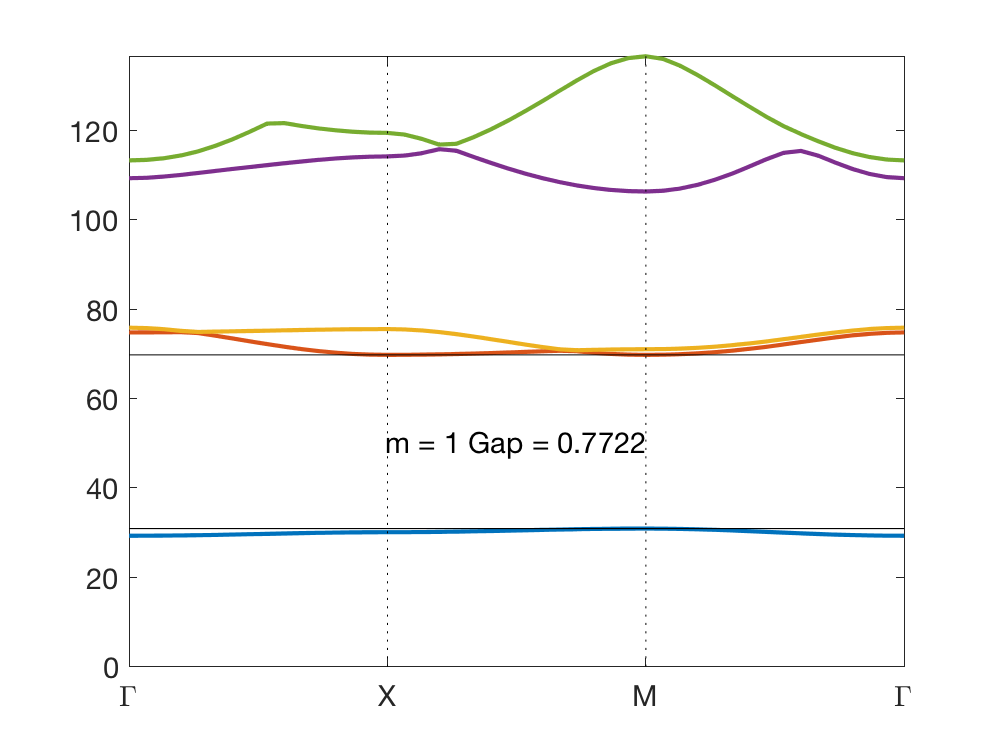}
\includegraphics[width=.40\textwidth]{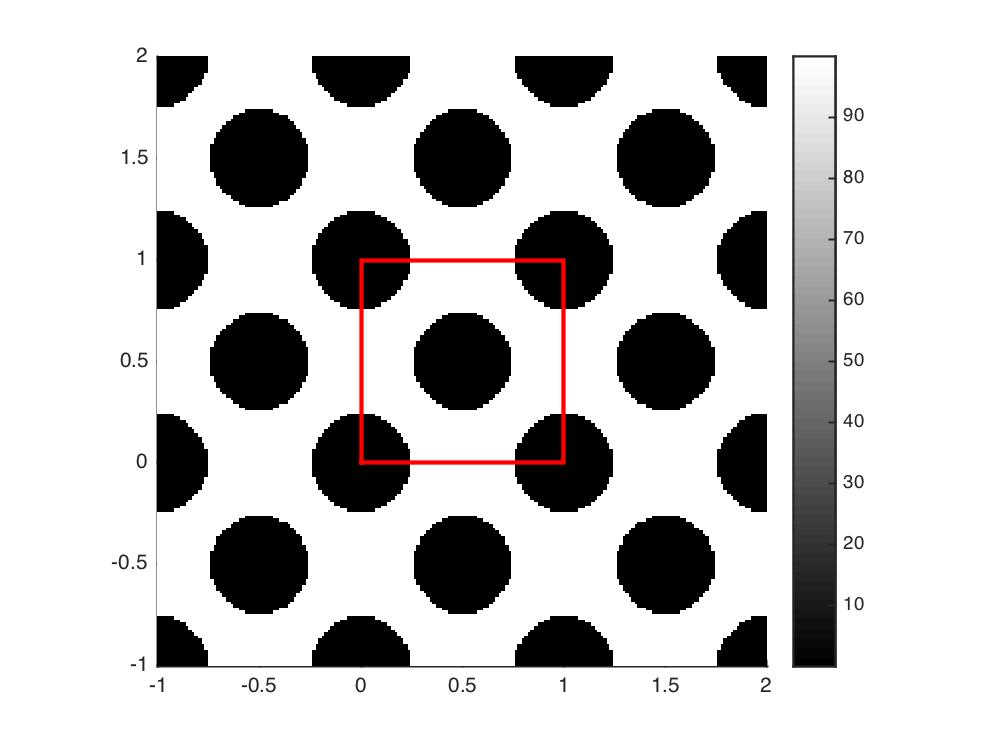}
\includegraphics[width=.40\textwidth]{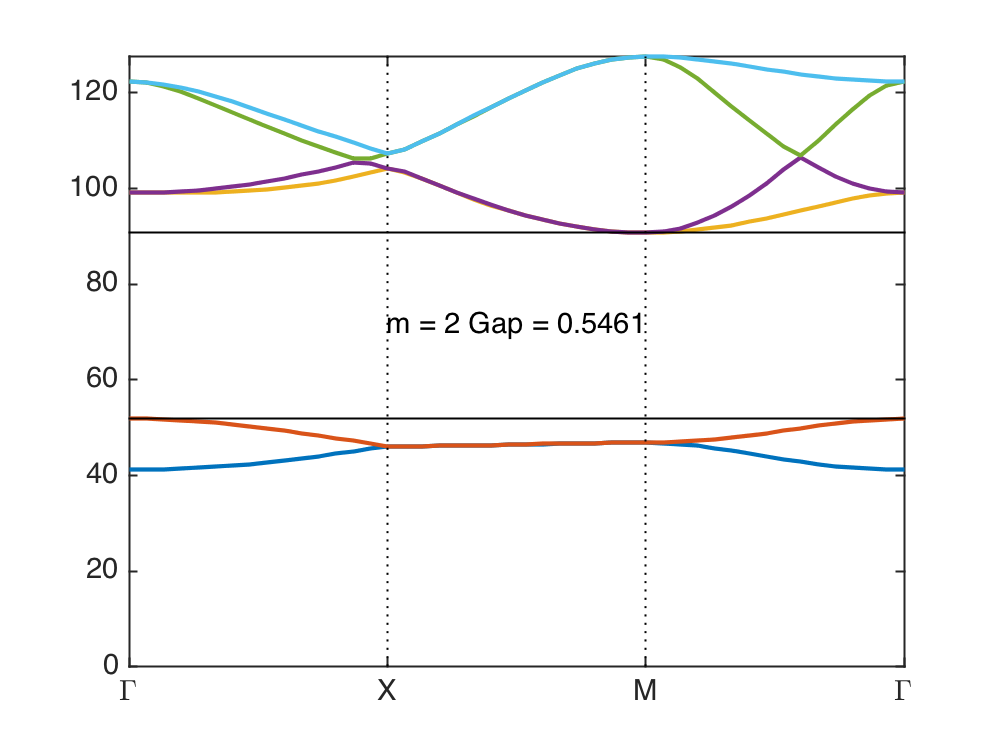}
\includegraphics[width=.40\textwidth]{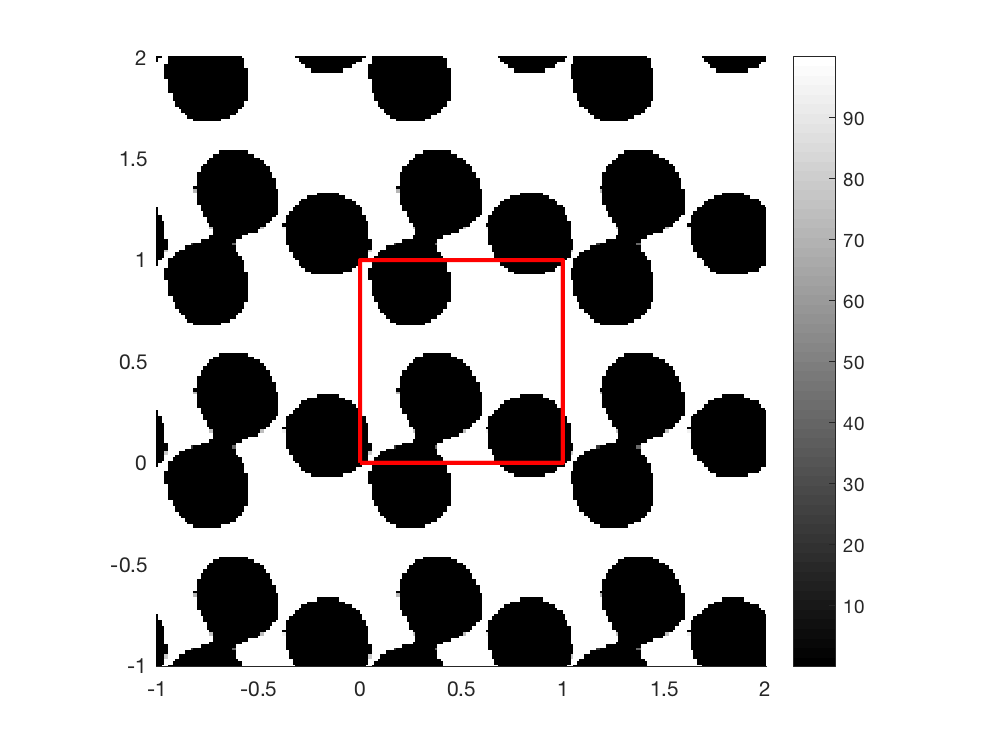}
\includegraphics[width=.40\textwidth]{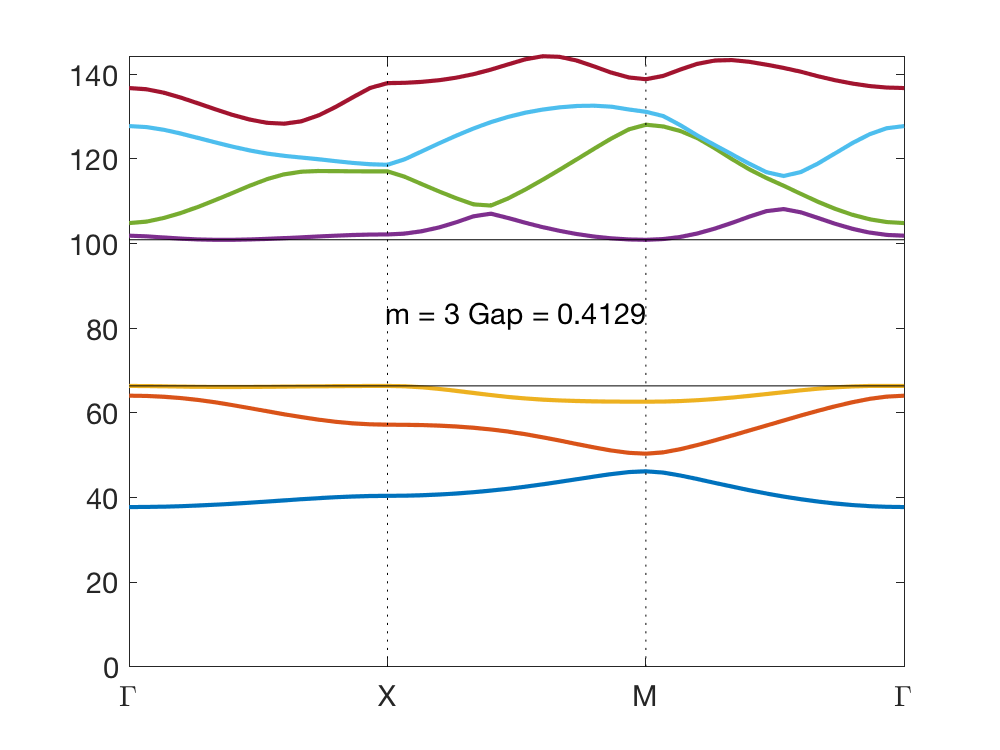}
\includegraphics[width=.40\textwidth]{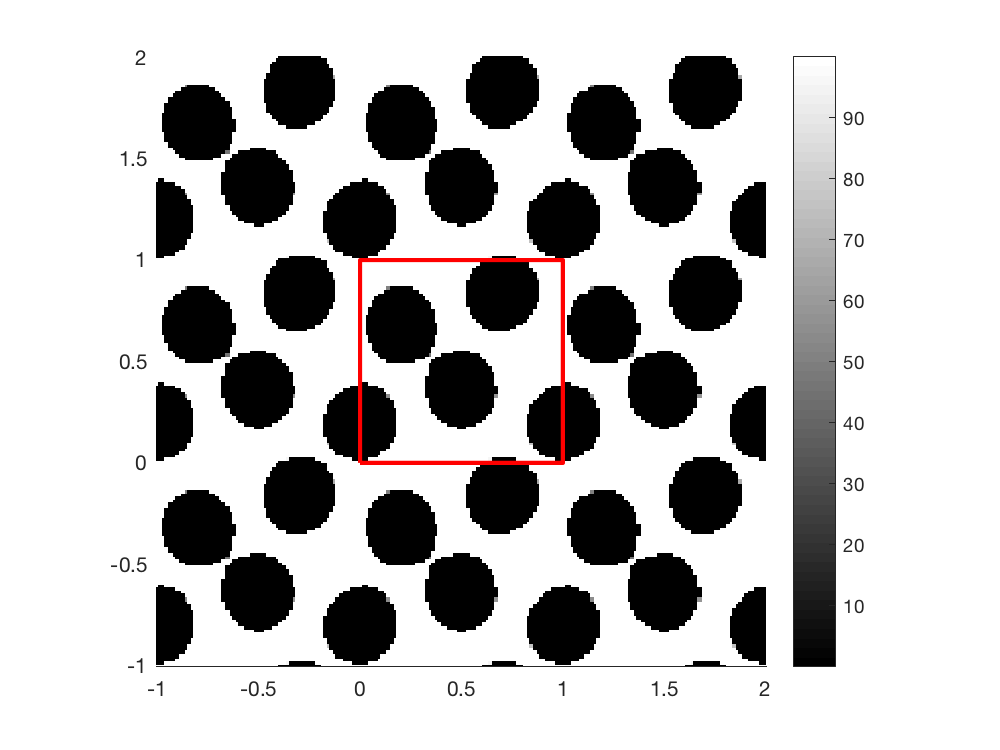}
\includegraphics[width=.40\textwidth]{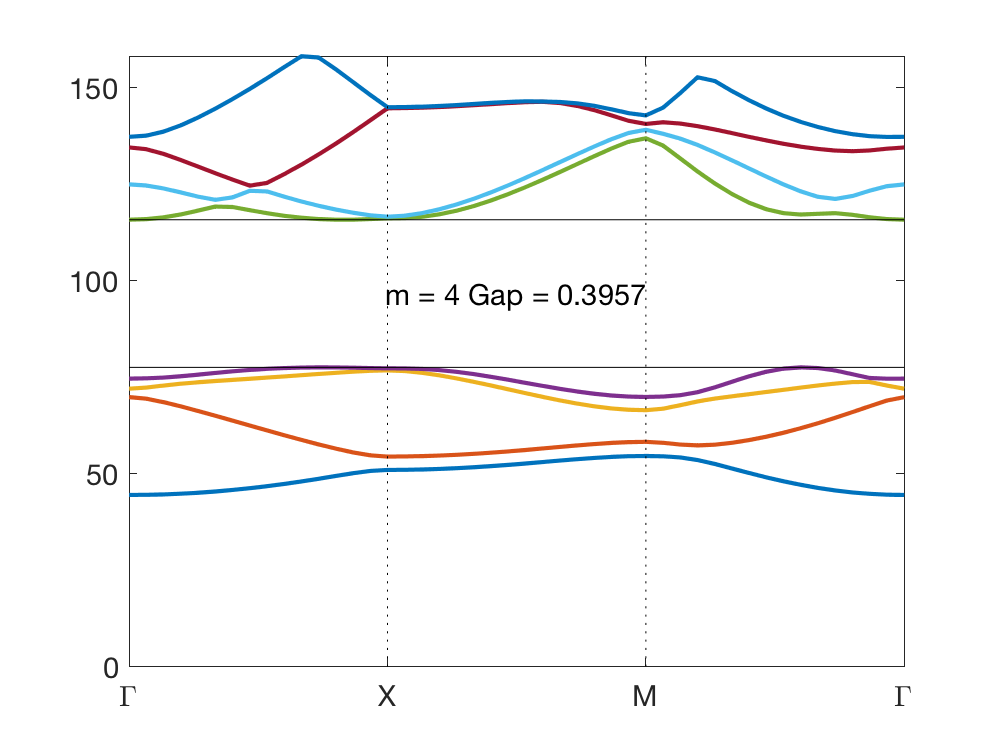}
\caption{For the square lattice and $V_+ = 100$, 
{\bf (left)} the  potential  maximizing $G_m$  and 
{\bf (right)} corresponding dispersion relation  over the irreducible Brillouin zone are plotted for $m=1,2,3,4$. 
}
\label{f:2DsqA}
\end{center}
\end{figure}

\begin{figure} [t!]
\begin{center}
\includegraphics[width=.40\textwidth]{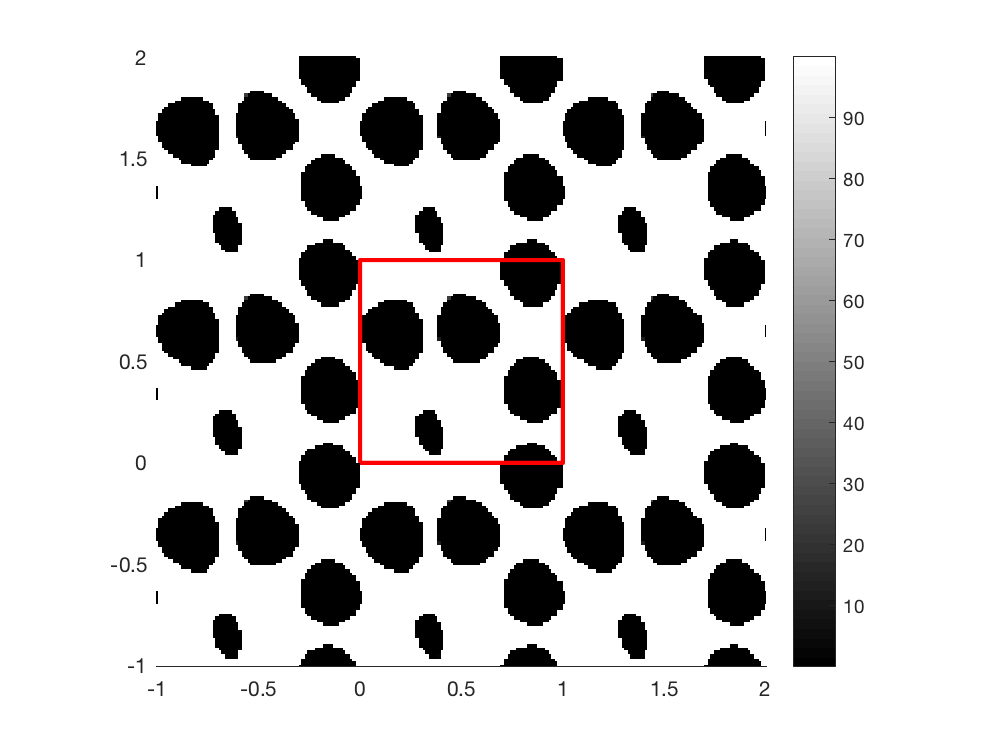}
\includegraphics[width=.40\textwidth]{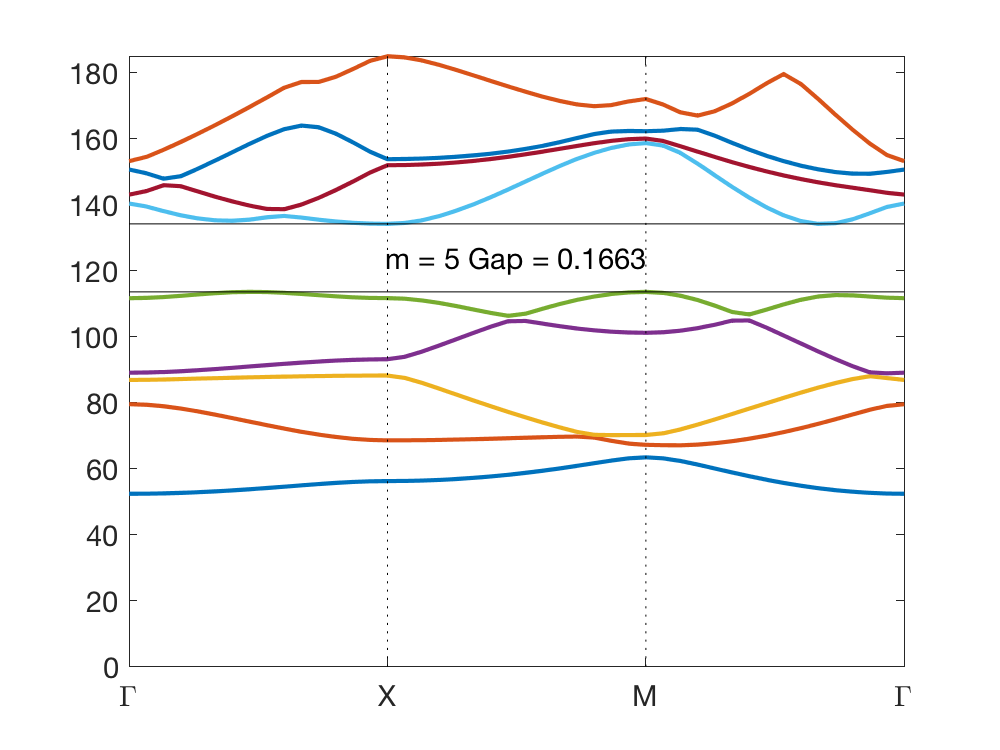}
\includegraphics[width=.40\textwidth]{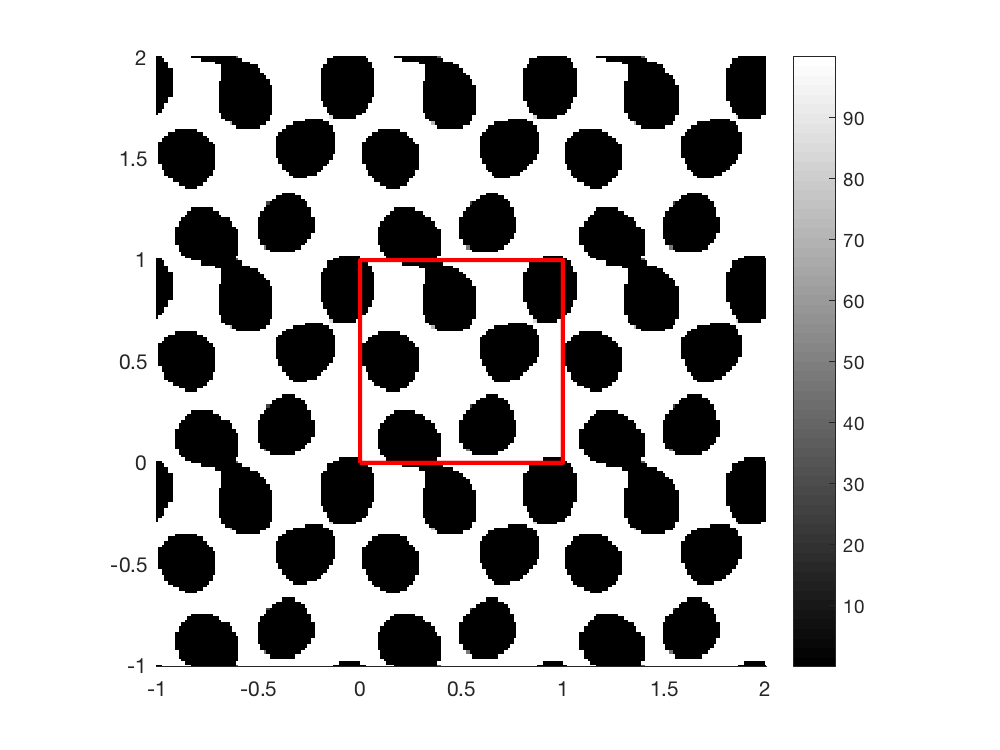}
\includegraphics[width=.40\textwidth]{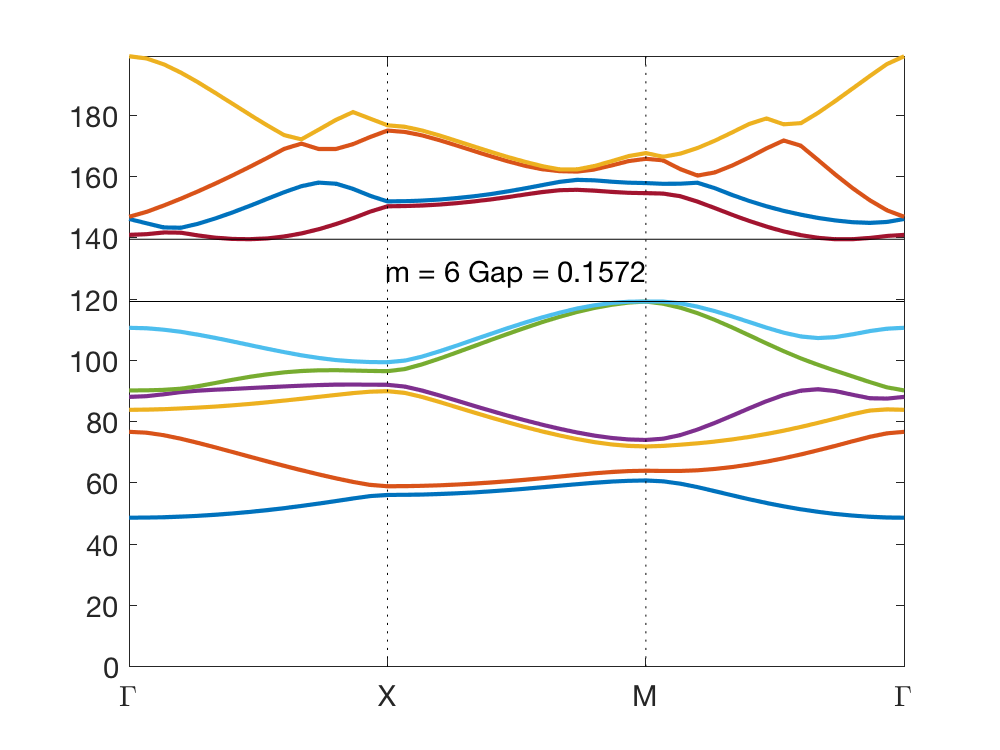}
\includegraphics[width=.40\textwidth]{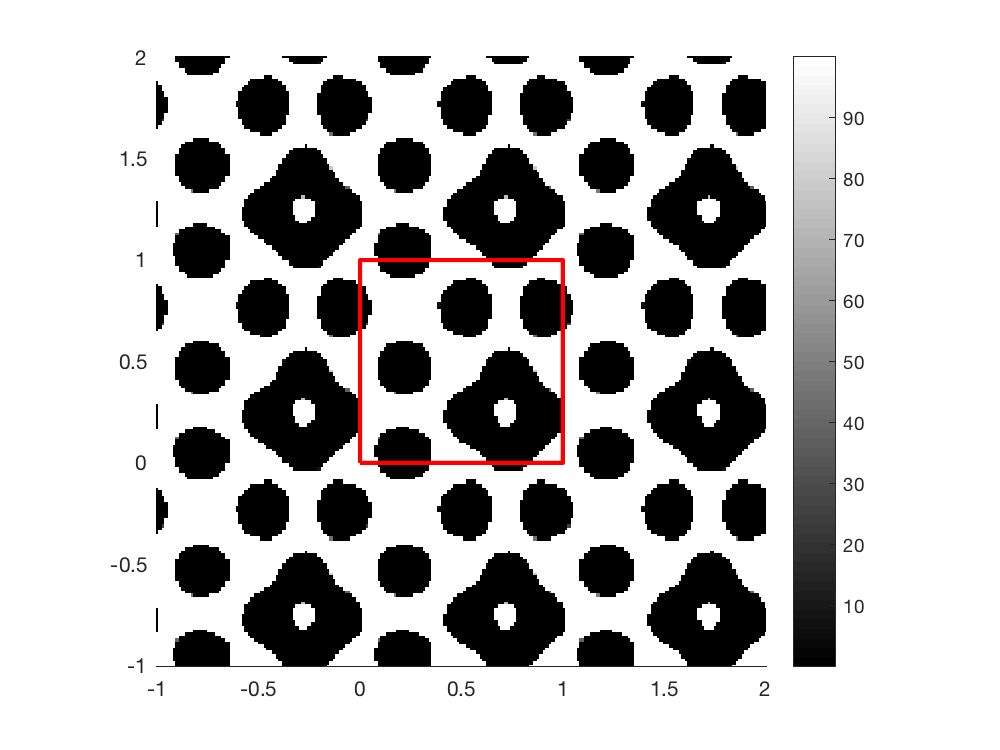}
\includegraphics[width=.40\textwidth]{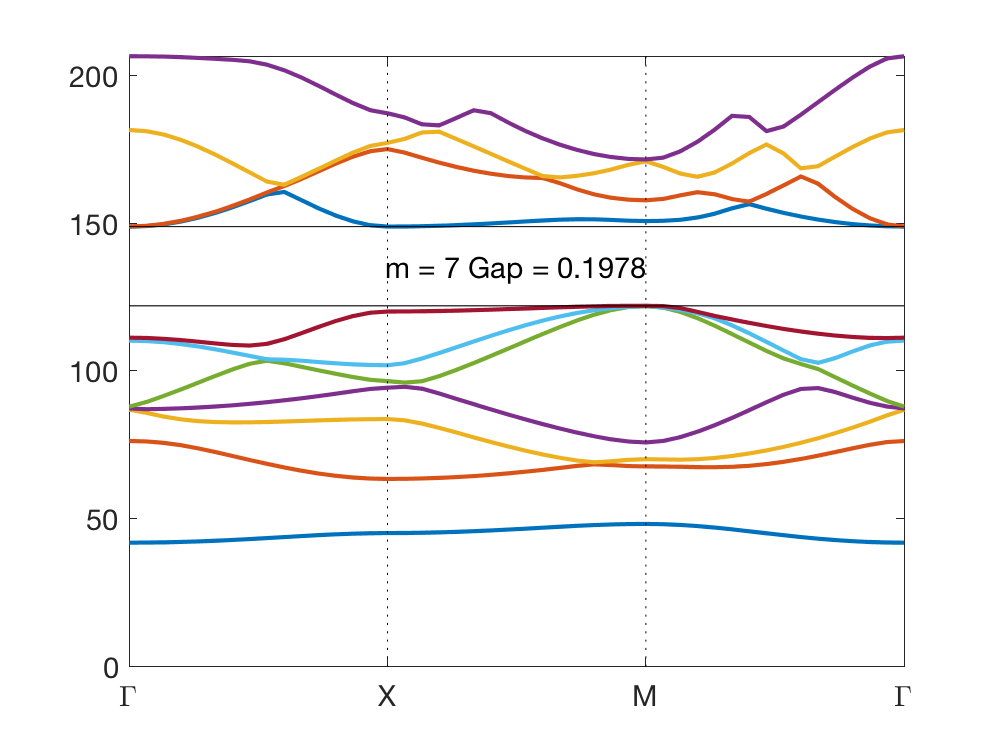}
\includegraphics[width=.40\textwidth]{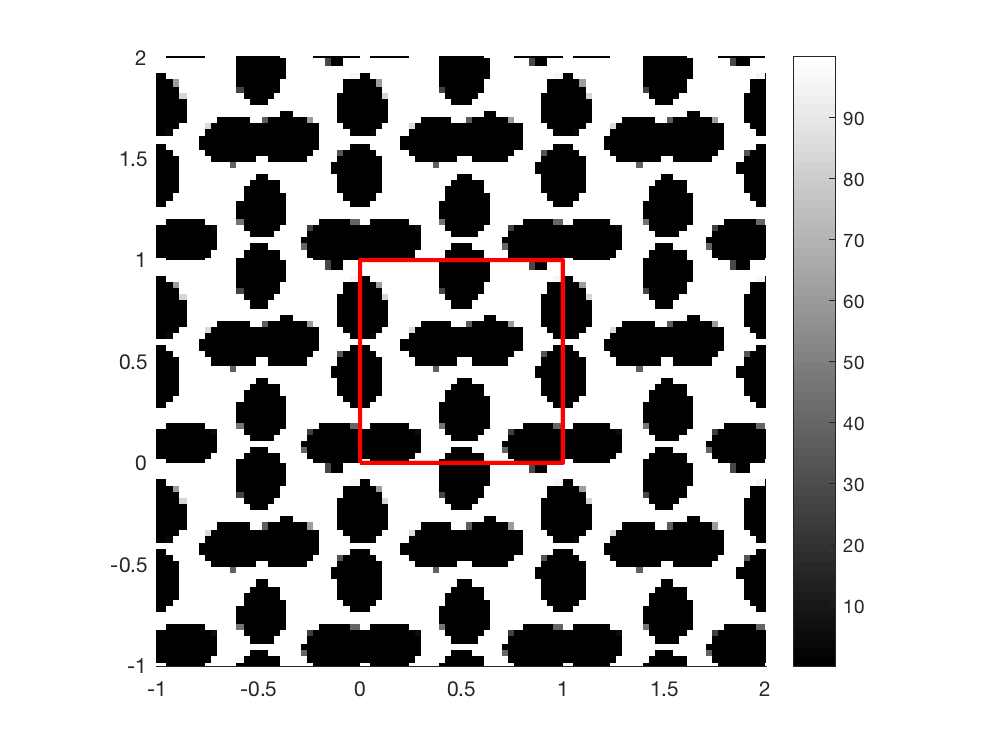}
\includegraphics[width=.40\textwidth]{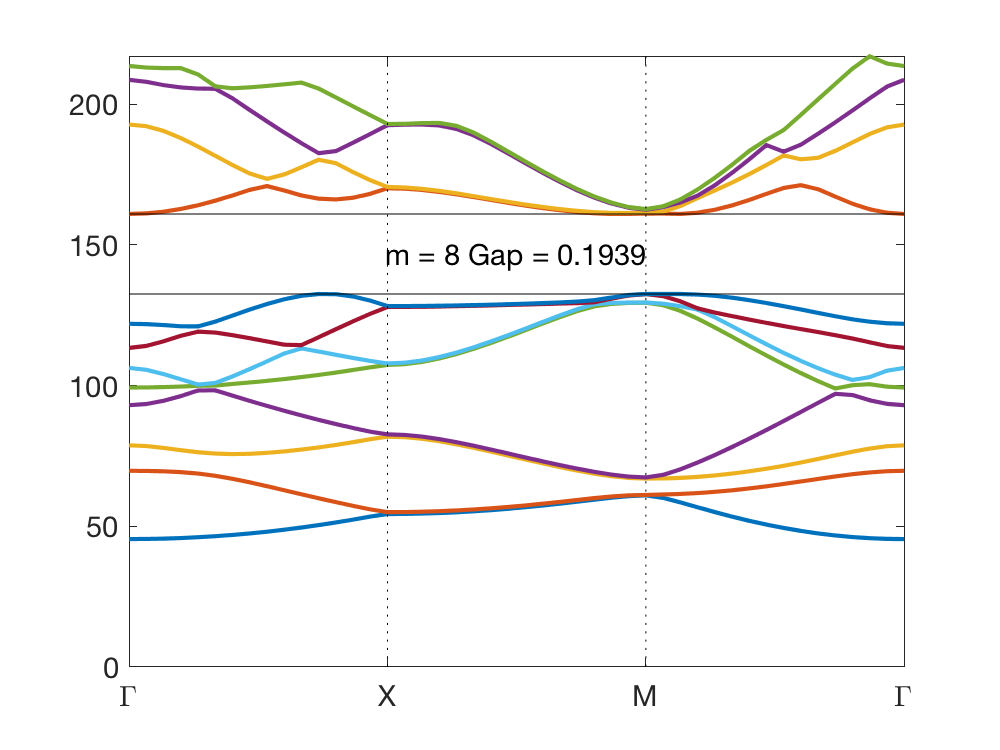}
\caption{For the square lattice and $V_+ = 100$, 
{\bf (left)} the  potential  maximizing $G_m$  and 
{\bf (right)} corresponding dispersion relation over the irreducible Brillouin zone are plotted for  $m=5,6,7,8$.
}
\label{f:2DsqB}
\end{center}
\end{figure}

\begin{figure} [t!]
\begin{center}
\includegraphics[width=.40\textwidth]{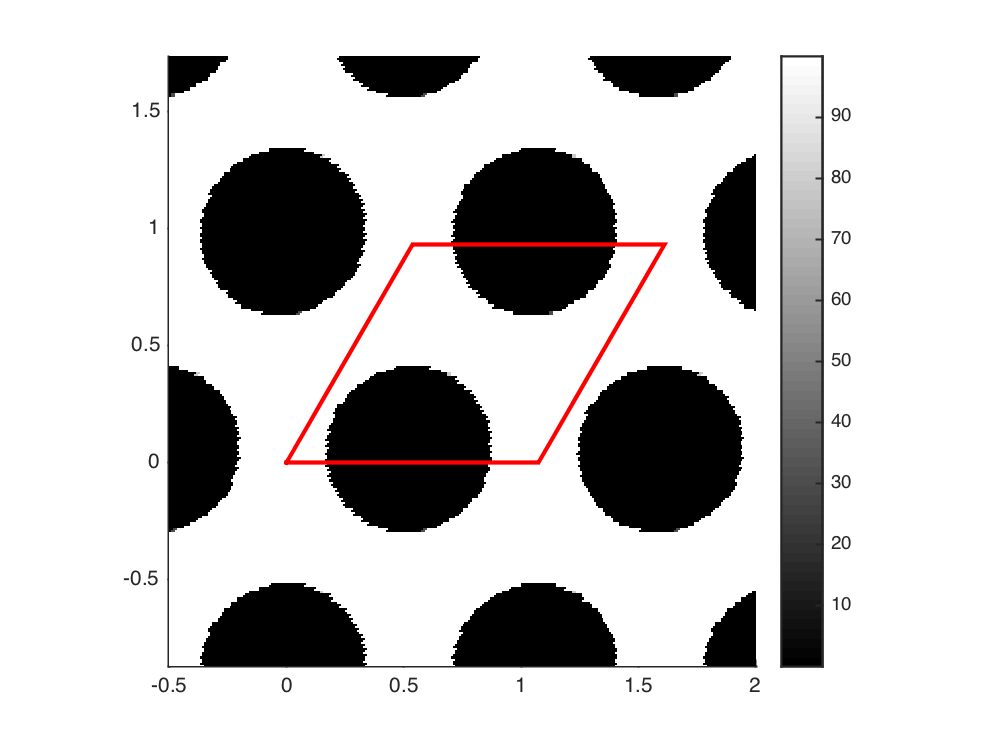}
\includegraphics[width=.40\textwidth]{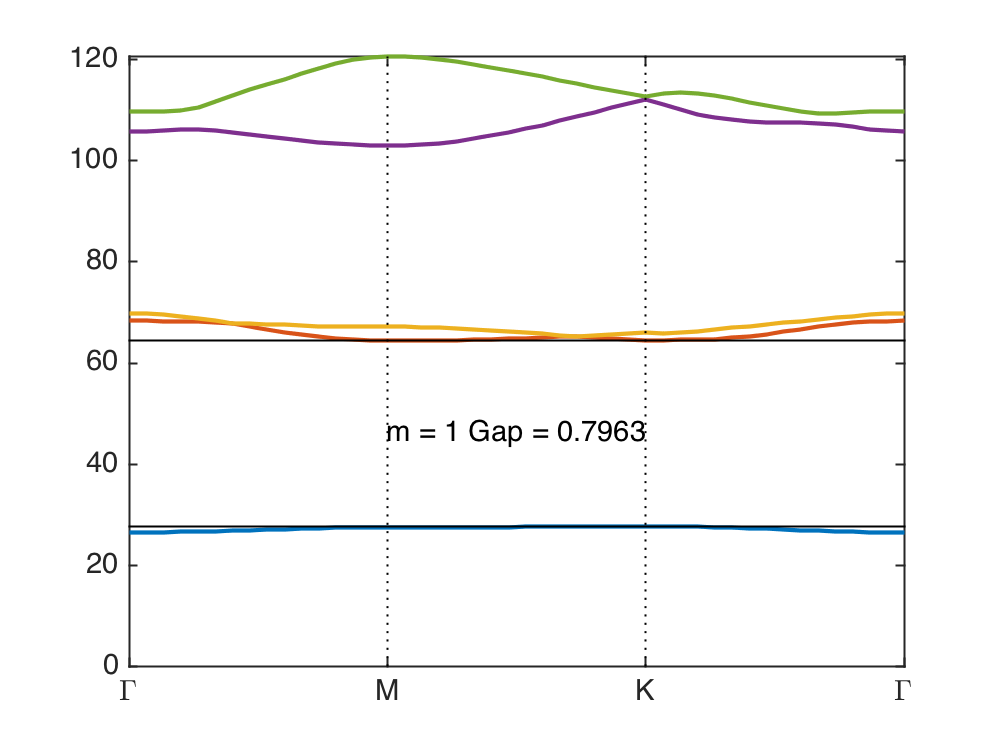}
\includegraphics[width=.40\textwidth]{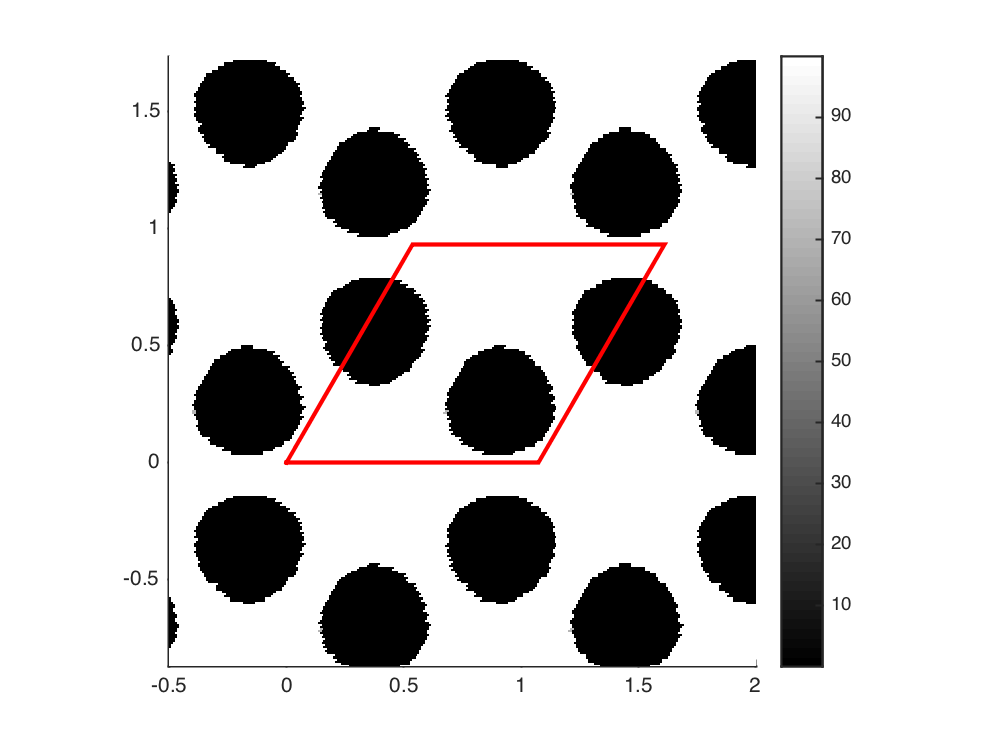}
\includegraphics[width=.40\textwidth]{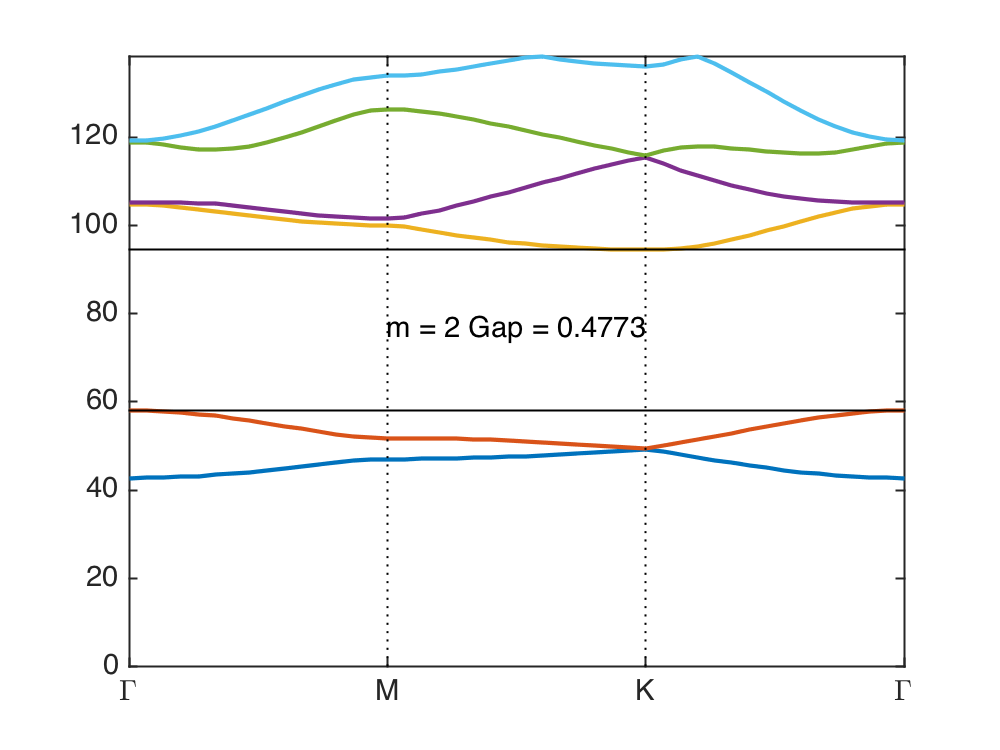}
\includegraphics[width=.40\textwidth]{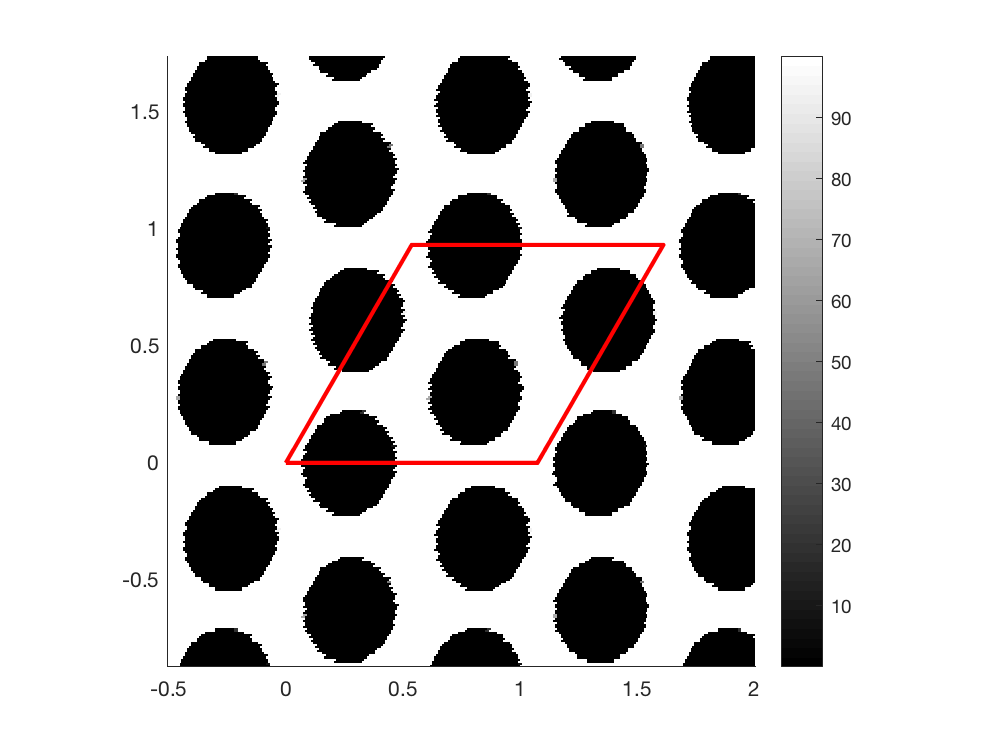}
\includegraphics[width=.40\textwidth]{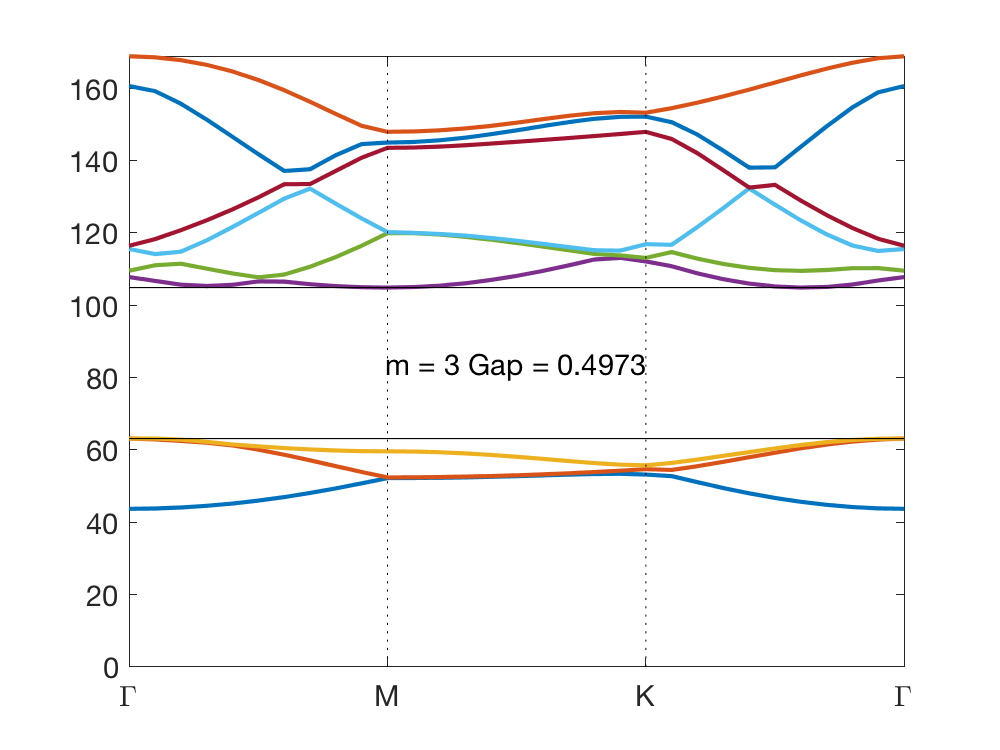}
\includegraphics[width=.40\textwidth]{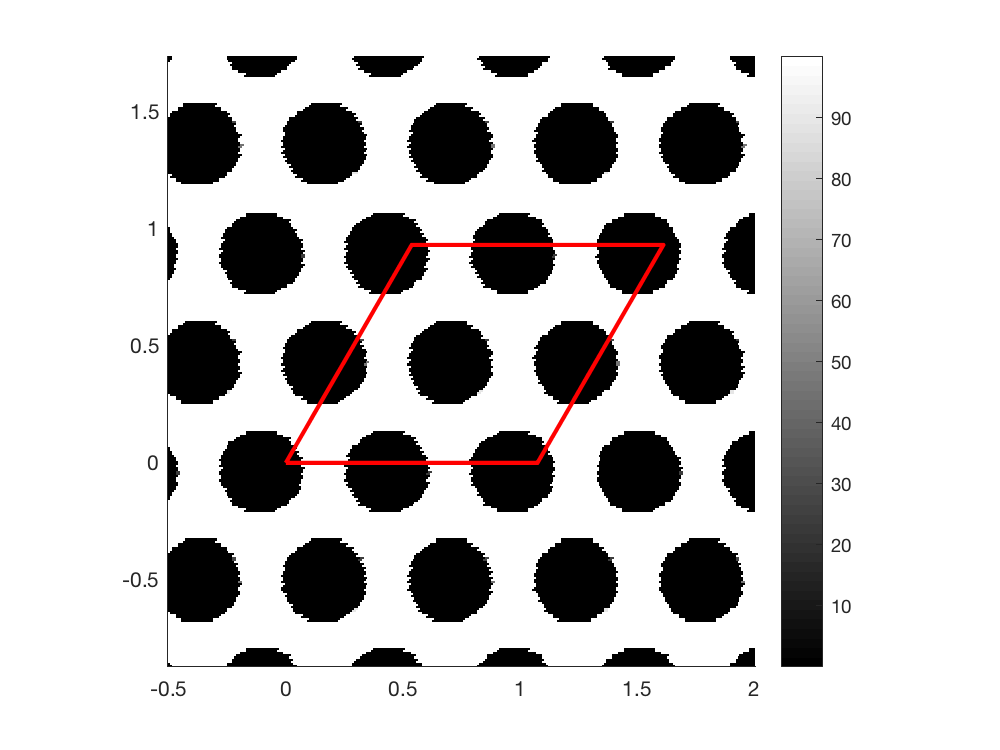}
\includegraphics[width=.40\textwidth]{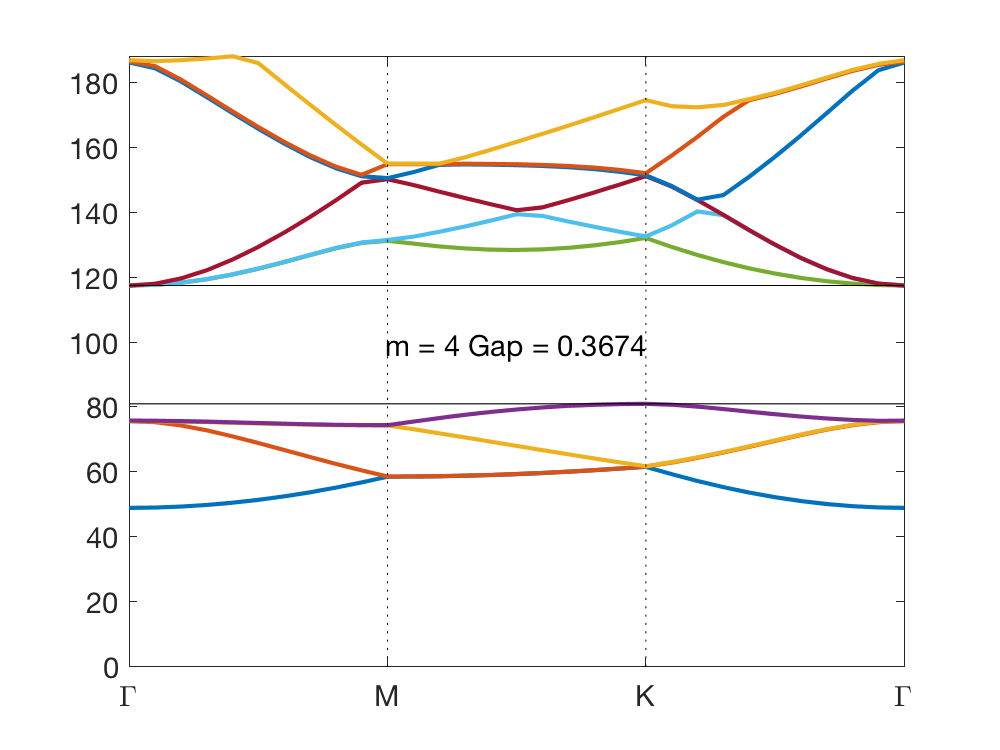}

\caption{For the triangular lattice and $V_+ = 100$, 
{\bf (left)} the  potential  maximizing $G_m$  and 
{\bf (right)} corresponding dispersion relation over the irreducible Brillouin zone are plotted for  $m=1,2,3,4$. }
\label{f:2DtriA}
\end{center}
\end{figure}

\begin{figure} [t!]
\begin{center}
\includegraphics[width=.40\textwidth]{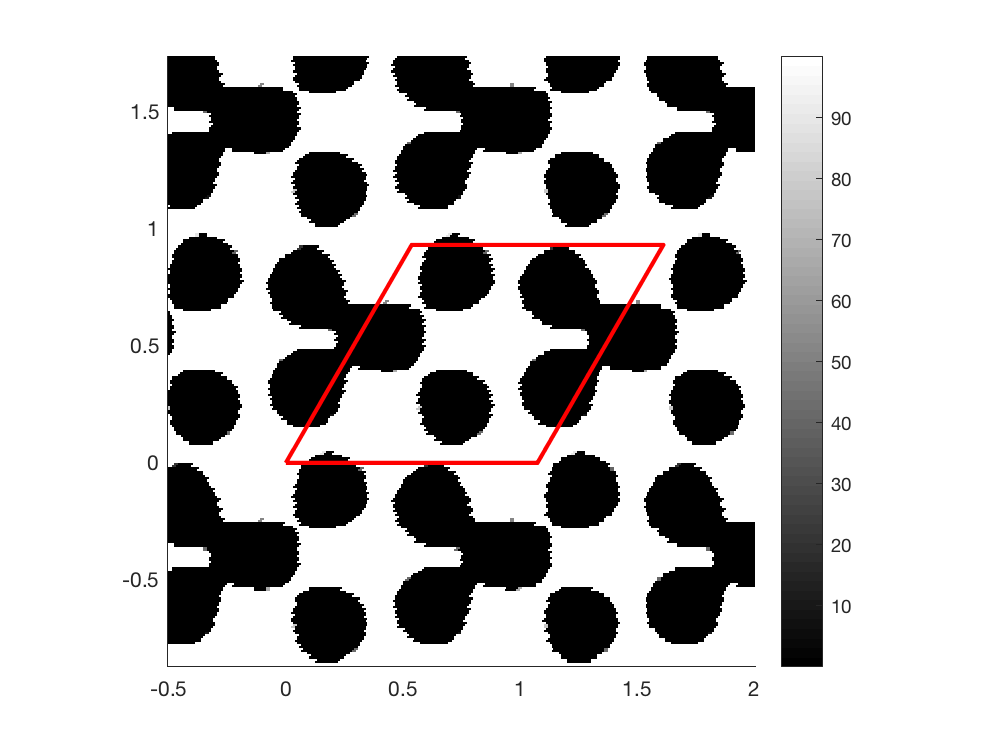}
\includegraphics[width=.40\textwidth]{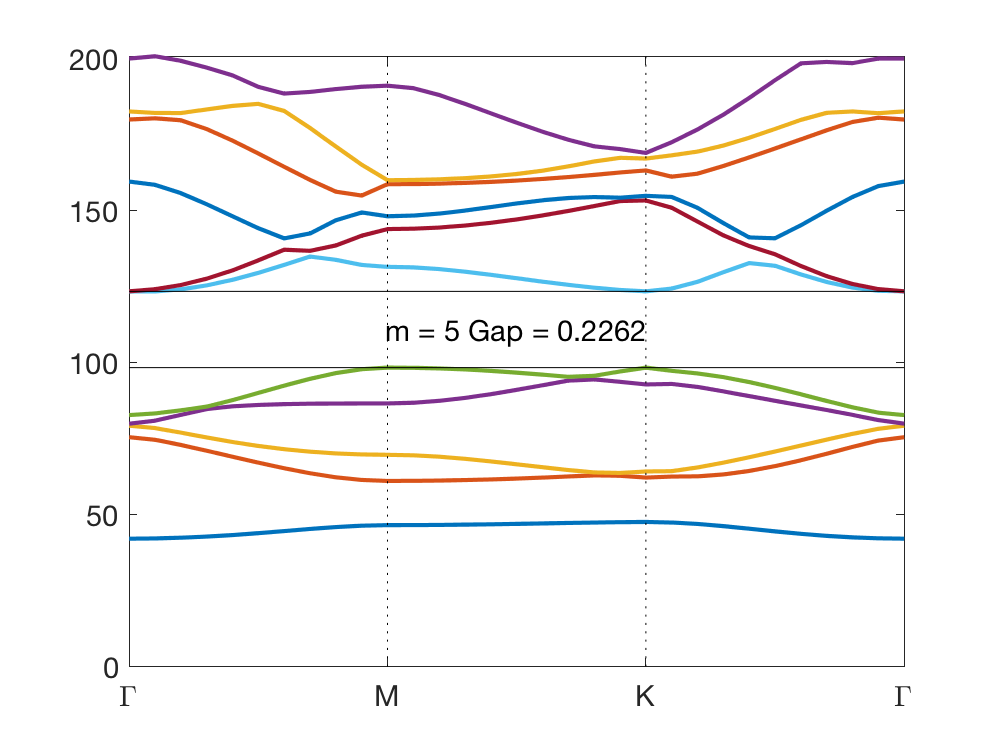}
\includegraphics[width=.40\textwidth]{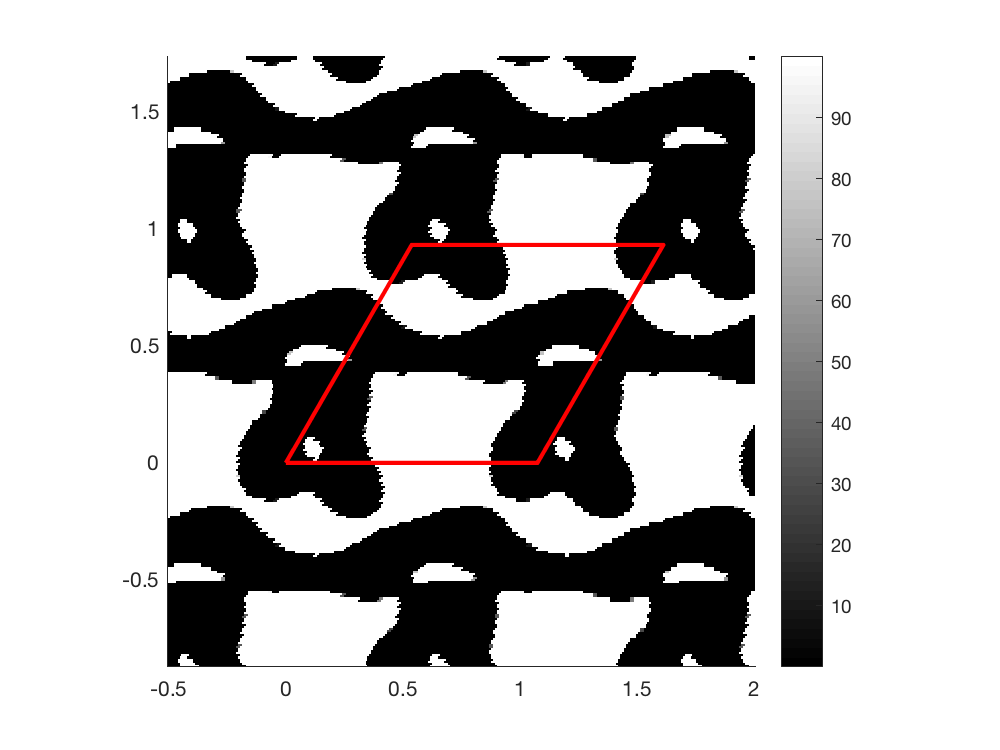}
\includegraphics[width=.40\textwidth]{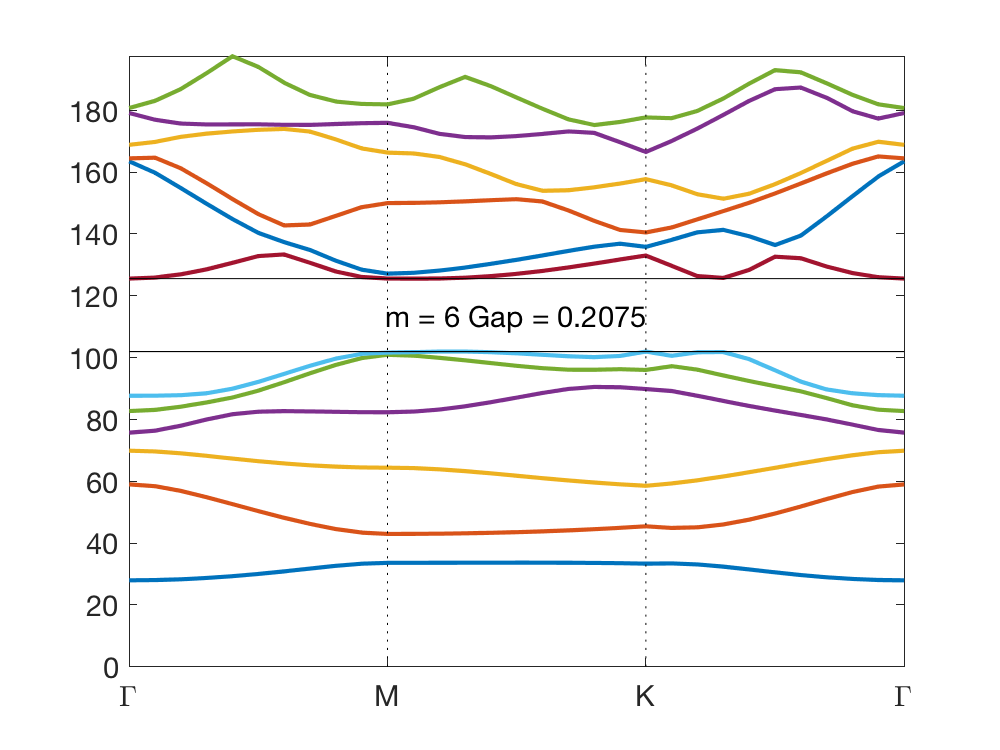}
\includegraphics[width=.40\textwidth]{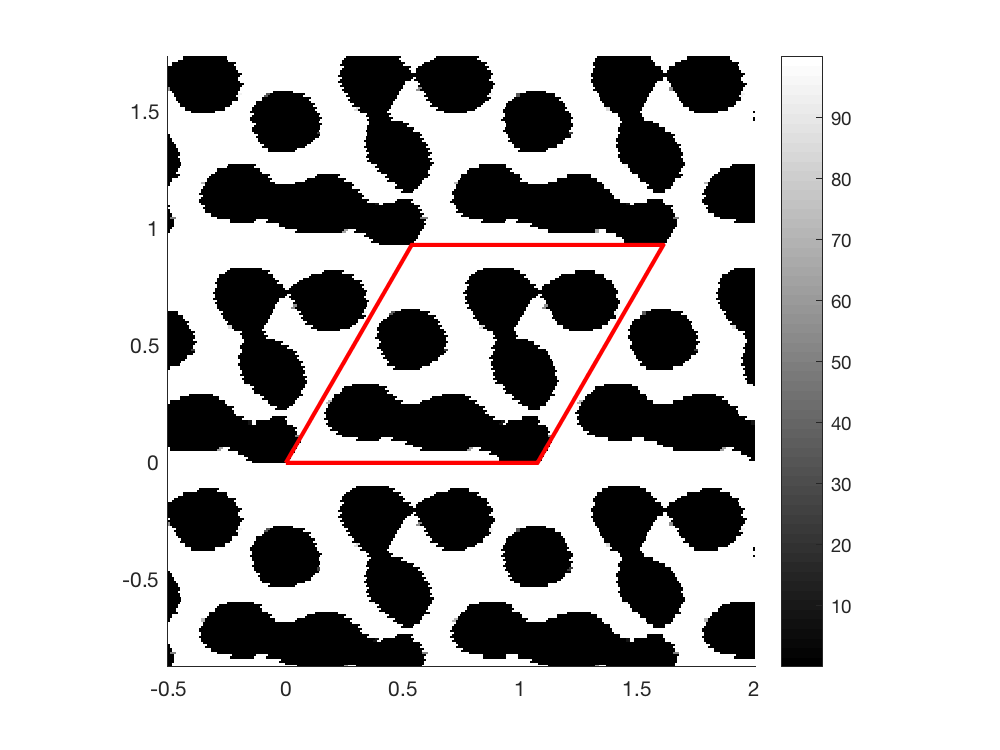}
\includegraphics[width=.40\textwidth]{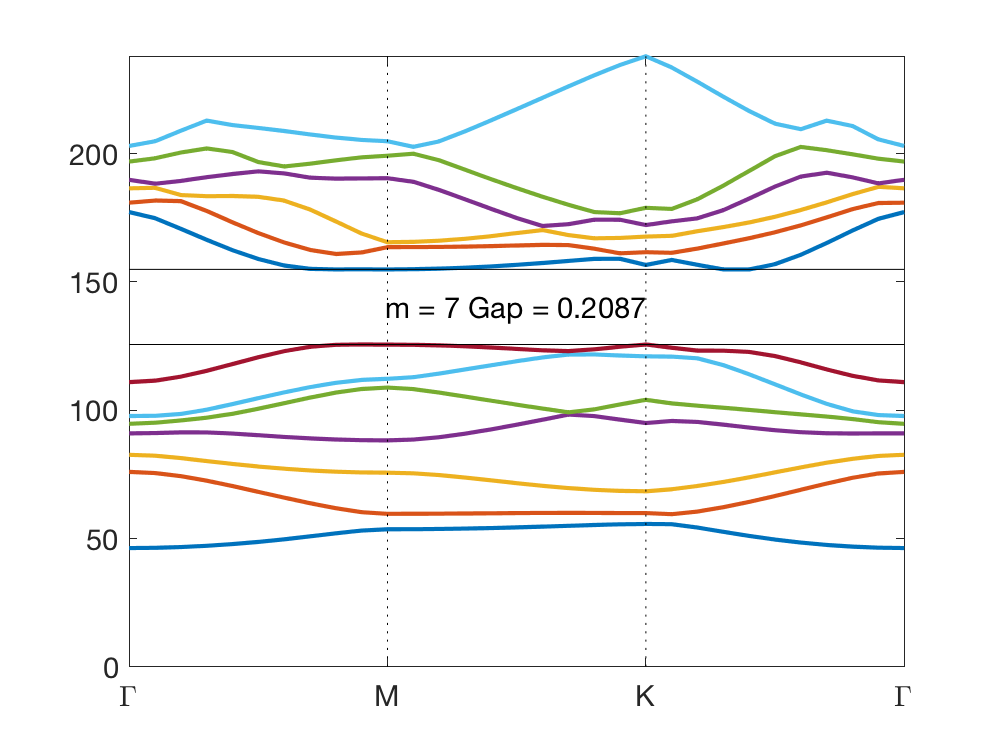}
\includegraphics[width=.40\textwidth]{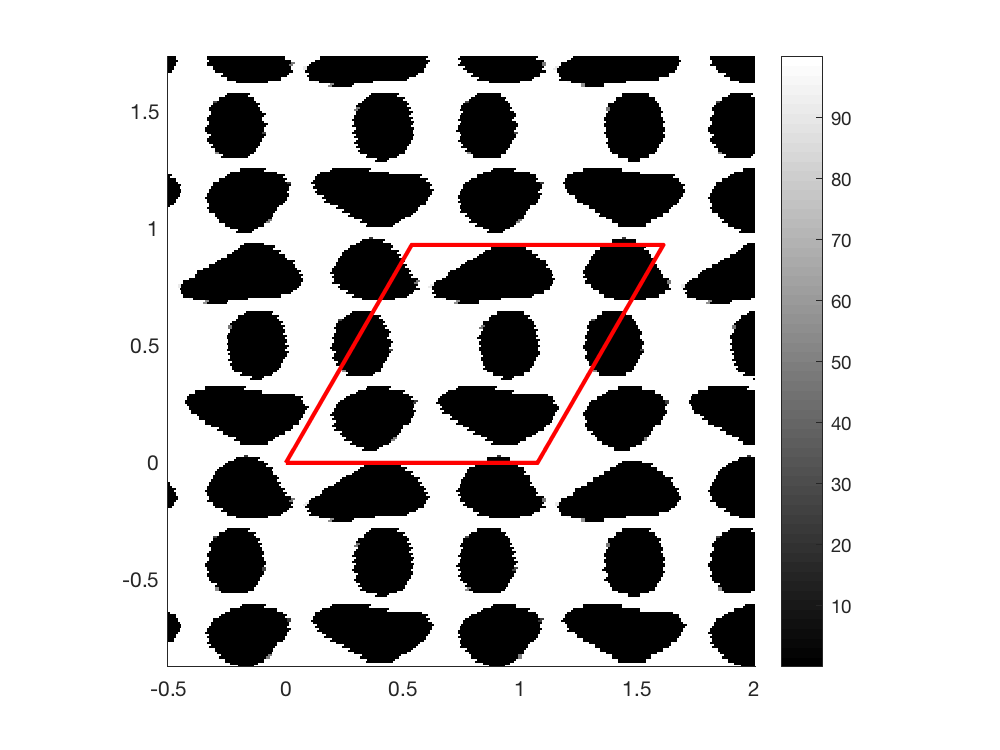}
\includegraphics[width=.40\textwidth]{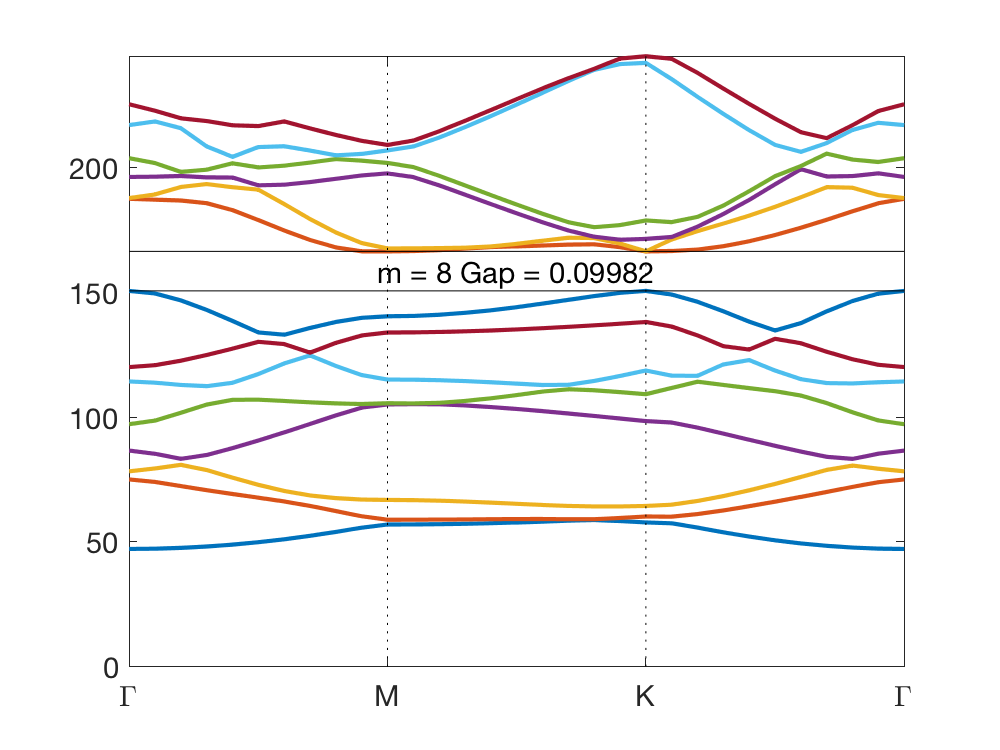}

\caption{For the triangular lattice and $V_+ = 100$, 
{\bf (left)} the  potential  maximizing $G_m$  and 
{\bf (right)} corresponding dispersion relation over the irreducible Brillouin zone are plotted for  $m=5,6,7,8$. }
\label{f:2DtriB}
\end{center}
\end{figure}

\begin{figure} 
\begin{center}
\begin{minipage}{.1\textwidth} \vspace{-5cm} $V_+ = 40$ \end{minipage}
\includegraphics[width=.40\textwidth]{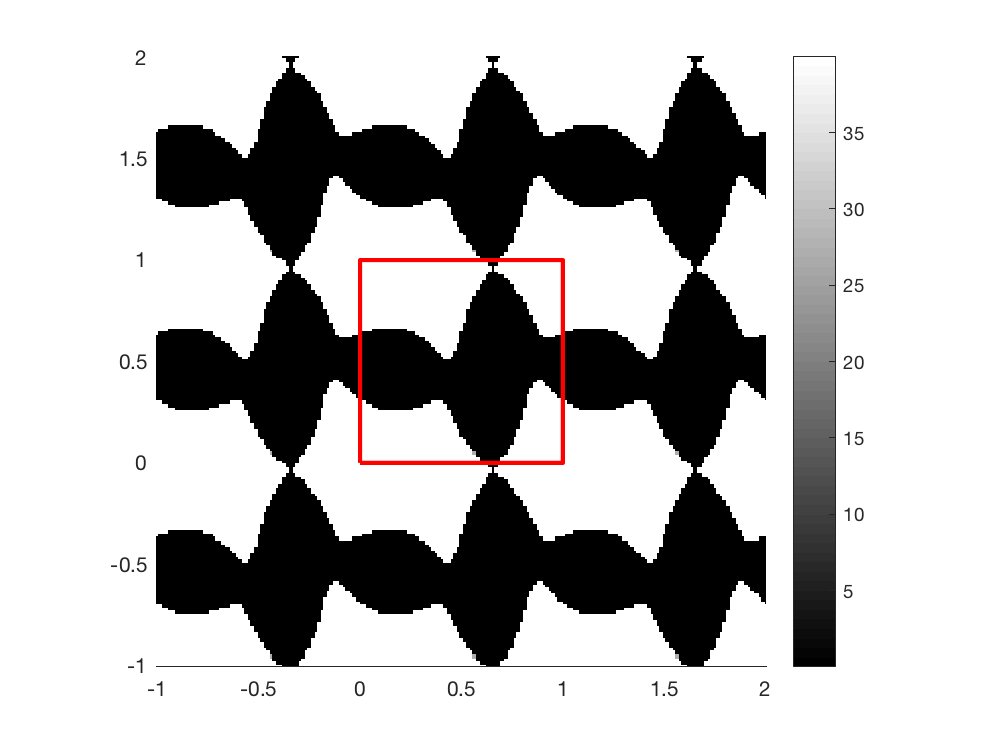}
\includegraphics[width=.40\textwidth]{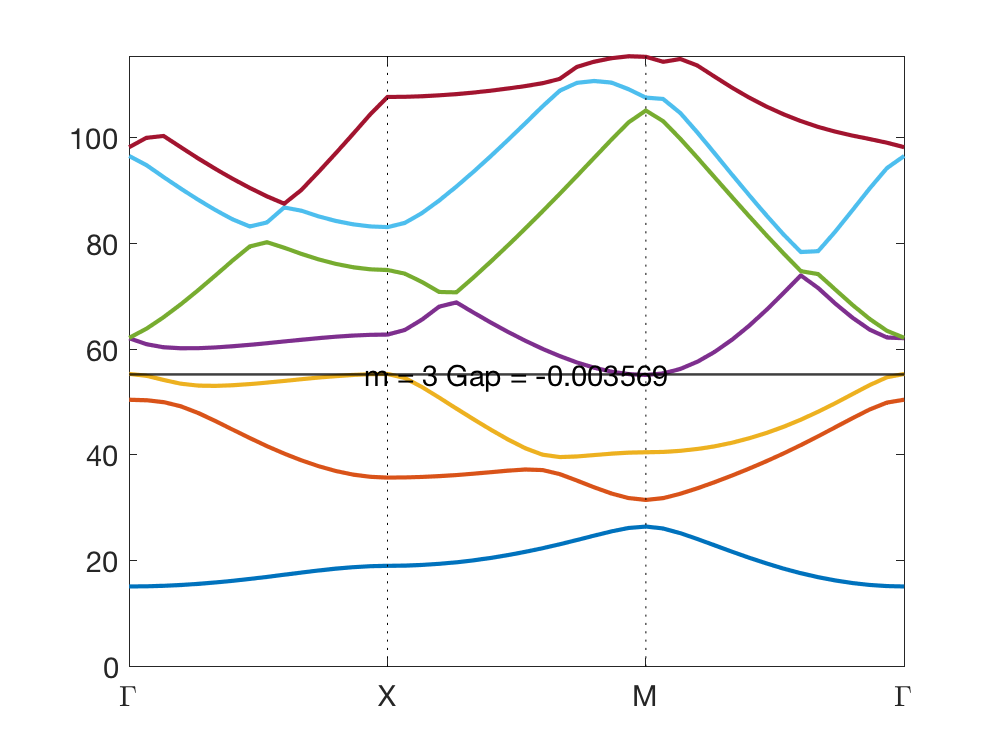}\\
\begin{minipage}{.1\textwidth} \vspace{-5cm} $V_+ = 60$ \end{minipage}
\includegraphics[width=.40\textwidth]{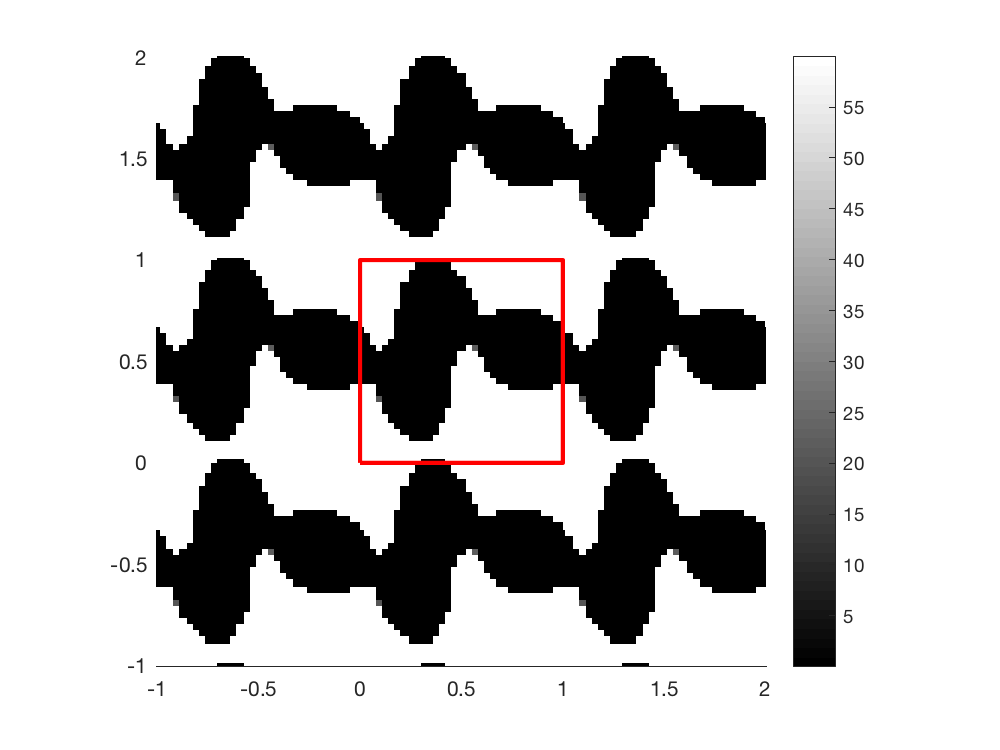}
\includegraphics[width=.40\textwidth]{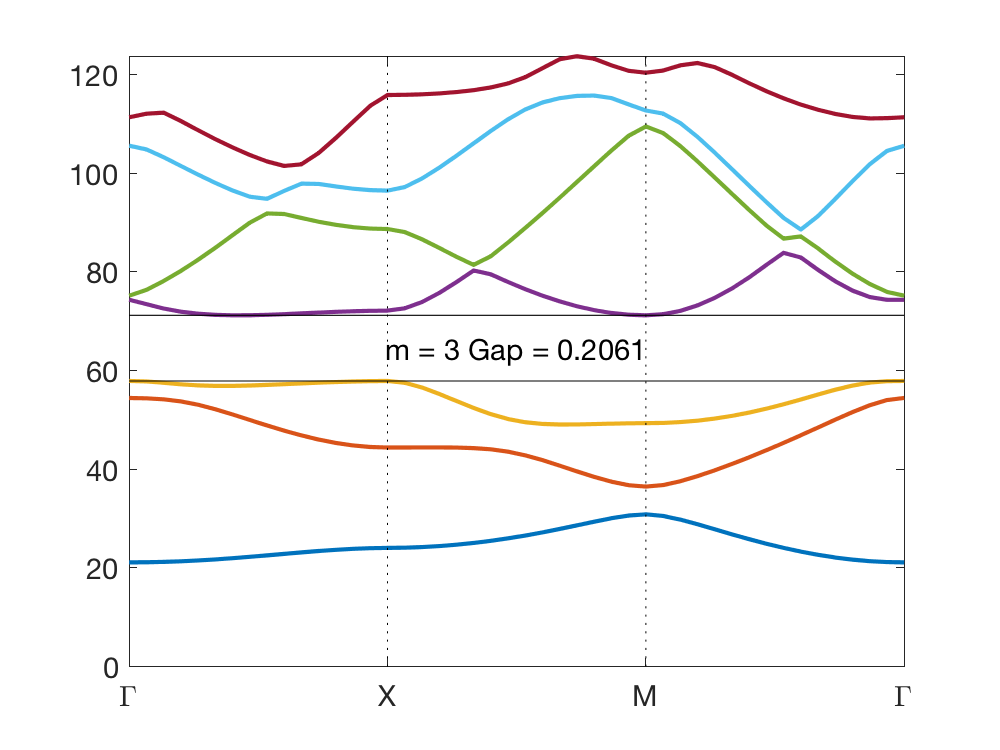}\\
\begin{minipage}{.1\textwidth} \vspace{-5cm} $V_+ = 100$ \end{minipage}
\includegraphics[width=.40\textwidth]{a_0_b_1_N64_V100_band_3_p4129_V_iter_019.png}
\includegraphics[width=.40\textwidth]{a_0_b_1_N64_V100_band_3_p4129_Disp_iter_019.png}
\begin{minipage}{.1\textwidth} \vspace{-5cm} $V_+ = 200$ \end{minipage}
\includegraphics[width=.40\textwidth]{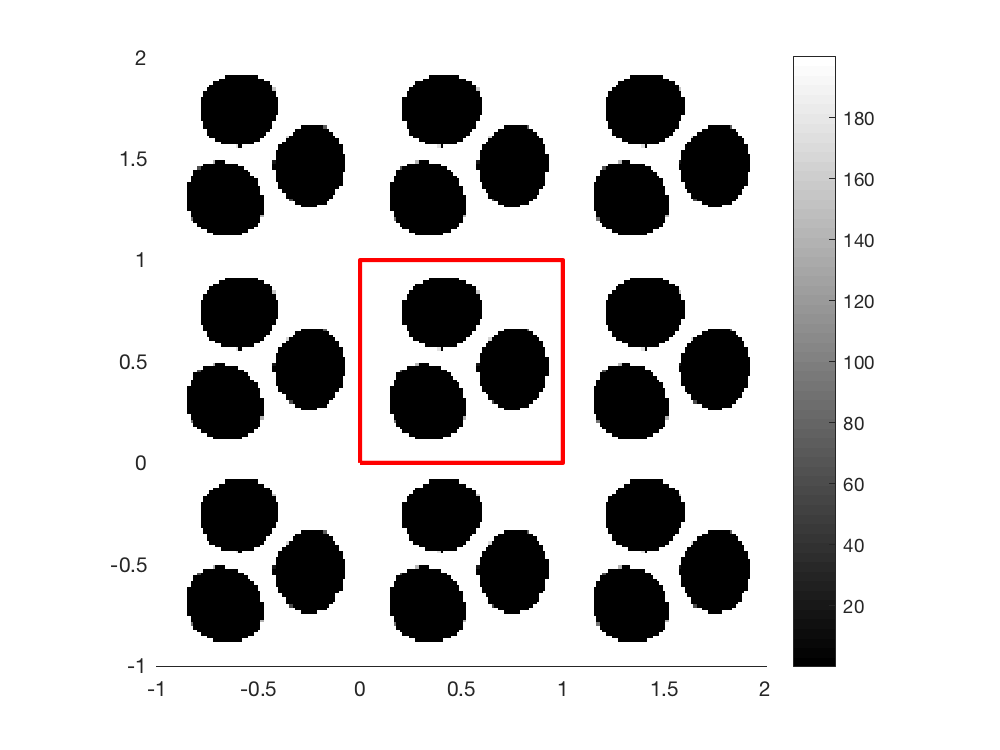}
\includegraphics[width=.40\textwidth]{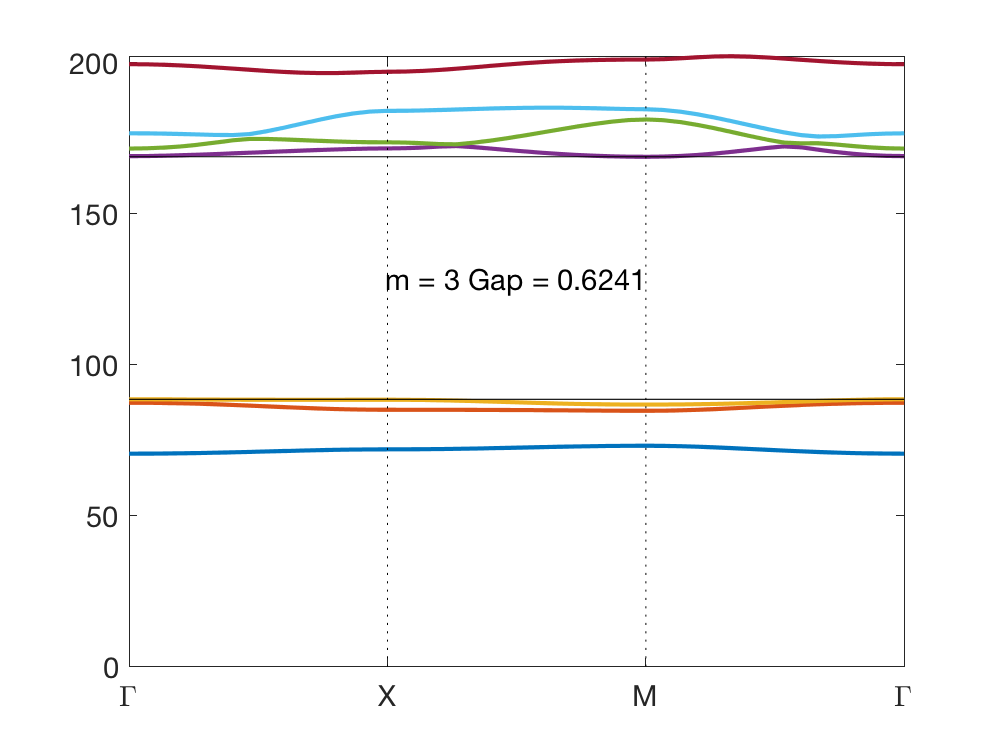}
\caption{For the square lattice and $m=3$, 
{\bf (left)} the  potential  maximizing $G_m$  and 
{\bf (right)} corresponding dispersion relation over the irreducible Brillouin zone are plotted for different values of $V_+$. As $V_+$ increases, the number of components per periodic cell where $V = 0$ changes. 
}
\label{f:varyV+}
\end{center}
\end{figure}

\begin{figure} [t!]
\begin{center}
$m=1$ \hspace{4cm} $m=2$  \hspace{4cm} $m=3$ \\
\includegraphics[width=.32\textwidth]{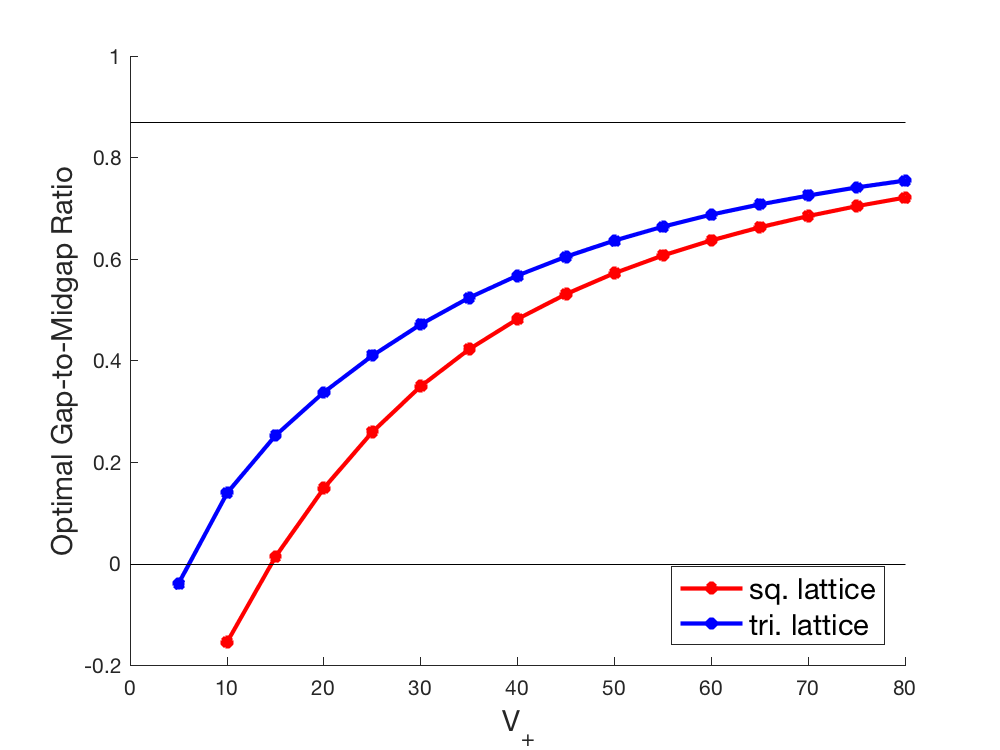}
\includegraphics[width=.32\textwidth]{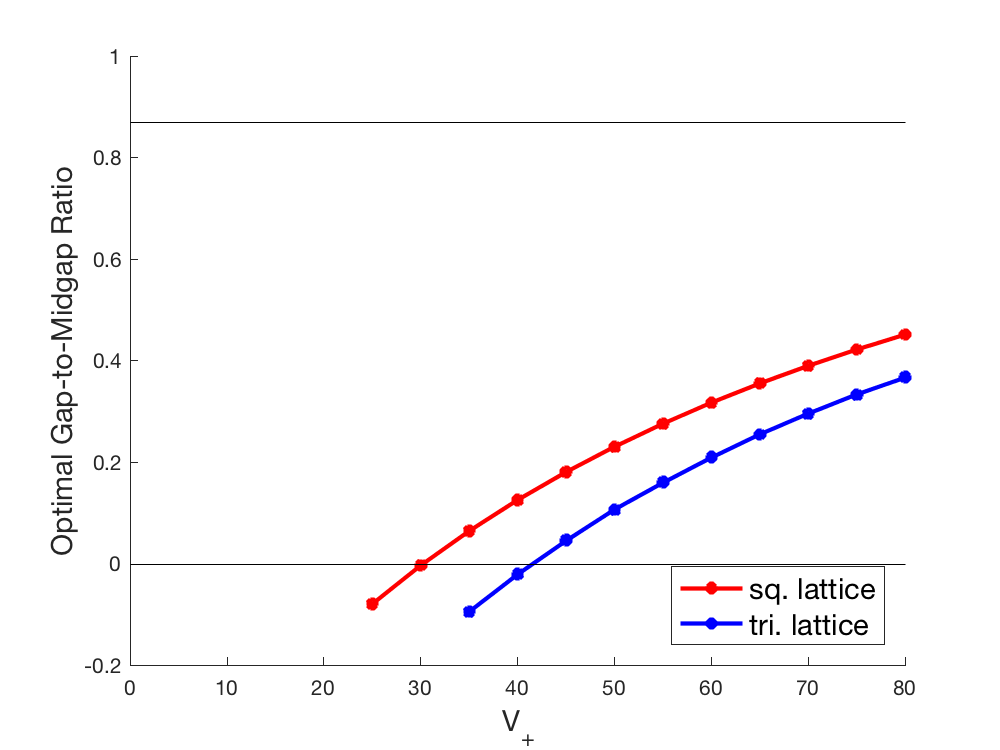}
\includegraphics[width=.32\textwidth]{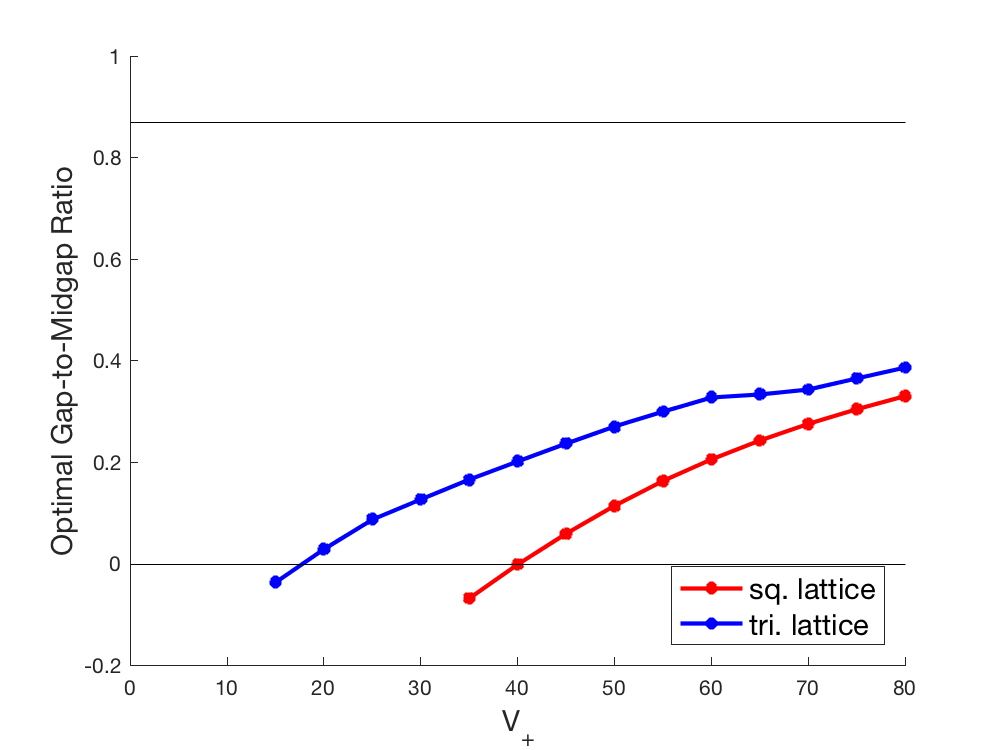}
\caption{A plot of  $G_{m,\Gamma,V_+}^*$ vs. $V_+$ for $m=1,2,3$ and $\Gamma$ the square lattice (red) and triangular lattice (blue). 
We observe that $G^*$ is increasing with $V_+$ and that for all values of $V_+$, the triangular lattice has a larger first and third gap than the square lattice, while the opposite is true for the second gap. }
\label{f:OptValvaryV+}
\end{center}
\end{figure}

\subsection{Optimization over lattices} \label{s:Opt2dLattices}
In the previous sections, we have fixed the lattice $\Gamma$ (either square or triangular) and studied properties of optimal potentials. In this section, we study how the optimal value, $G_{m, \Gamma, V_+}^{\star}$ and optimal potentials, $V^\star$ vary as we vary $\Gamma$ over equal-volume Bravais lattices $\Gamma$. Computing the extremal gaps for general $\Gamma$ is a more challenging problem since there are no rotational symmetries giving a small irreducible Brillouin zone. Using the fact that the potential is real, we discretize half of the Brillouin zone; see Section \ref{sec:IBZ}. 

A parameterization of equal-volume, two-dimensional lattices is given in Appendix \ref{sec:LattParam}. In particular, see Figure \ref{fig:FundDom}, where the set $U$ in Proposition~\ref{prop:LatticeParam} is illustrated. Using this parameterization of lattices, 
$\Gamma = \Gamma(a,b)$, in  Figure \ref{f:LatticeParam}, we plot $G_{m, \Gamma(a,b), V_+}^{\star}$ as $(a,b)$ varies over $U$  for  fixed $m=1,2$ and $V_+=100$. 

For $m=1$, we observe that the triangular lattice $(a,b) = \left( \frac{1}{2}, \frac{\sqrt{3}}{2} \right)$ is optimal. The optimal potential and corresponding dispersion surface along the boundary of Brillouin zone are shown in Figure \ref{f:2DtriA} (first row). 

For $m=2$, the optimal lattice has parameters $a=0$ and $b\approx \sqrt{3}$. In Figure \ref{f:Rectangular} we plot the optimal potential and corresponding dispersion surfaces over the entire Brillouin zone. The optimal potential has the symmetry of the triangular lattice, even though the primitive cell is rectangular. The lattice spacing in this triangular lattice is $3^{-1/4}$, which is smaller than the lattice spacing for the triangular lattice with unit area, equal to $\sqrt{2} \cdot3^{-1/4}$.

\begin{figure} 
\begin{center}
$m=1$ \hspace{7cm} $m=2$  \\
\includegraphics[width=.45\textwidth]{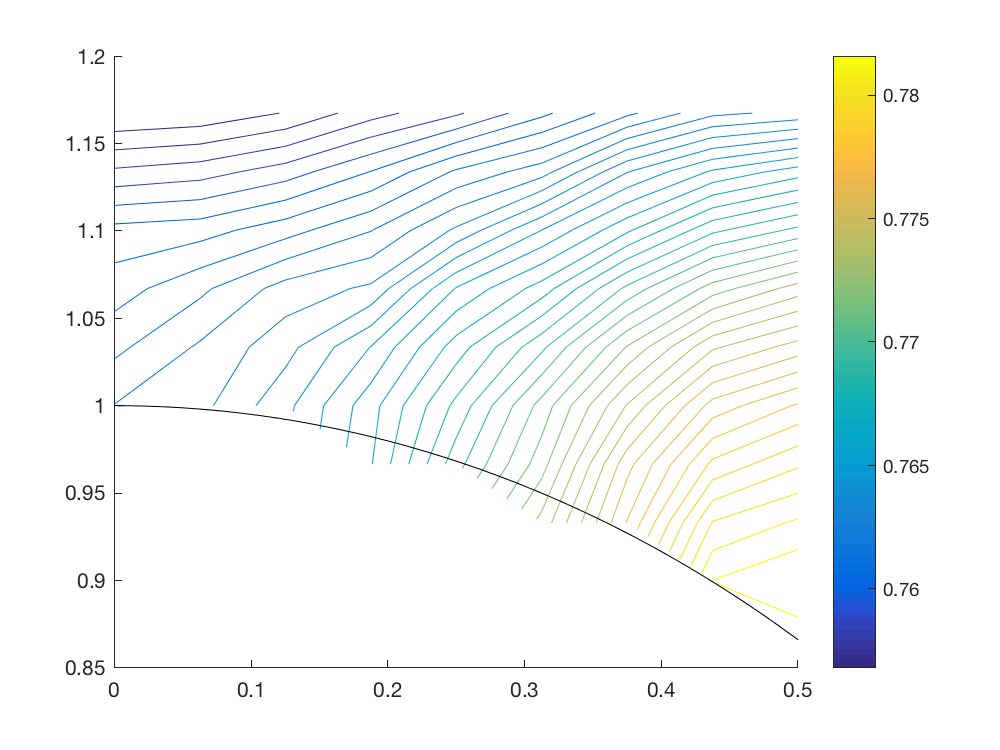}
\includegraphics[width=.45\textwidth]{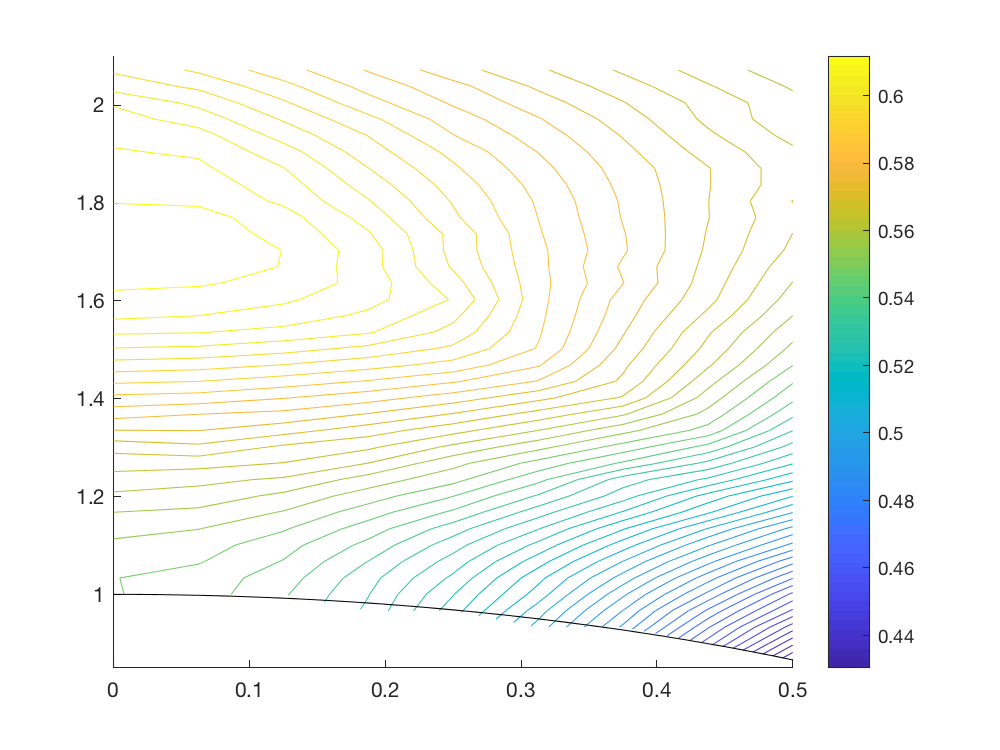}
\caption{For fixed $m=1,2$ and $V_+=100$, we plot $G_m^*$ for lattice $\Gamma = \Gamma(a,b)$  where $(a,b)$ varies over $U$. The set $U$, which gives a parameterization of two-dimensional lattices is described in  Appendix \ref{sec:LattParam} and, in particular,  illustrated in Figure \ref{fig:FundDom}.}
\label{f:LatticeParam}
\end{center}
\end{figure}

\begin{figure} 
\begin{center}
\includegraphics[height=.4\textwidth,trim={3cm 0 3cm 0},clip]{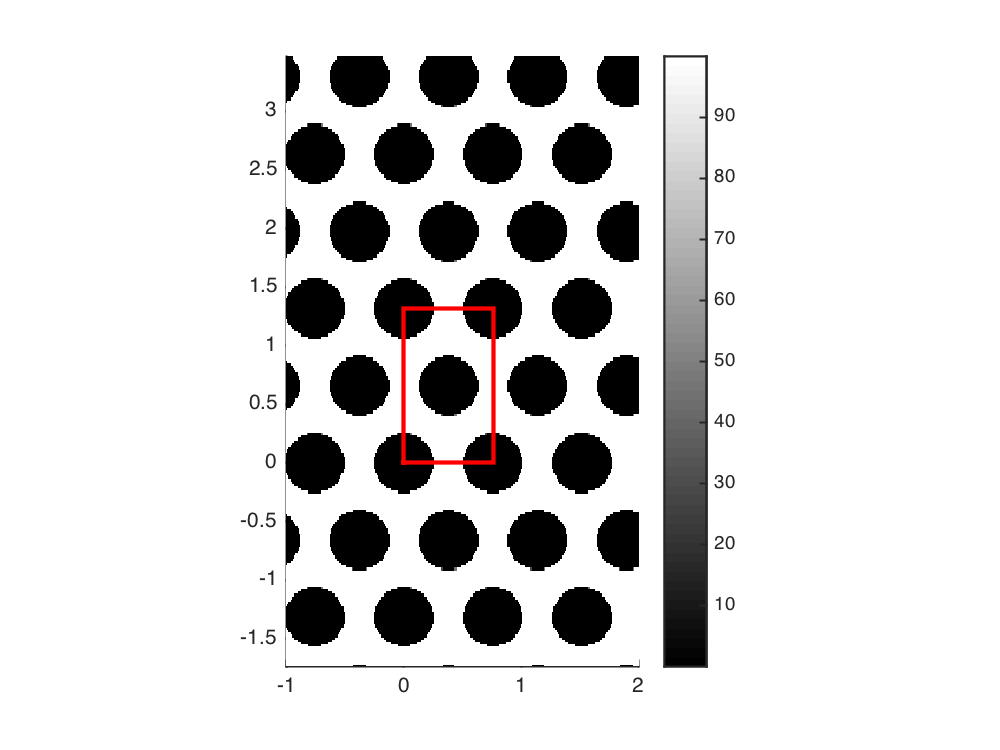}
\includegraphics[height=.4\textwidth]{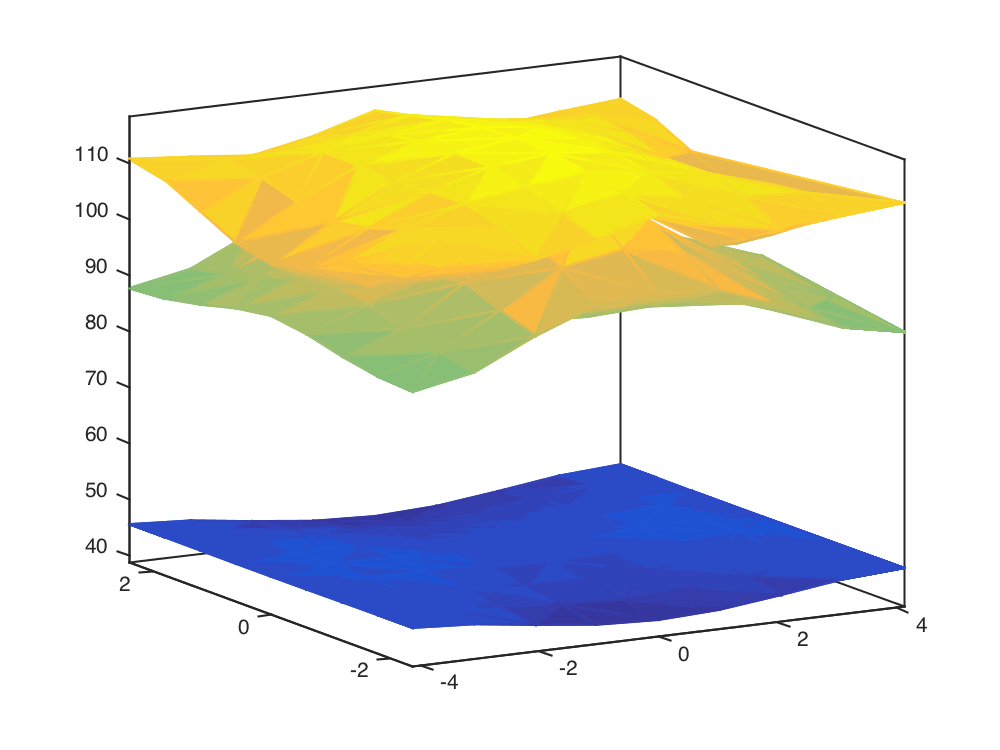}
\caption{The optimal potential {\bf (left)} and first four surfaces of the dispersion relation {\bf (right)} that maximize the $m=2$ gap for lattice parameters $(a,b) = (0,\sqrt{3})$; see also Figure \ref{f:LatticeParam}. The first two dispersion surfaces are very close to each other. }
\label{f:Rectangular}
\end{center}
\end{figure}

\section{Conclusion and discussion} \label{sec:disc}
For fixed $m$,  we have considered the problem of maximizing the gap-to-midgap ratio for the $m$-th spectral gap over the class of potentials which have fixed periodicity and are pointwise bounded above and below.  We show solutions this problem exist in Theorem \ref{thm:exist}.

 In Section  \ref{sec:1D}, we prove that the optimal potential in one dimension attains the pointwise bounds almost everywhere in the domain and is a step function attaining the imposed minimum and maximum values on exactly $m$ intervals. Optimal potentials are computed numerically using a rearrangement algorithm and found to be periodic. In Proposition \ref{p:1DhighContrast}, we prove that periodic potentials are optimal in the high contrast limit ($V_+ = \infty$). 
 
In Section \ref{s:Opt2d}, we develop an efficient rearrangement method for the two-dimensional problem based on a semi-definite program formulation (Algorithm \ref{alg:SDPrearrange}) and apply it to study properties of extremal potentials. In two-dimensional numerical simulations, we study the potential that maximizes the $m$-th bandgap, $G_{m, \Gamma, V_{+}}$, for $m=1,\ldots,8$ for a fixed $V_{+}$ on both square and triangular lattices. Although we are only able to prove in Proposition \ref{prop:Wbb} that the solution is weakly bang-bang, the computational results suggest that the solution is bang-bang. We also study the dependence of the optimal potentials on the parameter $V_+$. We observe from the computational results that the optimal potential as $V_+ \to \infty$ that the region where \revision{$V = 0$} consists of $m$ disks in the primitive cell. We prove, in Propositions \ref{prop:OptDisc} and Corollary \ref{prop:OptDiscMg2}, the infinite contrast asymptotic result ($V_+ = \infty$), that for $m\geq1$, subject to a geometric  assumption, that the optimal potential has \revision{$V = 0$} on exactly $m$ equal-size disks. Ultimately, we study the problem over equal-volume Bravais lattices. For $m=1$, the triangular lattice gives the maximal bandgap. For $m=2$, the maximal bandgap is achieved at $(a,b)=(\frac{1}{2},\frac{\sqrt{3}}{2})$. Even though the primitive cell is a rectangle but the optimal potential has the symmetry of the triangular lattice.

The numerical and asymptotic results suggest that for finite, but large values of $V_+$, the maximal potentials for $G_m$ are of the bang-bang form in 
 \eqref{e:BangBang}, where $\Omega_-$ is the union of $m$ connected sets, possibly disks. For the TM Helmholtz problem, it was conjectured by Sigmund and Hougaard that the optimal refractive index is given by a configuration of equal-sized disks with centers at the CVT \cite{sigmund2008geometric}.  This would be a reasonable conjecture for this problem as well. 

There are several open questions for this work. First, there are several questions for the rearrangement algorithms, namely,  can it be proven that the rearrangement algorithms are decreasing for non-stationary iterations? Can we estimate the number of iterations needed?  It is also desirable to establish under what conditions is the solution to the SDP in Algorithm  \ref{alg:SDPrearrange} is bang-bang. Can more information about the solution to the SDP in Algorithm \ref{alg:SDPrearrange} be used to speed up the implementation? 

As for properties of optimal potentials, in dimension $d\geq 2$, is the solution bang-bang?  (See Assumption \ref{assum:bangbang}.) While we have proven a partial result for an infinite contrast potential  (see Propositions \ref{prop:OptDisc} and \ref{prop:OptDiscMg1}), it is of interest to study the large but finite contrast case. One strategy for this is along the recent lines by R. Lipton and R. Viator  \cite{hempel2000spectral,Lipton:2016aa}, which we hope to pursue in future work.

\subsection*{Acknowledgements} Braxton Osting would like to thank the IMA, where he was visiting while most of this work was completed. 

\appendix
\section{Parameterization of Lattices} \label{sec:LattParam}
Let $B = [b_1, \ldots, b_n] \in \mathbb R^{n\times n}$ have linearly independent columns. The lattice generated by the basis $B$ is the set  of integer linear combinations of the columns of $B$, 
$$
\mathcal L(B) = \{ Bx \colon x \in \mathbb Z^n \}. 
$$
Let $B$ and $C$ be two lattice bases. We recall that $\mathcal L(B) = \mathcal L(C)$ if and only if there is a unimodular\footnote{A  matrix $A \in \mathbb Z^{n\times n}$ is \emph{unimodular} if $\mathrm{det} A  = \pm 1$.}  
matrix $U$ such that $B = CU$. Thus, there is a one-to-one correspondence between the unimodular $2\times 2$ matrices and the bases of a two-dimensional lattice. 

We say that two lattices are isometric if there is a rigid transformation that maps one to the other. The following proposition parameterizes the space of two-dimensional, unit-volume lattices modulo isometry. 
\begin{prop} \label{prop:LatticeParam}
Every two-dimensional lattice with volume one is isometric to a lattice parameterized by the basis 
$$ B_{a,b} = \begin{pmatrix} \frac{1}{\sqrt b} & \frac{a}{\sqrt b} \\ 0 & \sqrt b \end{pmatrix}, $$ 
where the parameters $a$ and $b$ are constrained to the set
$$
U := \left\{ (a,b) \in \mathbb R^2 \colon b>0, \ a \in [ 0,1/2  ], \ \text{and} \ a^2 + b^2 \geq 1 \right\}.
$$ 
\end{prop}

The set $U$ defined in Proposition~\ref{prop:LatticeParam} is illustrated  in Figure~\ref{fig:FundDom}.  

\begin{figure}[t]
\begin{center}
\begin{tikzpicture}[scale=.4,thick,>=stealth',dot/.style = {fill = black,circle,inner sep = 0pt,minimum size = 4pt}]

%fill region first
 \filldraw[fill=blue!6]  (0,10) arc(90:60:10cm)  -- (5,17) [dashed] -- (0,17) -- cycle;

% coordinate axis
\draw[->] (-3,0) -- (13,0) coordinate[label = {below:$a$}];
\draw[->] (0,-.1) coordinate[label={below:$0$}] -- (0,17) coordinate[label = {left:$b$}];

% lines
\draw (10,.1) -- (10,-.1) coordinate[label={below:$1$}]; % x tick
\draw (5,-.1)coordinate[label = {below:$0.5$}] -- (5,17) coordinate(top); % vertical line
\draw (10,0) arc (0:60:10) arc (60:90:10) node[midway]{\begin{tabular}{c} rhombic \\lattices\end{tabular}}{} arc(90:110:10) ; % arc
%\draw (10,0) arc (0:110:10); % arc
 
 % label regions 
 \draw (5,8.660254) node[dot,label={right: \begin{tabular}{c} triangular \\lattice\end{tabular}}](triLat){};
 \draw (0,10) -- (0,13) coordinate[label={center:\rotatebox{90}{\begin{tabular}{c}rectangular\\lattices\end{tabular}}}]{} --(0,17) ;
 \draw (2.5,12) node[label={\begin{tabular}{c} oblique \\lattices\end{tabular}}]{}; 
 \draw (0,10) node[dot,label={below left:\begin{tabular}{c} square \\lattice\end{tabular}}]{};

\end{tikzpicture}
\caption{The set $U$ in Proposition~\ref{prop:LatticeParam}. Parameters $(a,b)$ corresponding to  square, triangular, rectangular, rhombic, and oblique lattices are also indicated.}
\label{fig:FundDom}
\end{center}
\end{figure}
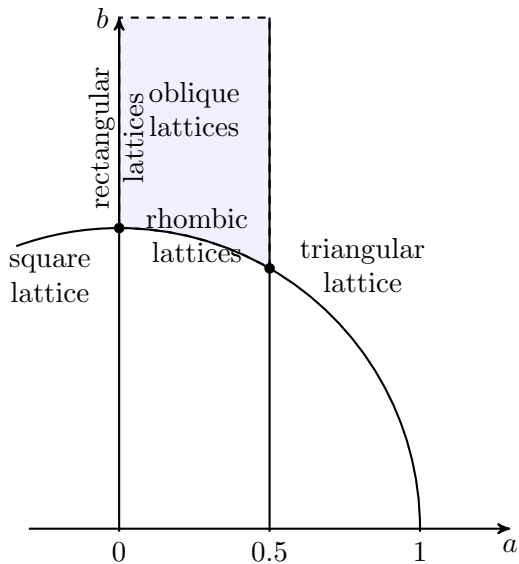

\begin{proof} Consider an arbitrary lattice with unit volume. 
We first choose the basis vectors so that the angle between them is acute.  After a suitable rotation and reflection, we can let the shorter basis vector (with length $\frac{1}{\sqrt b}$) be parallel with the $x$ axis and the longer basis  vector (with length $ \sqrt{ \frac{a^2}{b} + b} = \sqrt{\frac{1}{b} \left( a^2 + b^2 \right)} \geq \sqrt{ \frac{1}{b}}$) lie in the first quadrant (so $a\geq 0$). 
Multiplying on the right by a unimodular matrix, $ \begin{pmatrix} 1 & 1 \\ 0 & 1 \end{pmatrix}$, we compute 
$$
\begin{pmatrix} \frac{1}{\sqrt b} & \frac{a}{\sqrt b} \\ 0 & \sqrt b \end{pmatrix} 
 \begin{pmatrix} 1 & 1 \\ 0 & 1 \end{pmatrix} = 
\begin{pmatrix} \frac{1}{\sqrt b} & \frac{a+1}{\sqrt b} \\ 0 & \sqrt b \end{pmatrix} . 
$$
Since this is equivalent to taking $a\mapsto a+1$, it follows that we can identify the lattices associated to the points $(a,b)$ and $(a+1,b)$. 
Thus, we can  restrict the parameter $a$ to the interval  $\left[ 0, 1/2  \right]$.  
\end{proof}

%\clearpage
%\printbibliography
%\bibliographystyle{siamplain}
\bibliographystyle{abbrv}
\bibliography{SchroBandGaps.bib} 

\end{document}